\documentclass[12pt,english]{article}
\pdfoutput=1
\usepackage[T1]{fontenc}
\usepackage[latin9]{inputenc}
\usepackage{geometry}
\usepackage{array}
\usepackage{float}
\usepackage{mathtools}
\usepackage{multirow}
\usepackage{amsmath}
\usepackage{amsthm}
\usepackage{amssymb}

\makeatletter

\providecommand{\tabularnewline}{\\}

\numberwithin{equation}{section}
\theoremstyle{plain}
\newtheorem{thm}{\protect\theoremname}[section]
\theoremstyle{plain}
\newtheorem{lem}[thm]{\protect\lemmaname}
\theoremstyle{plain}
\newtheorem{cor}[thm]{\protect\corollaryname}
\theoremstyle{plain}
\newtheorem{prop}[thm]{\protect\propositionname}
\theoremstyle{definition}
\newtheorem{problem}[thm]{\protect\problemname}
\theoremstyle{plain}
\newtheorem{conjecture}[thm]{\protect\conjecturename}

\usepackage{bm}

\usepackage{tikz}
\usetikzlibrary{arrows}
\usepackage{hyperref}
\usepackage{needspace}

\usepackage{colortbl}
\usepackage{xparse}
\usepackage{caption}
\usepackage{microtype}

\usepackage{titlesec}
\titleformat{\section}
  {\normalfont\large\bfseries}{\thesection.}{.8ex}{} %changed from Large
\titleformat{\subsection}
  {\normalfont\bfseries}{\thesubsection.}{.8ex}{} %large removed
\usetikzlibrary{shapes}

\tikzstyle{pathdefault}=[draw, line width=1, solid, color=black]
\tikzstyle{nodedefault}=[circle, inner sep=1.3, fill=black]
\tikzstyle{nodered}=[circle, inner sep=1.1, fill=red]
\tikzstyle{nodeblue}=[circle, inner sep=1.1, fill=blue]
\tikzstyle{empty}=[]
\tikzstyle{nodeellipsis}=[circle, inner sep=0.5, fill=black]
\tikzstyle{pathcolor1}=[draw, line width=1.5, solid, color=black]
\tikzstyle{pathcolor2}=[draw, line width=1, solid, color=blue]
\tikzstyle{pathcolorlight}=[draw, line width=1, dotted, color=lightgray]

\tikzstyle{arbpathcolor0}=[line width=1, dashdotted, color=black]
\tikzstyle{arbpathcolor1}=[line width=1, densely dashed, color=red]
\tikzstyle{arbpathdefault}=[line width=1, densely dotted, color=blue]

\newcounter{id}
\newcommand{\drawlinedotswithstyle}[4]{
 \def\x{{#3}}
 \def\y{{#4}}
 \tikzstyle{thispathstyle}=[#1]
 \tikzstyle{thisnodestyle}=[#2]
 \setcounter{id}{-1} %start id at -1
 \foreach \j in {#3}{\stepcounter{id}} %id is one less than num of pts
 \foreach \i in {1,...,\the\value{id}}{  %loop through later indices
  \path[thispathstyle] (\x[\i],\y[\i]) --(\x[\i-1],\y[\i-1]); %draw edge
 }
 \foreach \i in {1,...,\the\value{id}}{  %loop through later indices
  \node[thisnodestyle] at (\x[\i],\y[\i]) {}; %draw node
 }
 \node[thisnodestyle] at (\x[0],\y[0]) {}; %draw first node outside of loop
}

\DeclareDocumentCommand{\drawlinedots}{ O{pathdefault} O{nodedefault} m m}{\drawlinedotswithstyle{#1}{#2}{#3}{#4}}
\DeclareDocumentCommand{\drawlinedotsred}{ O{pathcolor1} O{nodered} m m}{\drawlinedotswithstyle{#1}{#2}{#3}{#4}}
\DeclareDocumentCommand{\drawlinedotsblue}{ O{pathcolor2} O{nodeblue} m m}{\drawlinedotswithstyle{#1}{#2}{#3}{#4}}

\let\originalleft\left
\let\originalright\right
\renewcommand{\left}{\mathopen{}\mathclose\bgroup\originalleft}
\renewcommand{\right}{\aftergroup\egroup\originalright}

\newcommand{\leqnomode}{\tagsleft@true\let\veqno\@@leqno}
\newcommand{\reqnomode}{\tagsleft@false\let\veqno\@@eqno}
\reqnomode

\makeatother

\usepackage{babel}
\providecommand{\conjecturename}{Conjecture}
\providecommand{\corollaryname}{Corollary}
\providecommand{\lemmaname}{Lemma}
\providecommand{\problemname}{Problem}
\providecommand{\propositionname}{Proposition}
\providecommand{\theoremname}{Theorem}

\begin{document}
\global\long\def\Des{\operatorname{Des}}%

\global\long\def\pk{\operatorname{pk}}%

\global\long\def\des{\operatorname{des}}%

\global\long\def\lchd{\operatorname{lc}}%

\global\long\def\lbrch{\operatorname{lbrch}}%

\global\long\def\rbrch{\operatorname{rbrch}}%

\global\long\def\first{\operatorname{first}}%

\global\long\def\last{\operatorname{last}}%

\global\long\def\st{\operatorname{st}}%

\global\long\def\LRMin{\operatorname{LRMin}}%

\global\long\def\lrmax{\operatorname{lrmax}}%

\global\long\def\rlmin{\operatorname{rlmin}}%

\global\long\def\rlmax{\operatorname{rlmax}}%

\global\long\def\lrmin{\operatorname{lrmin}}%

\global\long\def\asc{\operatorname{asc}}%

\global\long\def\std{\operatorname{std}}%

\global\long\def\e{\operatorname{e}}%

\global\long\def\o{\operatorname{o}}%

\global\long\def\occ{\operatorname{occ}}%

\global\long\def\mi{\operatorname{mi}}%

\global\long\def\sn{\operatorname{sn}}%

\global\long\def\cn{\operatorname{cn}}%

\global\long\def\dn{\operatorname{dn}}%

\global\long\def\cpen{\operatorname{cpen}}%

\title{Restricted Jacobi permutations}
\author{Alyssa G.\ Henke, Kyle R.\ Hoffman, Derek H.\ Stephens,\\
Yongwei Yuan, and Yan Zhuang\\
Department of Mathematics and Computer Science\\
Davidson College\texttt{}~\\
\texttt{\small{}\{alhenke, kyhoffman, destephens, ivyuan, yazhuang\}@davidson.edu}}
\maketitle
\begin{abstract}
Jacobi permutations, introduced by Viennot in the context of Jacobi
elliptic functions, are counted by the Euler numbers $E_{n}$ appearing
in the series expansion $\sec x+\tan x=\sum_{n=0}^{\infty}E_{n}x^{n}/n!$.
We conduct a systematic study of pattern avoidance in Jacobi permutations,
achieving a complete enumeration of Jacobi permutations avoiding a
prescribed set of length 3 patterns. In the case of a single pattern
restriction, we obtain refined enumerations with respect to several
permutation statistics: the number of ascents (or descents), the number
of left-to-right minima, and the last letter. Bijections involving
certain subfamilies of binary trees and Dyck paths, as well as generating
function techniques, play important roles in our proofs.
\end{abstract}
\textbf{\small{}Keywords: }{\small{}Jacobi permutations, Euler numbers,
pattern avoidance, permutation statistics, alternating permutations,
Andr}\'{e}{\small{} permutations }{\let\thefootnote\relax}

\tableofcontents{}

\section{Introduction}

Let $S$ be a finite set of positive integers. A \textit{permutation}
$\pi=\pi_{1}\pi_{2}\cdots\pi_{n}$ of $S$\textemdash where $n=\left|S\right|$\textemdash is
a linear arrangement of the elements of $S$. For example, $\pi=3174$
is a permutation of $S=\{1,3,4,7\}$. We think of permutations as
words over the alphabet $S$ without repetition, so the $\pi_{k}$
are called \textit{letters} of $\pi$ and $n$ the \textit{length}
of $\pi$. Let $\left|\pi\right|$ denote the length of $\pi$, so
that $\left|\pi\right|=\left|S\right|$ whenever $\pi$ is a permutation
of $S$. If $S$ is the empty set, then the only permutation of $S$
is the empty word $\varepsilon$. The symmetric group of permutations
of $[n]\coloneqq\{1,2,\dots,n\}$ is denoted by $\mathfrak{S}_{n}$.

The \textit{Euler numbers} $E_{n}$, defined as the coefficients in
the series expansion
\[
\sec x+\tan x=\sum_{n=0}^{\infty}E_{n}\frac{x^{n}}{n!},
\]
have many combinatorial interpretations, including several interesting
families of permutations. We say that $\pi$ is \textit{up-down alternating}
(or simply \textit{up-down}) if $\pi_{1}<\pi_{2}>\pi_{3}<\pi_{4}>\cdots$.
Let $\mathfrak{U}_{n}$ be the set of up-down permutations in $\mathfrak{S}_{n}$.
It is a classical result of Andr\'{e}~\cite{Andre1881} that $\left|\mathfrak{U}_{n}\right|=E_{n}$
for all $n\geq0$. By symmetry, $E_{n}$ also counts \textit{down-up}
(\textit{alternating}) permutations $\pi\in\mathfrak{S}_{n}$, which
instead satisfy $\pi_{1}>\pi_{2}<\pi_{3}>\pi_{4}<\cdots$. Additionally,
the Euler numbers are known to count Andr\'{e} permutations \cite{Foata1973},
a variant of Andr\'{e} permutations called simsun permutations \cite{Sundaram1994},
as well as the subject of the present paper: Jacobi permutations \cite{Viennot1980}.

Given a permutation $\pi$ of $S$ and a letter $x$ of $\pi$, let
$\rho_{\pi}(x)$ be the maximal consecutive subword of $\pi$ consisting
of letters immediately to the right of $x$ which are all larger than
$x$. As an example, Table \ref{tb-rho} shows $\rho_{\pi}(x)$ for
each letter of $\pi=417926$.
\begin{table}[!h]
\begin{centering}
\renewcommand{\arraystretch}{1.1}%
\begin{tabular}{>{\centering}p{27bp}|>{\centering}p{27bp}|>{\centering}p{27bp}|>{\centering}p{27bp}|>{\centering}p{27bp}|>{\centering}p{27bp}|>{\centering}p{27bp}}
$x$ & $4$ & $1$ & $7$ & $9$ & $2$ & $6$\tabularnewline
\hline 
$\rho_{\pi}(x)$ & $\varepsilon$ & $7926$ & $9$ & $\varepsilon$ & $6$ & $\varepsilon$\tabularnewline
\end{tabular}
\par\end{centering}
\caption{\label{tb-rho}The subwords $\rho_{\pi}(x)$ for the permutation $\pi=417926$.}
\end{table}

\noindent We say that $\pi$ is \textit{Jacobi} if $\left|\rho_{\pi}(x)\right|$
is even for all $x\in S$. Observe that $\pi=417926$ is not Jacobi
because $\left|\rho_{\pi}(7)\right|=\left|\rho_{\pi}(2)\right|=1$,
but $\pi=416972$ is Jacobi.

Jacobi permutations can also be defined recursively \cite[Lemme 4]{Viennot1980}
in the following way. First, the empty permutation is Jacobi. If $S$
is nonempty, then let $y=\min S$ be the smallest element of $S$,
and let us decompose a permutation $\pi$ of $S$ as $\pi=\alpha y\beta$;
here, $\alpha$ is the subword of letters to the left of $y$ and
$\beta$ is the subword of letters to the right of $y$. Then, $\pi$
is Jacobi if both $\alpha$ and $\beta$ are Jacobi and $\left|\beta\right|$
is even. We will use these two definitions of Jacobi permutations
interchangeably.

Let $\mathfrak{J}_{n}$ denote the set of Jacobi permutations in $\mathfrak{S}_{n}$.
See Table \ref{tb-jac} for all permutations in $\mathfrak{J}_{n}$
up to $n=5$.
\begin{table}[!h]
\begin{centering}
\renewcommand{\arraystretch}{1.1}%
\begin{tabular}{|c|c|c|c|c|c|cc|c|ccccc|}
\multicolumn{1}{c}{$\mathfrak{J}_{0}$} & \multicolumn{1}{c}{} & \multicolumn{1}{c}{$\mathfrak{J}_{1}$} & \multicolumn{1}{c}{} & \multicolumn{1}{c}{$\mathfrak{J}_{2}$} & \multicolumn{1}{c}{} & \multicolumn{2}{c}{$\mathfrak{J}_{3}$} & \multicolumn{1}{c}{} &  &  & $\mathfrak{J}_{4}$ &  & \multicolumn{1}{c}{}\tabularnewline
\cline{1-1} \cline{3-3} \cline{5-5} \cline{7-8} \cline{8-8} \cline{10-14} \cline{11-14} \cline{12-14} \cline{13-14} \cline{14-14} 
$\varepsilon$ & \quad{} & $1$ & \quad{} & $21$ & \quad{} & $321$ & $132$ & \quad{} & $4321$ & $2431$ & $4132$ & $3142$ & $2143$\tabularnewline
\cline{1-1} \cline{3-3} \cline{5-5} \cline{7-8} \cline{8-8} \cline{10-14} \cline{11-14} \cline{12-14} \cline{13-14} \cline{14-14} 
\end{tabular}\bigskip{}
\par\end{centering}
\begin{centering}
\renewcommand{\arraystretch}{1.1}%
\begin{tabular}{|cccccccc|}
\multicolumn{8}{c}{$\mathfrak{J}_{5}$}\tabularnewline
\hline 
$54321$ & $35421$ & $52431$ & $42531$ & $32541$ & $54132$ & $53142$ & $43152$\tabularnewline
$52143$ & $42153$ & $32154$ & $15432$ & $13542$ & $15243$ & $14253$ & $13254$\tabularnewline
\hline 
\end{tabular}
\par\end{centering}
\caption{\label{tb-jac}Permutations in $\mathfrak{J}_{n}$ up to $n=5$.}
\end{table}

The \textit{standardization} $\std(\pi)$ of a permutation $\pi$
of length $n$ is the permutation in $\mathfrak{S}_{n}$ obtained
by replacing the smallest letter of $\pi$ by 1, the second smallest
by 2, and so on. For example, we have $\std(417926)=315624$. Note
that $\pi$ is Jacobi if and only if its standardization $\std(\pi)$
is Jacobi, a fact that we will tacitly use throughout this paper.

Jacobi permutations were introduced by Viennot \cite{Viennot1980}
in the context of Jacobi elliptic functions.\footnote{See also Viennot's survey \cite{Viennot1982}.}
Viennot studied a sequence of polynomials $J_{n}(k^{2})$ appearing
in the series expansions of the functions $\sn$, $\cn$, and $\dn$,
giving them a combinatorial interpretation as encoding the distribution
of a certain statistic over Jacobi permutations. He gave two proofs
that Jacobi permutations are counted by Euler numbers, including one
utilizing a bijection between Jacobi and alternating permutations.
However, unlike alternating, Andr\'{e}, and simsun permutations,
very little seems to have been done with Jacobi permutations since
their introduction. 

Pattern avoidance in permutations has received widespread attention
in recent decades, and a major theme within its study is the enumeration
of special families of pattern-avoiding permutations. There is a sizable
literature on pattern-avoiding alternating permutations\textemdash see,
e.g., \cite{Chen2014,Lewis2009,Lewis2011,Lewis2012,Mansour2003,Yan2013}\textemdash and
pattern avoidance has also been studied in simsun permutations~\cite{Deutsch2012}.
We pursue a variation of this theme by studying pattern avoidance
in Jacobi permutations.

\subsection{Pattern avoidance and permutation statistics}

Given permutations $\sigma$ and $\pi$, an \textit{occurrence} of
$\sigma$ in $\pi$ is a subword of $\pi$ whose standardization is
$\sigma$, and we say that $\pi$ \textit{avoids} $\sigma$ (as a
\textit{pattern})\textemdash and that $\pi$ is \textit{$\sigma$-avoiding}\textemdash if
$\pi$ does not contain an occurrence of $\sigma$. For example, $164$
is an occurrence of $\sigma=132$ in $\pi=315624$, but $\pi=315624$
avoids $\sigma=321$. 

Let $\mathfrak{S}_{n}(\sigma)$ denote the set of permutations in
$\mathfrak{S}_{n}$ which avoid $\sigma$. As seen above, we have
$315624\in\mathfrak{S}_{6}(321)$. More generally, given a set $\Pi$
of patterns, let $\mathfrak{S}_{n}(\Pi)$ denote the set of permutations
in $\mathfrak{S}_{n}$ which avoid every pattern in $\Pi$. It is
well known that for all $\sigma\in\mathfrak{S}_{3}$ and all $n\geq0$,
the number of permutations in $\mathfrak{S}_{n}(\sigma)$ is the $n$th
Catalan number. The systematic study of pattern-avoiding permutations
was initiated by Simion and Schmidt~\cite{Simion1985}, who enumerated
$\mathfrak{S}_{n}(\Pi)$ for all length 3 pattern sets $\Pi\subseteq\mathfrak{S}_{3}$.

Let $\mathfrak{J}_{n}(\Pi)$ denote the set of Jacobi permutations
in $\mathfrak{J}_{n}$ which avoid every pattern in $\Pi$, and let
$j_{n}(\Pi)\coloneqq\left|\mathfrak{J}_{n}(\Pi)\right|$. We achieve
a complete enumeration of the Jacobi avoidance classes $\mathfrak{J}_{n}(\Pi)$
where $\Pi\subseteq\mathfrak{S}_{3}$. A summary of these results
is provided in Table \ref{tb-pa}.\footnote{We exclude from Table \ref{tb-pa} pattern sets $\Pi$ for which $j_{n}(\Pi)\leq1$,
aside from the singleton $\Pi=\{132\}.$}

En passant, we find that all permutations $\pi$ in $\mathfrak{J}_{n}(213)$,
$\mathfrak{J}_{n}(231)$, and $\mathfrak{J}_{n}(312)$ are \textit{doubly
Jacobi}: both $\pi$ and their inverse $\pi^{-1}$ are Jacobi. This
leads us to study doubly Jacobi permutations in the other single-pattern
Jacobi avoidance classes. Notably, $123$-avoiding doubly Jacobi permutations
are counted by secondary structure numbers (see Theorem \ref{t-123doubly}).

Furthermore, we study the distributions of several permutation statistics
over Jacobi avoidance classes. For the following definitions, let
$\pi$ be a permutation of length $n$.
\begin{itemize}
\item We say that $k\in[n-1]$ is an \textit{ascent} of $\pi$ if $\pi_{k}<\pi_{k+1}$
and a \textit{descent} of $\pi$ if $\pi_{k}>\pi_{k+1}$. Let $\asc(\pi)$
denote the number of ascents of $\pi$, and $\des(\pi)$ the number
of descents of $\pi$.\footnote{Since $\asc(\pi)+\des(\pi)=n-1$, the value of $\asc$ determines
that of $\des$ and vice versa, but in some cases it will be easier
to work with one statistic over the other.}
\item We say that $\pi_{k}$ is a \textit{left-to-right minimum} of $\pi$
if $\pi_{k}<\pi_{i}$ for all $i<k$. Let $\lrmin(\pi)$ denote the
number of left-to-right minima of $\pi$.
\item Let $\last(\pi)\coloneqq\pi_{n}$, the \textit{last letter} of $\pi$.
\end{itemize}
For example, if $\pi=315624$, then the ascents of $\pi$ are $2$,
$3$, and $5$, the descents of $\pi$ are $1$ and $4$, the left-to-right
minima of $\pi$ are $3$ and $1$, and the last letter of $\pi$
is $4$. Thus, we have $\asc(\pi)=3$, $\des(\pi)=2$, $\lrmin(\pi)=2$,
and $\last(\pi)=4$.

\begin{table}[H]
\begin{centering}
\vspace{10bp}
\renewcommand{\arraystretch}{1.8}%
\begin{tabular}{|>{\raggedright}p{1.4in}|>{\raggedright}m{2.5in}|>{\raggedright}p{1in}|>{\raggedright}p{0.7in}|}
\hline 
$\Pi$ & $j_{n}(\Pi)$ & Reference & OEIS \cite{oeis}\tabularnewline
\hline 
$\{213\}$ & \multirow{3}{2.5in}[-5bp]{$\begin{cases}
{\displaystyle \frac{1}{2n+1}{3n \choose n},}\vphantom{\phantom{\frac{\frac{dy}{dx}}{\frac{dy}{\frac{dy}{dx}}}}} & \text{if }n\text{ is even,}\\
{\displaystyle \frac{1}{2n+1}{3n+1 \choose n+1},\vphantom{\phantom{\frac{\frac{\frac{dy}{dx}}{dx}}{\frac{dy}{dx}}}}} & \text{if }n\text{ is odd.}
\end{cases}$} & \multirow{2}{1in}{Theorem \ref{t-213-312}} & \multirow{3}{0.7in}{A047749}\tabularnewline
\cline{1-1} 
$\{312\}$ &  &  & \tabularnewline
\cline{1-1} \cline{3-3} 
$\{231\}$ &  & Theorem \ref{t-231} & \tabularnewline
\hline 
$\{123\}\phantom{\frac{\frac{\frac{dy}{ds}}{\frac{dy}{ds}}}{\frac{\frac{dy}{ds}}{\frac{\frac{dy}{ds}}{dy}}}}$ & ${\displaystyle \sum_{k=0}^{\left\lfloor (n-1)/2\right\rfloor }\frac{1}{2k+1}{n-k-1 \choose k}{n \choose 2k}}$ & Theorem \ref{t-123} & A101785\tabularnewline
\hline 
$\{132\}$ & $1$ & Corollary \ref{c-132} & A000012\tabularnewline
\hline 
$\{321\}$ & \multirow{1}{2.5in}{${\displaystyle C_{\left\lfloor n/2\right\rfloor }}$ (Catalan numbers)} & \multirow{1}{1in}{Corollary \ref{c-321}} & \multirow{1}{0.7in}{A000108}\tabularnewline
\hline 
$\{123,213\}$ & \multirow{1}{2.5in}{$f_{n-1}$ (Fibonacci numbers)} & Theorem \ref{t-123-213} & A000045\tabularnewline
\hline 
$\{123,231\}$ & \multirow{2}{2.5in}{${\displaystyle 1+\left\lfloor \frac{(n-1)^{2}}{4}\right\rfloor }$} & \multirow{2}{1in}{Theorem \ref{t-123-231or312}} & \multirow{2}{0.7in}{A033638}\tabularnewline
\cline{1-1} 
$\{123,312\}$ &  &  & \tabularnewline
\hline 
$\{213,231\}$ & \multirow{3}{2.5in}{$2^{\left\lfloor (n-1)/2\right\rfloor }$} & \multirow{2}{1in}{Theorem \ref{t-213-231or312}} & \multirow{3}{0.7in}{A016116}\tabularnewline
\cline{1-1} 
$\{213,312\}$ &  &  & \tabularnewline
\cline{1-1} \cline{3-3} 
$\{231,312\}$ &  & \multirow{1}{1in}{Theorem \ref{t-231-312}} & \tabularnewline
\hline 
$\{123,213,231\}$ & \multirow{3}{2.5in}{${\displaystyle \left\lceil \frac{n}{2}\right\rceil }$} & \multirow{2}{1in}{Theorem \ref{t-123-213-231or312}} & \multirow{3}{0.7in}{A065033}\tabularnewline
\cline{1-1} 
$\{123,213,312\}$ &  &  & \tabularnewline
\cline{1-1} \cline{3-3} 
$\{123,231,312\}$ &  & Theorem \ref{t-123-231-312} & \tabularnewline
\hline 
$\{213,231,312\}$ & \multirow{2}{2.5in}{$\begin{cases}
1, & \text{if }n\geq2\text{ is even,}\\
{\displaystyle 2,} & \text{if }n\geq3\text{ is odd.}
\end{cases}$} & \multirow{2}{1in}{Theorem \ref{t-213-231-312}} & \multirow{2}{0.7in}{A000034}\tabularnewline
\cline{1-1} 
$\{123,213,231,312\}$ &  &  & \tabularnewline
\hline 
\end{tabular}
\par\end{centering}
\caption{\label{tb-pa}Summary of results on the enumeration of Jacobi avoidance
classes.}
\end{table}

We determine the distributions of $\asc$/$\des$, $\lrmin$, and
$\last$ over $\mathfrak{J}_{n}(\sigma)$ for each $\sigma\in\mathfrak{S}_{3}$,
as well as over the full set $\mathfrak{J}_{n}$ of Jacobi permutations.
(The only exception is the distribution of $\last$ over $\mathfrak{J}_{n}(123)$,
which is left as an open problem.) See Table~\ref{tb-stat} for an
abbreviated summary of our results in this direction. Most of these
take the form of a formula for the numbers
\[
j_{n,k}^{\st}(\sigma)\coloneqq\left|\{\,\pi\in\mathfrak{J}_{n}(\sigma):\st(\pi)=k\,\}\right|\quad\text{and}\quad j_{n,k}^{\st}\coloneqq\left|\{\,\pi\in\mathfrak{J}_{n}:\st(\pi)=k\,\}\right|
\]
(where $\st$ is one of the above statistics) or a generating function
for these numbers.
\begin{table}
\begin{centering}
\renewcommand{\arraystretch}{1.4}%
\begin{tabular}{|>{\centering}p{30bp}|>{\centering}m{75bp}|>{\centering}m{75bp}|>{\centering}m{75bp}|}
\hline 
$\sigma$ & $\asc$/$\des$ & $\lrmin$ & $\last$\tabularnewline
\hline 
none & Theorem \ref{t-asc} & Theorem \ref{t-lrmin} & Theorem \ref{t-last} \tabularnewline
\hline 
$213$ & \multirow{2}{75bp}{\hphantom{a}Theorem \ref{t-213-312-des}} & \multirow{2}{75bp}{\hphantom{a}Theorem \ref{t-213-312-lrmin}} & Theorem \ref{t-213last}\tabularnewline
\cline{1-1} \cline{4-4} 
$312$ &  &  & Theorem \ref{t-312last}\tabularnewline
\hline 
$231$ & Theorem \ref{t-231-des} & Theorem \ref{t-231-lrmin} & Theorem \ref{t-231-last}\tabularnewline
\hline 
$123$ & Theorem \ref{t-123asc} & Theorem \ref{t-123lrmin} & open\tabularnewline
\hline 
$132$ & trivial & trivial & trivial\tabularnewline
\hline 
$321$ & Corollary \ref{c-321-asc} & Corollary \ref{c-321-lrmin} & Theorem \ref{t-321-last}\tabularnewline
\hline 
\end{tabular}
\par\end{centering}
\caption{\label{tb-stat}Summary of results on distributions of statistics
over $\mathfrak{J}_{n}(\sigma)$.}
\end{table}

The general techniques used here are also applicable to related statistics,
such as the number of right-to-left minima, but we shall restrict
our attention to the above statistics to keep the scope of this paper
manageable.

\subsection{Outline}

The structure of this paper is as follows:
\begin{itemize}
\item We begin in Section \ref{s-Jac} by establishing a collection of basic
facts about Jacobi permutations which will be used later in this paper,
and then we determine the distribution of $\asc$, $\lrmin$, and
$\last$ over $\mathfrak{J}_{n}$.
\item Permutations in $\mathfrak{S}_{n}(213)$, $\mathfrak{S}_{n}(312)$,
and $\mathfrak{S}_{n}(231)$ are in bijection with binary trees. In
Section \ref{s-trees}, we introduce Jacobi trees and dual Jacobi
trees, which are obtained by restricting these bijections to Jacobi
permutations.
\item We prove our results for the avoidance classes $\mathfrak{J}_{n}(213)$
and $\mathfrak{J}_{n}(312)$ in Section \ref{s-213-312}, while we
study $\mathfrak{J}_{n}(231)$ in Section \ref{s-231}. Jacobi and
dual Jacobi trees are used throughout.
\item Section \ref{s-123} focuses on the avoidance class $\mathfrak{J}_{n}(123)$.
Our main tool here is a variant of a bijection due to Krattenthaler
\cite{Krattenthaler2001} from $123$-avoiding permutations to Dyck
paths.
\item The avoidance classes $\mathfrak{J}_{n}(132)$ and $\mathfrak{J}_{n}(321)$
are studied in Section \ref{s-132-321}, although these turn out to
be less interesting: there is only one permutation in $\mathfrak{J}_{n}(132)$
for each $n\geq0$, while the permutations in $\mathfrak{J}_{n}(321)$
are $321$-avoiding down-up permutations when $n$ is even, and are
in bijection with $321$-avoiding down-up permutations when $n$ is
odd.
\item In Section \ref{s-mult}, we enumerate all $\mathfrak{J}_{n}(\Pi)$
where $\Pi\subseteq\mathfrak{S}_{3}$ contains at least two patterns.
\item We conclude in Section \ref{s-conclusion} by presenting open problems
and conjectures for future study.
\end{itemize}
We give a few remarks on notation and terminology before continuing.
Throughout this paper, $S$ represents a finite set of positive integers.
We will write $\mathfrak{S}_{S}$ for the set of permutations of $S$,
and notations such as $\mathfrak{J}_{S}$, $\mathfrak{U}_{S}$, and
$\mathfrak{S}_{S}(\Pi)$ are defined in the analogous way. We occasionally
use the term ``permutation'' to refer only to standardized permutations\textemdash i.e.,
permutations in $\mathfrak{S}_{n}$\textemdash but at other times
more generally to permutations of a set $S$; which one we mean should
always be clear from context.

When decomposing permutations, we use lowercase Greek symbols for
subwords and lowercase Latin symbols for letters. For example, if
we write $\pi=\alpha x\tau y\beta$, then $x$ and $y$ are letters
of $\pi$, while $\alpha$ is the subword of all letters appearing
before $x$, $\tau$ is the subword of all letters appearing between
$x$ and $y$, and $\beta$ is the subword of all letters appearing
after $y$.

Finally, given a specific pattern set $\Pi$, we will omit the curly
braces enclosing the elements of $\Pi$ when writing out $\mathfrak{J}_{n}(\Pi)$
and similar notations. For example, we write $\mathfrak{J}_{n}(123,213)$
as opposed to $\mathfrak{J}_{n}(\{123,213\})$.

\section{\label{s-Jac}Jacobi permutations without restrictions}

Before imposing any pattern avoidance restrictions, let us prove some
basic facts about Jacobi permutations and study the distributions
of the statistics $\asc$, $\lrmin$, and $\last$ over $\mathfrak{J}_{n}$.

\subsection{Basic facts about Jacobi permutations}
\begin{lem}
\label{l-Jacobi}Let $\pi=\pi_{1}\pi_{2}\cdots\pi_{n}\in\mathfrak{J}_{S}$.
\begin{enumerate}
\item [\normalfont{(a)}]Any suffix of $\pi$ is Jacobi. That is, $\pi_{k}\pi_{k+1}\cdots\pi_{n}$
is Jacobi for each $k\in[n]$.
\item [\normalfont{(b)}]If $y$ is larger than every element of $S$, then
$y\pi$ is Jacobi.
\item [\normalfont{(c)}]If $\pi=y_{1}\tau^{(1)}y_{2}\tau^{(2)}\cdots y_{m}\tau^{(m)}$
where $y_{1},y_{2},\dots,y_{m}$ are the left-to-right minima of $\pi$,
then $\tau^{(1)},\tau^{(2)},\dots,\tau^{(m)}$ are Jacobi permutations
of even length.
\item [\normalfont{(d)}]If $n\geq2$, then $\pi_{n-1}>\pi_{n}$.
\end{enumerate}
\end{lem}

By iterating Lemma \ref{l-Jacobi} (b), we see that the decreasing
permutation $n\cdots21$ is Jacobi for every $n\geq1$. 
\begin{proof}
Let $\beta=\pi_{k}\pi_{k+1}\cdots\pi_{n}$. We have $\rho_{\beta}(\pi_{i})=\rho_{\pi}(\pi_{i})$
for each $k\leq i\leq n$, and since $\pi$ is Jacobi, all of these
$\rho_{\pi}(\pi_{i})$ have even length. Therefore, $\beta$ is Jacobi,
which proves (a).

To prove (b), let $\tau=y\pi$. Then $\rho_{\tau}(y)=\varepsilon$
because $y$ is larger than all letters of $\pi$, whereas $\rho_{\tau}(\pi_{i})=\rho_{\pi}(\pi_{i})$
for each $i\in[n]$. Again, all of these $\rho_{\pi}(\pi_{i})$ have
even length, so $\tau$ is Jacobi.

For (c), fix $k\in[m]$. Since $y_{1},y_{2},\dots,y_{m}$ are the
left-to-right minima of $\pi$, we have that $\rho_{\pi}(y_{k})=\tau^{(k)}$
and $\rho_{\pi}(x)=\rho_{\tau^{(k)}}(x)$ for each letter $x$ of
$\tau^{(k)}$. Because $\pi$ is Jacobi, the former implies that $\tau^{(k)}$
is of even length, while the latter implies that $\tau^{(k)}$ is
Jacobi.

Finally, if $\pi_{n-1}<\pi_{n}$, then we would have $\rho_{\pi}(\pi_{n-1})=\pi_{n}$
which is of length 1; this contradicts $\pi$ being Jacobi. Thus,
(d) is proven.
\end{proof}
\begin{lem}
\label{l-lrminJacobi}Let $\pi=\alpha y\beta$ where $y$ is a left-to-right
minimum of $\pi$. If $\alpha$ and $y\beta$ are Jacobi, then $\pi$
is Jacobi.
\end{lem}

\begin{proof}
Let $x$ be an arbitrary letter of $\pi$. If $x$ is in $\alpha$,
then $\rho_{\pi}(x)=\rho_{\alpha}(x)$ because $y$ is a left-to-right
minimum of $\pi$. Otherwise, $\rho_{\pi}(x)=\rho_{y\beta}(x)$. Since
$\alpha$ and $y\beta$ are Jacobi, $\rho_{\alpha}(x)$ has even length
for every letter $x$ in $\alpha$, and $\rho_{y\beta}(x)$ has even
length for every letter $x$ in $y\beta$. Hence, $\rho_{\pi}(x)$
has even length for every letter $x$ of $\pi$.
\end{proof}
Mirroring the definition of left-to-right minimum, $\pi_{k}$ is called
a \textit{right-to-left minimum} of $\pi$ if $\pi_{k}<\pi_{i}$ for
all $i>k$.
\begin{lem}
\label{l-rlminJacobi}Let $\pi=\alpha y\beta$ where $y$ is a right-to-left
minimum of $\pi$. If $\alpha y$ and $\beta$ are Jacobi and $\beta$
is of even length, then $\pi$ is Jacobi.
\end{lem}

\begin{proof}
Since $\alpha y$ and $\beta$ are Jacobi, $\rho_{\alpha y}(x)$ has
even length for every letter $x$ in $\alpha y$, and similarly with
$\rho_{\beta}(x)$. Now, let $x$ be an arbitrary letter of $\pi$.
Consider the following cases:
\begin{itemize}
\item If $x$ is in $\beta$, then $\rho_{\pi}(x)=\rho_{\beta}(x)$, which
has even length.
\item If $x=y$, then because $y$ is a right-to-left minimum of $\pi$,
we have $\rho_{\pi}(x)=\beta$, which is of even length.
\item If $\rho_{\alpha y}(x)$ contains $y$, then $\rho_{\pi}(x)=\rho_{\alpha y}(x)\beta$
because $y$ is a right-to-left minimum of $\pi$. Since $\rho_{\alpha y}(x)$
and $\beta$ are both of even length, the same is true of $\rho_{\pi}(x)$.
\item Otherwise, if $x$ is a letter in $\alpha$ and $\rho_{\alpha y}(x)$
contains $y$, then $\rho_{\pi}(x)=\rho_{\alpha y}(x)$, which has
even length.
\end{itemize}
Therefore, $\pi$ is Jacobi.
\end{proof}

\subsection{\label{ss-ascents}Ascents}

We now study the distribution of several statistics over $\mathfrak{J}_{n}$,
beginning with the number of ascents. Let
\[
J^{\asc}(t,x)\coloneqq\sum_{n=0}^{\infty}\sum_{\pi\in\mathfrak{J}_{n}}t^{\asc(\pi)}\frac{x^{n}}{n!}=\sum_{n=0}^{\infty}\sum_{k=0}^{n}j_{n,k}^{\asc}t^{k}\frac{x^{n}}{n!}.
\]
be the bivariate generating function for Jacobi permutations with
respect to ascents. See Table \ref{tb-asc} for the coefficients of
this generating function up to $n=9$.
\begin{table}
\begin{centering}
\renewcommand{\arraystretch}{1.1}%
\begin{tabular}{|>{\centering}p{23bp}|>{\centering}p{23bp}|>{\centering}p{23bp}|>{\centering}p{23bp}|>{\centering}p{23bp}|>{\centering}p{23bp}|}
\hline 
$n\backslash k$ & $0$ & $1$ & $2$ & $3$ & $4$\tabularnewline
\hline 
$0$ & $1$ &  &  &  & \tabularnewline
\hline 
$1$ & $1$ &  &  &  & \tabularnewline
\hline 
$2$ & $1$ &  &  &  & \tabularnewline
\hline 
$3$ & $1$ & $1$ &  &  & \tabularnewline
\hline 
$4$ & $1$ & $4$ &  &  & \tabularnewline
\hline 
$5$ & $1$ & $11$ & $4$ &  & \tabularnewline
\hline 
$6$ & $1$ & $26$ & $34$ &  & \tabularnewline
\hline 
$7$ & $1$ & $57$ & $180$ & $34$ & \tabularnewline
\hline 
$8$ & $1$ & $120$ & $768$ & $496$ & \tabularnewline
\hline 
$9$ & $1$ & $247$ & $2904$ & $4288$ & $496$\tabularnewline
\hline 
\end{tabular}
\par\end{centering}
\caption{\label{tb-asc}The numbers $j_{n,k}^{\protect\asc}$ up to $n=9$.}
\end{table}

\begin{thm}
\label{t-asc}We have ${\displaystyle J^{\asc}(t,x)=1+\frac{2(e^{rx}-1)}{1+r+(r-1)e^{rx}}}$,
where $r=\sqrt{1-2t}$.
\end{thm}

\begin{proof}
Let
\[
J_{\e}^{\asc}=J_{\e}^{\asc}(t,x)\coloneqq\sum_{n=0}^{\infty}\sum_{\pi\in\mathfrak{J}_{2n}}t^{\asc(\pi)}\frac{x^{n}}{n!}\quad\text{and}\quad J_{\o}^{\asc}=J_{\o}^{\asc}(t,x)\coloneqq\sum_{n=0}^{\infty}\sum_{\pi\in\mathfrak{J}_{2n+1}}t^{\asc(\pi)}\frac{x^{n}}{n!}.
\]
We claim that $J_{\e}^{\asc}$ and $J_{\o}^{\asc}$ satisfy the following
system of differential equations:
\begin{align}
\frac{\partial J_{\e}^{\asc}}{\partial x} & =J_{\o}^{\asc}+tJ_{\o}^{\asc}(J_{\e}^{\asc}-1)\label{e-asc1}\\
\frac{\partial J_{\o}^{\asc}}{\partial x} & =J_{\e}^{\asc}+tJ_{\e}^{\asc}(J_{\e}^{\asc}-1).\label{e-asc2}
\end{align}
To see why these equations hold, first note that every nonempty Jacobi
permutation of even length can be written as $\pi=\alpha1\beta$ where
$\alpha$ and $\beta$ are Jacobi, $\alpha$ is of odd length, and
$\beta$ is of even length. When $\beta$ is empty, we have $\asc(\pi)=\asc(\alpha)$,
so this case contributes the term $J_{\o}^{\asc}$ to (\ref{e-asc1}).
When $\beta$ is nonempty, we have $\asc(\pi)=\asc(\alpha)+\asc(\beta)+1$,
which results in the contribution $tJ_{\o}^{\asc}(J_{\e}^{\asc}-1)$.
Thus we have (\ref{e-asc1}), and (\ref{e-asc2}) is obtained using
the same reasoning except that both $\alpha$ and $\beta$ are of
even length.

Solving the above system with the initial conditions $J_{\e}^{\asc}(t,0)=1$
and $J_{\o}^{\asc}(t,0)=0$, adding the solutions, and simplifying
yields the desired expression for $J^{\asc}(t,x)$.
\end{proof}
The generating function $J^{\asc}(t,x)$ has appeared in an equivalent
form in the study of Andr\'{e} permutations and Andr\'{e} polynomials.
Following Foata and Han \cite{Foata[20142016]}, \textit{Andr\'{e}
permutations} (\textit{of the first kind}) are defined recursively
as follows. First, the empty permutation is Andr\'{e}. If $\pi\in\mathfrak{S}_{S}$
where $S$ is nonempty, then write $\pi=\alpha y\beta$ where $y=\min S$;
then $\pi$ is Andr\'{e} if both $\alpha$ and $\beta$ are Andr\'{e},
and if the largest letter in the concatenation $\alpha\beta$ is a
letter of $\beta$. Several equivalent definitions of Andr\'{e} permutation
can be found in the literature, e.g., in \cite[Definitions 2.1 and 2.2]{Foata[20142016]}.
Let $\mathfrak{A}_{n}$ be the set of Andr\'{e} permutations in $\mathfrak{S}_{n}$.

Foata and Sch\"utzenberger \cite{Foata1973} introduced a family
$D_{n}(s,t)$ of polynomials\textemdash called \textit{Andr\'{e}
polynomials}\textemdash whose $s=1$ evaluation has exponential generating
function
\begin{equation}
\sum_{n=0}^{\infty}D_{n}(1,t)\frac{x^{n}}{n!}=\frac{r(1+we^{rx})}{1-we^{rx}}\label{e-Degf}
\end{equation}
where $r=\sqrt{1-2t}$ and $w=(1-r)/(1+r)$. In particular, they showed
that these polynomials count Andr\'{e} permutations by descents:
\begin{equation}
D_{n}(1,t)=\sum_{\pi\in\mathfrak{A}_{n}}t^{\des(\pi)+1}\label{e-Ddes}
\end{equation}
for all $n\geq1$. In fact, because
\[
\frac{1}{t}\left(\frac{r(1+we^{rx})}{1-we^{rx}}-1\right)=\frac{2(e^{rx}-1)}{1+r+(r-1)e^{rx}},
\]
it follows from Theorem \ref{t-asc}, Equation (\ref{e-Degf}), and
Equation (\ref{e-Ddes}) that $\asc$ has the same distribution over
$\mathfrak{J}_{n}$ as $\des$ does over $\mathfrak{A}_{n}$. We state
this as a corollary below.
\begin{cor}
\label{c-jacandreasc}For all $n\geq0$, the distribution of $\asc$
over $\mathfrak{J}_{n}$ is equal to the distribution of $\des$ over
$\mathfrak{A}_{n}$.
\end{cor}

\subsection{Left-to-right minima}

Our next statistic of interest is the number of left-to-right minima.
Let 
\[
J^{\lrmin}(t,x)\coloneqq\sum_{n=0}^{\infty}\sum_{\pi\in\mathfrak{J}_{n}}t^{\lrmin(\pi)}\frac{x^{n}}{n!}=\sum_{n=0}^{\infty}\sum_{k=0}^{n}j_{n,k}^{\lrmin}t^{k}\frac{x^{n}}{n!};
\]
see Table \ref{tb-lrmin} for the first several coefficients. This
generating function has a remarkably elegant expression which we state
below.
\begin{table}
\begin{centering}
\renewcommand{\arraystretch}{1.1}%
\begin{tabular}{|>{\centering}p{23bp}|>{\centering}p{23bp}|>{\centering}p{23bp}|>{\centering}p{23bp}|>{\centering}p{23bp}|>{\centering}p{23bp}|>{\centering}p{23bp}|>{\centering}p{23bp}|>{\centering}p{23bp}|>{\centering}p{23bp}|>{\centering}p{23bp}|}
\hline 
$n\backslash k$ & $0$ & $1$ & $2$ & $3$ & $4$ & $5$ & $6$ & $7$ & $8$ & $9$\tabularnewline
\hline 
$0$ & $1$ &  &  &  &  &  &  &  &  & \tabularnewline
\hline 
$1$ & $0$ & $1$ &  &  &  &  &  &  &  & \tabularnewline
\hline 
$2$ & $0$ & $0$ & $1$ &  &  &  &  &  &  & \tabularnewline
\hline 
$3$ & $0$ & $1$ & $0$ & $1$ &  &  &  &  &  & \tabularnewline
\hline 
$4$ & $0$ & $0$ & $4$ & $0$ & $1$ &  &  &  &  & \tabularnewline
\hline 
$5$ & $0$ & $5$ & $0$ & $10$ & $0$ & $1$ &  &  &  & \tabularnewline
\hline 
$6$ & $0$ & $0$ & $40$ & $0$ & $20$ & $0$ & $1$ &  &  & \tabularnewline
\hline 
$7$ & $0$ & $61$ & $0$ & $175$ & $0$ & $35$ & $0$ & $1$ &  & \tabularnewline
\hline 
$8$ & $0$ & $0$ & $768$ & $0$ & $560$ & $0$ & $56$ & $0$ & $1$ & \tabularnewline
\hline 
$9$ & $0$ & $1385$ & $0$ & $4996$ & $0$ & $1470$ & $0$ & $84$ & $0$ & $1$\tabularnewline
\hline 
\end{tabular}
\par\end{centering}
\caption{\label{tb-lrmin}The numbers $j_{n,k}^{\protect\lrmin}$ up to $n=9$.}
\end{table}

\begin{thm}
\label{t-lrmin}We have $J^{\lrmin}(t,x)=(\sec x+\tan x)^{t}$.
\end{thm}

While this theorem can be proven by solving a system of differential
equations as we did for ascents, we shall instead take a more direct
approach utilizing the exponential formula \cite[Section 5.1]{Stanley2024}.
\begin{proof}
Recall from Lemma \ref{l-Jacobi} (c) that every $\pi\in\mathfrak{J}_{n}$
can be written as
\[
\pi=y_{1}\tau^{(1)}y_{2}\tau^{(2)}\cdots y_{m}\tau^{(m)}
\]
where $y_{1},y_{2},\dots,y_{m}$ are the left-to-right minima of $\pi$
and $\tau^{(1)},\tau^{(2)},\dots,\tau^{(m)}$ are Jacobi permutations
of even length. Conversely, we can uniquely obtain all permutations
$\pi\in\mathfrak{J}_{n}$ with exactly $m$ left-to-right minima in
the following way: 
\begin{enumerate}
\item Partition $[n]$ into subsets $B_{1},B_{2},\dots,B_{m}$ of odd size,
where the subsets are listed in decreasing order of their smallest
element.

For example, taking $n=9$ and $m=3$, we can take $B_{1}=\{8\}$,
$B_{2}=\{3,4,6,7,9\}$, and $B_{3}=\{1,2,5\}$.
\item For each $k\in[m]$, let $y_{k}$ be the smallest element of $B_{k}$,
and choose a Jacobi permutation $\tau^{(k)}$ on the remaining elements
of $B_{k}$. 

Continuing the example, we have $y_{1}=8$, $y_{2}=3$, and $y_{3}=1$,
and take $\tau^{(1)}=\varepsilon$, $\tau^{(2)}=6974$, and $\tau^{(3)}=52$.
\item Take $\pi=y_{1}\tau^{(1)}y_{2}\tau^{(2)}\cdots y_{m}\tau^{(m)}$. 

Continuing the example, we have $\pi=836974152$.
\end{enumerate}
By construction, the left-to-right minima of $\pi$ are $y_{1},y_{2},\dots,y_{m}$.
Moreover, for all $k\in[m]$, we have $\rho_{\pi}(y_{k})=\tau^{(k)}$
and $\rho_{\pi}(x)=\rho_{\tau^{(k)}}(x)$ for each letter $x$ of
$\tau^{(k)}$. Since the $\tau^{(k)}$ and $\rho_{\tau^{(k)}}(x)$
are all of even length, $\pi$ is indeed Jacobi. This gives us a decomposition
of permutations in $\mathfrak{J}_{n}$ into ``irreducible'' factors.

Because $\sec x$ is the exponential generating function for even-length
Jacobi permutations, it follows that
\[
\int_{0}^{x}\sec u\,du=\log(\sec x+\tan x)
\]
is the exponential generating function for the irreducibles, so we
have
\[
J^{\lrmin}(t,x)=\exp(t\log(\sec x+\tan x))=(\sec x+\tan x)^{t}
\]
by the exponential formula.
\end{proof}

\subsection{Last letter}

\begin{table}
\begin{centering}
\renewcommand{\arraystretch}{1.1}%
\begin{tabular}{|>{\centering}p{24bp}|>{\centering}p{24bp}|>{\centering}p{24bp}|>{\centering}p{24bp}|>{\centering}p{24bp}|>{\centering}p{24bp}|>{\centering}p{24bp}|>{\centering}p{24bp}|>{\centering}p{24bp}|}
\hline 
$n\backslash k$ & $1$ & $2$ & $3$ & $4$ & $5$ & $6$ & $7$ & $8$\tabularnewline
\hline 
$1$ & $1$ &  &  &  &  &  &  & \tabularnewline
\hline 
$2$ & $1$ &  &  &  &  &  &  & \tabularnewline
\hline 
$3$ & $1$ & $1$ &  &  &  &  &  & \tabularnewline
\hline 
$4$ & $2$ & $2$ & $1$ &  &  &  &  & \tabularnewline
\hline 
$5$ & $5$ & $5$ & $4$ & $2$ &  &  &  & \tabularnewline
\hline 
$6$ & $16$ & $16$ & $14$ & $10$ & $5$ &  &  & \tabularnewline
\hline 
$7$ & $61$ & $61$ & $56$ & $46$ & $32$ & $16$ &  & \tabularnewline
\hline 
$8$ & $272$ & $272$ & $256$ & $224$ & $178$ & $122$ & $61$ & \tabularnewline
\hline 
$9$ & $1385$ & $1385$ & $1324$ & $1202$ & $1024$ & $800$ & $544$ & $272$\tabularnewline
\hline 
\end{tabular}
\par\end{centering}
\caption{\label{tb-entringer}The Entringer numbers $E_{n,k}$ up to $n=9$.}
\end{table}

The \textit{Entringer number} $E_{n,k}$ \cite{Entringer1966} is
defined to be the number of up-down alternating permutations $\pi\in\mathfrak{U}_{n}$
with $\pi_{1}=k$. In other words, by introducing the statistic $\first(\pi)\coloneqq\pi_{1}$,
the Entringer numbers give the distribution of $\first$ over $\mathfrak{U}_{n}$.
It can be shown bijectively that these numbers satisfy the recurrence
\[
E_{n,k}=E_{n-1,n-k}+E_{n,k+1}
\]
for all $n\geq2$ and $k\geq1$.

Entringer numbers capture the distribution of various other statistics
over objects counted by Euler numbers. For example, $E_{n,k}$ is
also the number of Andr\'{e} permutations $\pi\in\mathfrak{A}_{n}$
with $\pi_{1}=k$, as well as the number of $\pi\in\mathfrak{A}_{n}$
with $\pi_{n-1}=n-k$ \cite[Theorem 1.1]{Foata[20142016]}. Our next
result shows that $\last$ has the same distribution over $\mathfrak{J}_{n}$.
\begin{thm}
\label{t-last}For all $n,k\geq1$, we have $j_{n,k}^{\last}=E_{n,k}$.
Consequently, the distribution of $\last$ over $\mathfrak{J}_{n}$
is equal to the distribution of $\first$ over $\mathfrak{U}_{n}$
or $\mathfrak{A}_{n}$.
\end{thm}

For the proof of this theorem, we will need the notion of the \textit{complement}
$\pi^{c}$ of a permutation $\pi$, which is obtained by (simultaneously)
replacing the $i$th smallest letter of $\pi$ with the $i$th largest
letter of $\pi$ for all $i$. For example, given $\pi=319754$, we
have $\pi^{c}=791345$. Note that $\pi$ is up-down if and only if
$\pi^{c}$ is down-up.
\begin{proof}
We provide a bijection $\Phi\colon\mathfrak{J}_{S}\rightarrow\mathfrak{U}_{S}$\textemdash which
is a variant of a map due to Viennot \cite[Section 3]{Viennot1980}\textemdash such
that, for all $\pi\in\mathfrak{J}_{S}$, the first letter of $\Phi(\pi)$
is equal to the last letter of $\pi$. (This condition is vacuously
satisfied for $S=\emptyset$, since the empty permutation has no first
or last letter.) Specializing to $S=[n]$ then implies the result.

We define $\Phi$ recursively as follows. First, $\Phi$ sends the
empty permutation to itself. Otherwise, if $\pi$ is nonempty, then
\[
\Phi(\pi)\coloneqq\Phi(\beta)y\Phi(\alpha)^{c}
\]
where $\pi=\alpha y\beta$ and $y=\min S$. We use induction to show
that $\Phi$ has the desired properties.

Assume that whenever $\left|S^{\prime}\right|<n$, the map $\Phi$
is a bijection from $\mathfrak{J}_{S^{\prime}}$ to $\mathfrak{U}_{S^{\prime}}$
such that $\first(\Phi(\tau))=\last(\tau)$ for all $\tau\in\mathfrak{J}_{S^{\prime}}$.
Now, let $\pi$ be a Jacobi permutation of a set $S$ of size $n$,
and write $\pi=\alpha y\beta$ as described above. In particular,
we know that $\beta$ is of even length because $\pi$ is Jacobi.
Then by the induction hypothesis, both $\Phi(\alpha)$ and $\Phi(\beta)$
are up-down permutations on their respective letter sets\textemdash so
$\Phi(\alpha)^{c}$ is down-up\textemdash and $\first(\Phi(\beta))=\last(\beta)$
if $\beta$ is nonempty. Since $\Phi(\beta)$ is up-down of even length,
$\Phi(\alpha)^{c}$ is down-up, and $y$ is smaller than all letters
in $\Phi(\beta)$ and $\Phi(\alpha)^{c}$, it follows that the concatenation
$\Phi(\pi)=\Phi(\beta)y\Phi(\alpha)^{c}$ is an up-down permutation
on $S$. Also, we have 
\[
\first(\Phi(\pi))=\first(\Phi(\beta))=\last(\beta)=\last(\pi)
\]
if $\beta$ is nonempty, and 
\begin{alignat*}{1}
\first(\Phi(\pi)) & =y=\last(\pi)
\end{alignat*}
otherwise. Finally, if $\pi\in\mathfrak{U}_{S}$ with $\pi=\bar{\beta}y\bar{\alpha}$,
then $\Phi^{-1}(\pi)=\Phi^{-1}(\bar{\alpha}^{c})y\Phi^{-1}(\bar{\beta})$.
Thus, $\Phi\colon\mathfrak{J}_{S}\rightarrow\mathfrak{U}_{S}$ is
a bijection satisfying $\first(\Phi(\pi))=\last(\pi)$ for all $\pi\in\mathfrak{J}_{S}$.
\end{proof}

\section{\label{s-trees}Jacobi permutations and binary trees}

Our study of $213$-, $312$-, and $231$-avoiding Jacobi permutations
will extensively use tree representations of these permutations, based
on classical correspondences between permutations and increasing/decreasing
binary trees. The purpose of this section is to recall these correspondences
and their relevant properties, and to examine their restriction to
Jacobi permutations.

\subsection{Binary trees and traversals}

A \textit{binary tree} is a rooted tree in which each node has at
most two children, which are distinguished as the \textit{left child}
and the \textit{right child}. Each node of a binary tree has a \textit{left
subtree} and a \textit{right subtree}, both of which are (possibly
empty) binary trees. As such, binary trees have a simple recursive
decomposition: a binary tree is either the empty tree, or it consists
of a root with a left subtree and a right subtree. The number of nodes
of a binary tree $T$ is called its \textit{size}, which we will denote
by $\left|T\right|$.

A binary tree can either have labeled nodes or unlabeled nodes. A
binary tree \textit{on} $S$ is a binary tree of size $\left|S\right|$
whose nodes are labeled with the elements of $S$. Such a binary tree
is naturally associated with three permutations of $S$, each of which
is obtained by a depth-first traversal of the tree but with nodes
read in a different order. Each of the three traversals is defined
recursively, starting at the root. In an \textit{in-order traversal},
we do the following:
\begin{enumerate}
\item [\normalfont{(1)}]Recursively traverse the current node's left subtree.
\item [\normalfont{(2)}]Record the label of the current node. (The node
is said to be \textit{visited} at this step.)
\item [\normalfont{(3)}]Recursively traverse the current node's right subtree.
\end{enumerate}
In a \textit{pre-order traversal}, the order of steps (1) and (2)
are swapped, whereas in a \textit{post-order traversal}, the order
of steps (2) and (3) are swapped. See Figure \ref{f-treebij} for
an example of a binary tree on $S=[9]$ and the permutations obtained
via each of these three traversals.
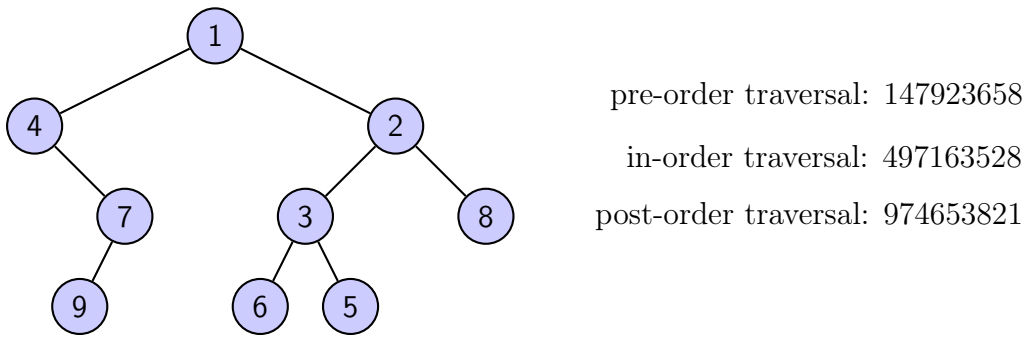
\begin{figure}
\noindent \begin{centering}
\begin{center}
\begin{tikzpicture}[scale=0.4,auto,   thick,branch/.style={circle,fill=blue!20,draw,font=\sffamily,minimum size=1.5em},subtree/.style={circle,draw,dashed,font=\sffamily,minimum size=2em}]
%[every node/.style={circle,fill=blue!20}]

\node[branch] (1) at (10,10) {1};

\node[branch] (4) at (4,7) {4};
\node[branch] (7) at (7,4) {7};
\node[branch] (9) at (5.5,1) {9};

\node[branch] (2) at (16,7) {2};
\node[branch] (3) at (13,4) {3};
\node[branch] (8) at (19,4) {8};
\node[branch] (6) at (11.5,1) {6};
\node[branch] (5) at (14.5,1) {5};

\foreach \from/\to in {1/4, 4/7, 7/9, 1/2, 2/3, 2/8, 3/6, 3/5}
\draw (\from) -- (\to);

\node at (30,8) {pre-order traversal: $147923658$};
\node at (30.25,6) {in-order traversal: $497163528$};
\node at (29.75,4) {post-order traversal: $974653821$};

\end{tikzpicture}
\end{center}
\par\end{centering}
\caption{\label{f-treebij}A binary tree $T$ and its three associated permutations.
Since $T$ is an increasing binary tree, we have $T=\Theta(\pi)$
where $\pi=497163528$.}
\end{figure}

From this point forward, we will identify nodes with their labels
when working with labeled binary trees. That is, if a node is labeled
$x$, then we will refer to the node itself as $x$.

An \textit{increasing binary tree} on $S$ is a binary tree on $S$
such that $x<y$ whenever $y$ is a child of $x$. For example, the
binary tree in Figure \ref{f-treebij} is increasing.\textit{ Decreasing
binary trees} are defined analogously.

The classical bijection $\Theta$ between permutations on $S$ and
increasing binary trees on $S$ is defined as follows. If $\pi$ is
the empty permutation, then $\Theta(\pi)$ is the empty tree. Otherwise,
let $y=\min S$, and write $\pi=\alpha y\beta$. Then we define $\Theta(\pi)$
as the increasing binary tree obtained by taking $y$ as the root,
attaching $\Theta(\alpha)$ as the left subtree of $y$, and attaching
$\Theta(\beta)$ as the right subtree of $y$. Note that $\pi$ is
precisely the permutation obtained from $\Theta(\pi)$ upon performing
an in-order traversal; see Figure \ref{f-treebij} for an example.

\subsection{Permutation statistics and tree statistics}

Given a permutation statistic $\st$, one can often find a statistic
$\st^{\prime}$ on binary trees for which $\st(\pi)=\st^{\prime}(\Theta(\pi))$.
Here we examine such correspondences for the permutation statistics
studied in this paper.

We define the \textit{left branch} of a binary tree recursively as
consisting of the root along with the left child of each node on the
left branch; the \textit{right branch} is defined in the analogous
way. For example, the left branch of the tree in Figure \ref{f-treebij}
consists of the nodes labeled $1$ and $4$, whereas the right branch
consists of the nodes labeled $1$, $2$, and $8$. Observe that $1$
and $4$ are precisely the left-to-right minima of $\pi=497163528$.
\begin{prop}
\label{p-branch}Let $\pi\in\mathfrak{S}_{S}$. Then $x\in S$ is
on the left branch of $\Theta(\pi)$ if and only if $x$ is a left-to-right
minimum of $\pi$.
\end{prop}

\begin{proof}
Suppose that $x$ is on the left branch of $\Theta(\pi)$. Observe
that, in an in-order traversal of $\Theta(\pi)$, the nodes which
are read before $x$ are precisely those in the left subtree of $x$,
so these are the only letters that appear before $x$ in $\pi$. Because
$\Theta(\pi)$ is an increasing tree, $x$ is smaller than every node
in its left subtree, so $x$ is a left-to-right minimum of $\pi$.

For the converse, we will use the fact that every node of $\Theta(\pi)$
is either on the left branch of $\Theta(\pi)$ or is in the right
subtree of a node on the left branch. So, assuming that $x$ is not
on the left branch of $\Theta(\pi)$, we have that $x$ is in the
right subtree of a node $y$. This means that the letter $x$ appears
after $y$ in $\pi$, and also that $y<x$ because $\Theta(\pi)$
is an increasing tree. Therefore, $x$ is not a left-to-right minimum
of $\pi$.
\end{proof}
Let us call the letter $\pi_{k}$ of a permutation $\pi$ a \textit{descent
bottom} if $k-1$ is a descent of $\pi$. (In other words, $\pi_{k}$
is the ``bottom'' of a pair of letters corresponding to a descent
of $\pi$.) For example, the descent bottoms of $\pi=497163528$ are
$7$, $1$, $3$, and $2$, which are precisely the nodes with left
children in the tree $\Theta(\pi)$ displayed in Figure \ref{f-treebij}.
\begin{prop}
\label{p-leftchild}Let $\pi\in\mathfrak{S}_{S}$. Then $x\in S$
has a left child in $\Theta(\pi)$ if and only if $x$ is a descent
bottom of $\pi$.
\end{prop}

\begin{proof}
Suppose that $x\in S$ has a left child in $\Theta(\pi)$, i.e., the
left subtree of $x$ is nonempty. Then, in an in-order traversal of
$\Theta(\pi)$, the node read immediately before $x$ is in the left
subtree of $x$, and this is the letter which immediately precedes
$x$ in $\pi$. Because $\Theta(\pi)$ is an increasing tree, this
letter is larger than $x$, so $x$ is a descent bottom of $\pi$.

Conversely, suppose that $x\in S$ does not have a left child in $\Theta(\pi)$.
If $x$ is on the left branch of $\Theta(\pi)$, then this means that
$x$ is the first letter of $\pi$, in which case it is not a descent
bottom. If $x$ is not on the left branch of $\Theta(\pi)$, then
the node read immediately before $x$ in an in-order traversal is
an ancestor of $x$, which is less than $x$ by virtue of $\Theta(\pi)$
being an increasing tree. Since this is the letter immediately preceding
$x$ in $\pi$, it follows that $x$ is not a descent bottom of $\pi$.
\end{proof}
Let $\lbrch(T)$ denote the number of nodes on the left branch of
a binary tree $T$, and define $\rbrch(T)$ in the same way but for
the right branch. Later, we'll refer to $\lbrch(T)$ as the \textit{size}
of the left branch of $T$, and similarly with $\rbrch(T)$. Also,
let $\lchd(T)$ denote the number of left children in $T$. Then the
following is immediate from Propositions \ref{p-branch}\textendash \ref{p-leftchild}.
\begin{cor}
\label{c-branchild}Let $\pi\in\mathfrak{S}_{S}$. Then 
\begin{enumerate}
\item [\normalfont{(a)}] $\lrmin(\pi)=\lbrch(\Theta(\pi))$ and
\item [\normalfont{(b)}] $\des(\pi)=\lchd(\Theta(\pi))$.
\end{enumerate}
\end{cor}

While permutations in $\mathfrak{S}_{S}$ are in bijection with increasing
binary trees on $S$, the $213$-avoiding permutations in $\mathfrak{S}_{S}$
and the $312$-avoiding permutations in $\mathfrak{S}_{S}$ are each
in bijection with \textit{unlabeled} binary trees on $\left|S\right|$
nodes. For example, notice that the permutation $\pi=645132$ is $213$-avoiding,
and performing a post-order traversal of its corresponding tree $\Theta(\pi)$
yields the decreasing permutation $654321$. Similarly, performing
a pre-order traversal of the tree corresponding to the $312$-avoiding
permutation $324165$ yields the increasing permutation $123456$.
(See Figure \ref{f-treeav}.) In fact, it can be verified that a permutation
$\pi$ of $S$ is $213$-avoiding if and only if a post-order traversal
of $\Theta(\pi)$ yields the elements of $S$ in decreasing order,
and that $\pi$ is $312$-avoiding if and only if a pre-order traversal
of $\Theta(\pi)$ yields the elements of $S$ in increasing order.
Thus, if we know in advance that $\pi$ is $213$- or $312$-avoiding,
then we can remove the labels of $\Theta(\pi)$ and still be able
to recover $\pi$ by traversing the tree appropriately.
\begin{figure}
\noindent \begin{centering}
\begin{center}
\begin{tikzpicture}[scale=0.4,auto,   thick,branch/.style={circle,fill=blue!20,draw,font=\sffamily,minimum size=1.5em},subtree/.style={circle,draw,dashed,font=\sffamily,minimum size=2em}]
%[every node/.style={circle,fill=blue!20}]

\node[branch] (1) at (10,7) {1};
\node[branch] (4) at (7,4) {4};
\node[branch] (6) at (5.5,1) {6};
\node[branch] (5) at (8.5,1) {5};
\node[branch] (2) at (13,4) {2};
\node[branch] (3) at (11.5,1) {3};

\node[branch] (u1) at (24.5,7) {};
\node[branch] (u5) at (21.5,4) {};
\node[branch] (u6) at (20,1) {};
\node[branch] (u3) at (23,1) {};
\node[branch] (u2) at (27.5,4) {};
\node[branch] (u4) at (26,1) {};

\node[branch] (1') at (39,7) {1};
\node[branch] (2') at (36,4) {2};
\node[branch] (3') at (34.5,1) {3};
\node[branch] (4') at (37.5,1) {4};
\node[branch] (5') at (42,4) {5};
\node[branch] (6') at (40.5,1) {6};

\foreach \from/\to in {1/4,4/6,4/5,1/2,2/3,u1/u5,u5/u6,u1/u2,u2/u4,u5/u3,1'/2',2'/3',2'/4',1'/5',5'/6'}
\draw (\from) -- (\to);

\end{tikzpicture}
\end{center}
\par\end{centering}
\caption{\label{f-treeav}The increasing binary trees $\Theta(\pi)$ and $\Theta(\tau)$
for the $213$-avoiding permutation $\pi=645132$ (left) and the $312$-avoiding
permutation $\tau=324165$ (right), along with their corresponding
unlabeled binary tree $T$ (middle).}
\end{figure}
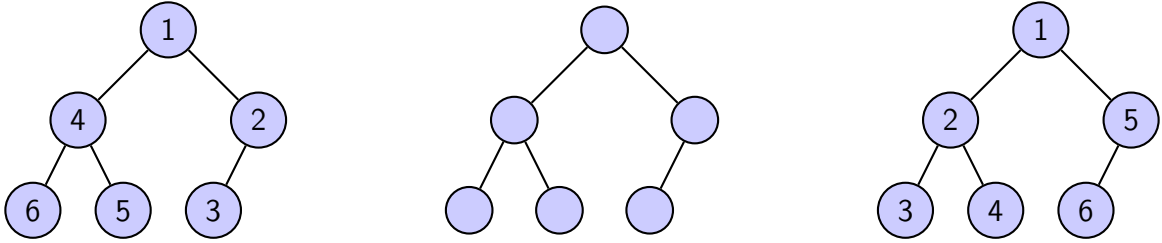

The next proposition allows us to determine the last letter of a $213$-
or $312$-avoiding permutation in $\mathfrak{S}_{n}$ from its tree.
The result for $312$-avoiding permutations $\pi$ involves the number
of nodes in a certain subtree of $\Theta(\pi)$ which we shall now
define. Given a nonempty binary tree $T$, the \textit{terminus} of
$T$ is the bottommost node of the right branch of $T$, and the \textit{undergrowth}
of $T$ is the left subtree of its terminus. For example, the undergrowth
of the tree in Figure \ref{f-treebij} is the empty tree, whereas
the undergrowth of the tree in Figure \ref{f-treeav} consists of
a single node, labeled $3$ in $\Theta(\pi)$ and $6$ in $\Theta(\tau)$.
\begin{prop}
\label{p-213-312-last}Let $\pi\in\mathfrak{S}_{n}$.
\begin{enumerate}
\item [\normalfont{(a)}] If $\pi$ is $213$-avoiding, then $\last(\pi)=\rbrch(\Theta(\pi))$.
\item [\normalfont{(b)}] If $\pi$ is $312$-avoiding, then $\last(\pi)=n-\left|U\right|$,
where $U$ is the undergrowth of $\Theta(\pi)$.
\end{enumerate}
\end{prop}

\begin{proof}
Suppose that $\pi\in\mathfrak{S}_{n}$ is $213$-avoiding, and let
$k=\rbrch(\Theta(\pi))$. The last letter of $\pi$ is the last node
visited in an in-order traversal of $\Theta(\pi)$, which is the terminus
of $\Theta(\pi)$; call this node $x$. On the other hand, in a post-order
traversal of $\Theta(\pi)$, the only nodes visited after $x$ are
the other nodes on the right branch; in other words, $x$ is the $k$-to-last
node visited. Recall that a post-order traversal of $\Theta(\pi)$
necessarily yields the decreasing permutation $n\cdots21$. Therefore,
we have $x=k$, which proves (a).

For (b), suppose that $\pi\in\mathfrak{S}_{n}$ is $312$-avoiding,
and let $k=\left|U\right|$. As above, let $x$ be the terminus of
$\Theta(\pi)$, which is the last letter of $\pi$. In a pre-order
traversal of $\Theta(\pi)$, the only nodes visited after $x$ are
those in the undergrowth $U$, so $x$ is the ($n-k$)th node visited.
Because a pre-order traversal of $\Theta(\pi)$ yields $12\cdots n$,
it follows that $x=n-k$, and we are done.
\end{proof}

\subsection{Jacobi trees}

All of what has been discussed so far hold for permutations in general.
Let us now describe the trees that correspond specifically to Jacobi
permutations. 

We say that an increasing binary tree $T$ on $S$ is a \textit{labeled}
\textit{Jacobi tree} if $T=\Theta(\pi)$ for some Jacobi permutation
$\pi$ of $S$. An \textit{unlabeled Jacobi tree} is an unlabeled
binary tree obtained by taking a labeled Jacobi tree and removing
the labels. We will usually be working with unlabeled Jacobi trees
(as opposed to labeled ones), so we simply call these \textit{Jacobi
trees} from this point forward. Let $\mathcal{J}_{n}$ denote the
set of Jacobi trees on $n$ nodes. From the discussion earlier in
this section, applying $\Theta$ and removing the labels gives us
a bijection between $\mathfrak{J}_{n}(213)$ and $\mathcal{J}_{n}$,
and between $\mathfrak{J}_{n}(312)$ and $\mathcal{J}_{n}$.

Observe that Jacobi trees are precisely unlabeled binary trees in
which every right subtree is of even size. Equivalently, we have the
following recursive decomposition for Jacobi trees: every nonempty
Jacobi tree consists of a root whose left and right subtrees are both
Jacobi trees and the right subtree is of even size.

The enumeration of Jacobi trees is closely related to that of (\textit{unlabeled})
\textit{ternary trees}, which are defined in the same way as (unlabeled)
binary trees, except that each node has three subtrees, distinguished
as the \textit{left subtree}, \textit{middle subtree}, and \textit{right
subtree}. It is well known \cite[A001764 and A006013]{oeis} that,
for all $n\geq0$, there are $\frac{1}{2n+1}{3n \choose n}$ ternary
trees on $n$ nodes and $\frac{1}{2n+1}{3n+1 \choose n+1}$ pairs
of ternary trees with $n$ nodes in total. As shown below, these numbers
also enumerate Jacobi trees of even and odd size, respectively.
\begin{thm}
\label{t-jactree}For all $n\geq0$, we have
\[
\left|\mathcal{J}_{2n}\right|=\frac{1}{2n+1}{3n \choose n}\quad\text{and}\quad\left|\mathcal{J}_{2n+1}\right|=\frac{1}{2n+1}{3n+1 \choose n+1}.
\]
\end{thm}

Of course, Theorem \ref{t-jactree} directly translates to an enumeration
for the Jacobi avoidance classes $\mathfrak{J}_{n}(213)$ and $\mathfrak{J}_{n}(312)$;
this will be stated explicitly in Theorem \ref{t-213-312}.
\begin{proof}
To prove the result for even-sized Jacobi trees, we give a recursive
bijection $\lambda$ between $\mathcal{J}_{2n}$ and ternary trees
on $n$ nodes. As the base case, $\lambda$ sends the empty tree to
the empty tree. Now, let $T$ be a nonempty Jacobi tree of size $2n$.
The root $v$ of $T$ must have a left child $w$, or else the right
subtree of $T$ would have odd size. Let $L$ denote the left subtree
of $w$, let $R$ denote the right subtree of $w$, and let $V$ be
the right subtree of $v$. Note that $L$, $R$, and $V$ all have
even size; after all, $R$ and $V$ are right subtrees, and $T$ would
have an odd number of nodes if $\left|L\right|$ were odd. Thus, we
may define $\lambda(T)$ to be the ternary tree whose root has left
subtree $\lambda(L)$, middle subtree $\lambda(R)$, and right subtree
$\lambda(V)$; see Figure \ref{f-terbij} for an illustration. Since
$\lambda(L)$, $\lambda(R)$, and $\lambda(V)$ have half as many
nodes as $L$, $R$, and $V$, it follows that $\lambda(T)$ is a
ternary tree on $n$ nodes.
\begin{figure}
\noindent \begin{centering}
\begin{center}
\begin{tikzpicture}[scale=0.4,auto,   thick,branch/.style={circle,fill=blue!20,draw,font=\sffamily,minimum size=1em},subtree/.style={circle,draw,dashed,font=\sffamily,minimum size=2em}]
%[every node/.style={circle,fill=blue!20}]

\node (arrow) at (16.5,8.2) {\large $\lambda$};
\node (arrow) at (16.5,7) {\huge $\longmapsto$};

\node[branch] (1) at (10,10) {};
\node[branch] (2) at (8,8) {};
\node[subtree] (3) at (6,6) {$L$};
\node[subtree] (4) at (10,6) {$R$};
\node[subtree] (5) at (12,8) {$V$};

\node[branch] (6) at (25,10) {};
\node[subtree] (7) at (21,6.5) {\footnotesize $\lambda(L)$};
\node[subtree] (8) at (25,6.5) {\footnotesize $\lambda(R)$};
\node[subtree] (9) at (29,6.5) {\footnotesize $\lambda(V)$};

\foreach \from/\to in {1/2}
\draw (\from) -- (\to);

\foreach \from/\to in {2/3, 2/4, 1/5, 6/7, 6/8, 6/9}
\draw[dashed] (\from) -- (\to);

\end{tikzpicture}
\end{center}
\par\end{centering}
\caption{\label{f-terbij}The bijection $\lambda$.}
\end{figure}
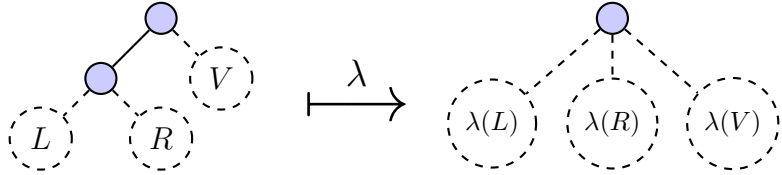

Now, let $T\in\mathcal{J}_{2n+1}$, and let $L$ and $R$ be the left
and right subtrees of its root. By similar reasoning as above, both
$L$ and $R$ must have even size, and $\lambda(L)$ and $\lambda(R)$
have $n$ nodes in total. Then $T\mapsto(\lambda(L),\lambda(R))$
is a bijection between $\mathcal{J}_{2n+1}$ and pairs of ternary
trees with $n$ nodes in total; this yields the enumeration for odd-sized
Jacobi trees.
\end{proof}
Later, we will make use of another recursive decomposition for Jacobi
trees that we call the \textit{undergrowth decomposition}: every nonempty
Jacobi tree $T$ either consists of a root with a left subtree $U$
and an empty right subtree, or can be obtained by taking a nonempty
Jacobi tree $T^{\prime}$, attaching a new node as the right child
of the terminus of $T^{\prime}$, and attaching a Jacobi tree $U$
of odd size\footnote{Note that $U$ must have odd size, or else the terminus of $T^{\prime}$
would have a right subtree of odd size in $T$.} as the left subtree of this new terminus. See Figure \ref{f-rbranchdecomp}
for an illustration of this decomposition. We call this the ``undergrowth
decomposition'' because the tree $U$ in the decomposition is precisely
the undergrowth of $T$.
\begin{figure}
\noindent \begin{centering}
\begin{center}
\begin{tikzpicture}[scale=0.4,auto,   thick,branch/.style={circle,fill=blue!20,draw,font=\sffamily,minimum size=1em},subtree/.style={circle,draw,dashed,font=\sffamily,minimum size=2em}]
%[every node/.style={circle,fill=blue!20}]

\begin{scope}[shift={(2.5,-14.5)},rotate=-45]
\draw[dashed,red!50,fill=red!10] (0,30) ellipse (6.5 and 4.5);
\end{scope}

\node (T') at (27,7.2) {$T'$};

\node at (4,3) {\large $T = $};
\node at (16,3) {\large or};

\node[branch] (1) at (11.5,4) {};
\node[subtree] (1') at (9.5,2) {$U$};

\node[branch] (6) at (22,10) {};
\node[branch] (7) at (24,8) {};
\node[branch] (8) at (27,5) {};
\node[branch] (9) at (31,1) {};
\node[subtree] (6') at (20,8) {};
\node[subtree] (7') at (22,6) {};
\node[subtree] (8') at (25,3) {};
\node[subtree] (9') at (29,-1) {$U$};
\node at (26.25,-1) {\footnotesize (odd)};

\foreach \from/\to in {6/7}
\draw (\from) -- (\to);

\foreach \from/\to in {1/1',7/8,8/9,6/6',7/7',8/8',9/9'}
\draw[dashed] (\from) -- (\to);

\end{tikzpicture}
\end{center}
\par\end{centering}
\caption{\label{f-rbranchdecomp}The undergrowth decomposition of a nonempty
Jacobi tree $T$.}
\end{figure}
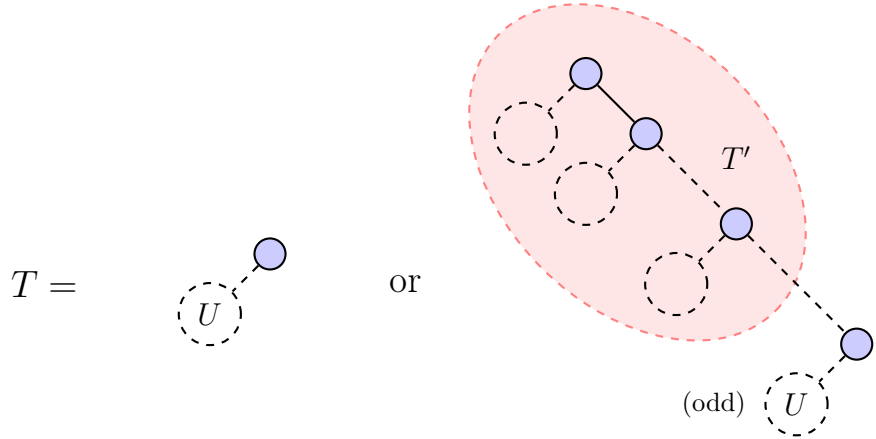

Recall that the proof of Theorem \ref{t-lrmin} utilized a unique
factorization of Jacobi permutations into ``irreducibles'', whose
first letters were the left-to-right minima. There is an analogous
factorization where the last letters of the irreducibles are right-to-left
minima, and all irreducibles other than the first are of even length;
this factorization corresponds to the undergrowth decomposition defined
above.

\subsection{Dual Jacobi trees}

While Jacobi trees are useful for studying $213$- and $312$-avoiding
Jacobi permutations, we will need a different model for Jacobi permutations
avoiding $231$. To that end, let $\tilde{\Theta}$ be the bijection
from permutations on $S$ to decreasing binary trees on $S$, defined
in the same way as $\Theta$ except that we take $y=\max S$ (as opposed
to $y=\min S$). As with $\Theta$, an in-order traversal of $\tilde{\Theta}(\pi)$
recovers the permutation $\pi$. It is readily checked that a permutation
$\pi$ of $S$ is $231$-avoiding if and only if a post-order traversal
of $\tilde{\Theta}(\pi)$ yields the entries of $S$ in increasing
order, which implies that 231-avoiding permutations of $S$ are in
bijection with unlabeled binary trees of size $\left|S\right|$. See
Figure \ref{f-treeav2} for an example.
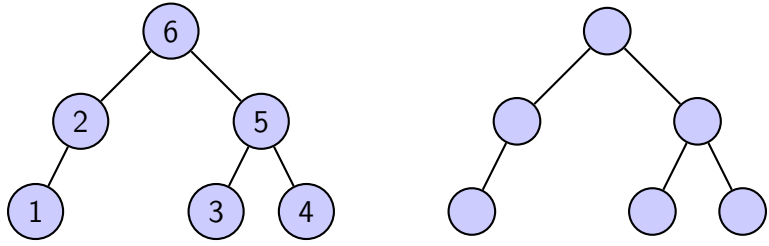
\begin{figure}
\noindent \begin{centering}
\begin{center}
\begin{tikzpicture}[scale=0.4,auto,   thick,branch/.style={circle,fill=blue!20,draw,font=\sffamily,minimum size=1.5em},subtree/.style={circle,draw,dashed,font=\sffamily,minimum size=2em}]
%[every node/.style={circle,fill=blue!20}]

\node[branch] (1) at (10,7) {6};
\node[branch] (5) at (7,4) {2};
\node[branch] (6) at (5.5,1) {1};
\node[branch] (2) at (13,4) {5};
\node[branch] (4) at (11.5,1) {3};
\node[branch] (3) at (14.5,1) {4};

\node[branch] (u1) at (24.5,7) {};
\node[branch] (u5) at (21.5,4) {};
\node[branch] (u6) at (20,1) {};
\node[branch] (u2) at (27.5,4) {};
\node[branch] (u4) at (26,1) {};
\node[branch] (u3) at (29,1) {};

\foreach \from/\to in {1/5,5/6,1/2,2/4,2/3,u1/u5,u5/u6,u1/u2,u2/u4,u2/u3}
\draw (\from) -- (\to);

\end{tikzpicture}
\end{center}
\par\end{centering}
\caption{\label{f-treeav2}The decreasing binary tree $\tilde{\Theta}(\pi)$
for the $231$-avoiding permutation $\pi=126354$, along with the
corresponding unlabeled binary tree $T$.}
\end{figure}

The next proposition is proven similarly to Proposition \ref{p-213-312-last};
we will not give the details.
\begin{prop}
\label{p-231-last}If $\pi\in\mathfrak{S}_{n}(231)$, then $\last(\pi)=n-\rbrch(\tilde{\Theta}(\pi))+1$.
\end{prop}

Let us now restrict to Jacobi permutations. We say that a decreasing
binary tree $T$ on $S$ is a \textit{labeled dual} \textit{Jacobi
tree} if $T=\tilde{\Theta}(\pi)$ for some $\pi\in\mathfrak{J}_{S}$.
An unlabeled binary tree is called an (\textit{unlabeled})\textit{
dual Jacobi tree} if it can be obtained by taking a labeled dual Jacobi
tree and removing the labels. Let $\mathcal{\tilde{J}}_{n}$ denote
the set of dual Jacobi trees on $n$ nodes. Then applying $\tilde{\Theta}$
and removing the labels gives a bijection between $\mathfrak{J}_{n}(231)$
and $\mathcal{\tilde{J}}_{n}$.

Later, we will show that $\mathfrak{J}_{n}(312)$ and $\mathfrak{J}_{n}(231)$
are in bijection, which induces a bijection between $\mathcal{J}_{n}$
and $\mathcal{\tilde{J}}_{n}$. This, along with Theorem \ref{t-jactree},
will imply the following. 
\begin{thm}
\label{t-dualjactree}For all $n\geq0$, we have
\[
\vert\mathcal{\tilde{J}}_{2n}\vert=\frac{1}{2n+1}{3n \choose n}\quad\text{and}\quad\vert\mathcal{\tilde{J}}_{2n+1}\vert=\frac{1}{2n+1}{3n+1 \choose n+1}.
\]
\end{thm}

Our goal for the remainder of this section is to develop a recursive
decomposition for dual Jacobi trees which we shall use when studying
$231$-avoiding Jacobi permutations. We begin with two lemmas.
\begin{lem}
\label{l-231dualrho}Let $\pi\in\mathfrak{S}_{S}(231)$, let $x\in S$,
and let $T=\tilde{\Theta}(\pi)$.
\begin{enumerate}
\item [\normalfont{(a)}] If $x$ has a right child in $T$, then $\rho_{\pi}(x)$
is empty.
\item [\normalfont{(b)}] If $x$ does not have a right child in $\tilde{\Theta}(\pi)$,
then $x$ is a right-to-left minimum of $\pi$, so $\rho_{\pi}(x)$
consists of all subsequent nodes visited in an in-order traversal
of $\tilde{\Theta}(\pi)$.
\end{enumerate}
\end{lem}

\begin{proof}
Suppose that $x$ has a right child in $\tilde{\Theta}(\pi)$. Then,
the letter immediately following $x$ in $\pi$ is a node in the right
subtree of $x$, and this node must be smaller than $x$ because $\tilde{\Theta}(\pi)$
is a decreasing binary tree. Therefore, $\rho_{\pi}(x)$ is empty,
which proves (a).

For a contrapositive proof of (b), suppose that $x$ is not a right-to-left
minimum of $\pi$, so there exists a letter $y<x$ appearing after
$x$ in $\pi$. Due to their relative positions in $\pi$, we see
that $y$ is visited after $x$ in an in-order traversal of $\tilde{\Theta}(\pi)$.
Recall that a post-order traversal of $\tilde{\Theta}(\pi)$ where
$\pi$ is a 231-avoiding permutation returns the letters in increasing
order; since $y<x$, this means that $y$ is visited before $x$ in
a post-order traversal of $\tilde{\Theta}(\pi)$. The nodes visited
after $x$ in an in-order traversal but before $x$ in a post-order
traversal are precisely those in the right subtree of $x$. Thus,
the right subtree of $x$ is nonempty, so $x$ has a right child in
$\tilde{\Theta}(\pi)$.
\end{proof}
\begin{lem}
\label{l-dualug}If $T$ is a dual Jacobi tree, then the undergrowth
of $T$ is empty.
\end{lem}

\begin{proof}
Let $T$ be a dual Jacobi tree and $\pi$ the corresponding $231$-avoiding
Jacobi permutation. Suppose instead that the undergrowth $U$ of $T$
is nonempty, and let $x$ be the terminus of $U$. By Lemma \ref{l-231dualrho}
(b), $\rho_{\pi}(x)$ consists of a single letter\textemdash the terminus
of $T$\textemdash which contradicts $\pi$ being a Jacobi permutation.
\end{proof}
\begin{prop}
\label{p-dualodd}Let $T^{\prime}$ be a nonempty dual Jacobi tree
with terminus $u$, and let $T^{\prime\prime}$ be a dual Jacobi tree
of odd size. Moreover, let $T$ be the tree obtained from $T^{\prime}$
by attaching $T^{\prime\prime}$ as the left subtree of $u$, and
attaching a new node $v$ as the right child of $u$. Then $T$ is
a dual Jacobi tree.
\end{prop}

Note that we are implicitly using Lemma \ref{l-dualug} in the statement
of Proposition \ref{p-dualodd}, as it guarantees that $u$ does not
have a left subtree prior to attaching $T^{\prime\prime}$.
\begin{proof}
For convenience, if $\pi$ is the 231-avoiding permutation corresponding
to $T$, then let us write $\rho_{T}$ in place of $\rho_{\pi}$.
We wish to show that $\rho_{T}(x)$ has even length for every node
$x$ in $T$, and we will do this through cases, appealing to Lemma
\ref{l-231dualrho} in each one.
\begin{itemize}
\item Note that $\rho_{T}(u)$ is empty because $u$ has a right child,
and that $\rho_{T}(v)$ is empty because it is the last node visited
in an in-order traversal of $T$.
\item If $x\neq u$ is a node in $T^{\prime}$ and $\rho_{T^{\prime}}(x)$
is empty, then $\rho_{T}(x)$ is empty because we did not remove any
right children in constructing $T$.
\item If $x$ is a node in $T^{\prime\prime}$ and $\rho_{T^{\prime\prime}}(x)$
is empty, then $\rho_{T}(x)$ is empty for the same reason as above.
\item If $x$ is a node in $T^{\prime}$ and $\rho_{T^{\prime}}(x)$ is
nonempty, then $\rho_{T}(x)$ contains every node in $\rho_{T^{\prime}}(x)$
along with all nodes from $T^{\prime\prime}$ and $v$, as these are
the nodes visited after $x$ in an in-order traversal of $T$. It
follows that
\[
\left|\rho_{T}(x)\right|=\left|\rho_{T^{\prime}}(x)\right|+\left|T^{\prime\prime}\right|+1,
\]
which is even because $\left|\rho_{T^{\prime}}(x)\right|$ is even
and $\left|T^{\prime\prime}\right|$ is odd.
\item If $x$ is a node in $T^{\prime\prime}$ and $\rho_{T^{\prime\prime}}(x)$
is nonempty, then for the same reason as above, $\rho_{T}(x)$ contains
every node in $\rho_{T^{\prime\prime}}(x)$ along with $u$ and $v$.
So, we have 
\[
\left|\rho_{T}(x)\right|=\left|\rho_{T^{\prime\prime}}(x)\right|+2,
\]
which is even because $\left|\rho_{T^{\prime\prime}}(x)\right|$ is
even.\qedhere
\end{itemize}
\end{proof}
We will omit the proofs of the next two propositions as they are similar
to that of Proposition~\ref{p-dualodd}.
\begin{prop}
Let $T^{\prime}$ be a dual Jacobi tree, let $u$ be the terminus
of $T^{\prime}$ \textup{(}if $T^{\prime}$ is nonempty\textup{)},
and let $T^{\prime\prime}$ be a dual Jacobi tree of even size. Moreover,
let $T$ be the tree obtained in the following way:
\begin{itemize}
\item If $T^{\prime}$ is nonempty, attach a new node $v$ as the right
child of $u$, attach $T^{\prime\prime}$ as the left subtree of $v$,
and attach a new node $w$ as the right child of $v$.
\item If $T^{\prime}$ is empty, let $v$ be the root of $T$, attach $T^{\prime\prime}$
as the left subtree of $v$, and attach a new node $w$ as the right
child of $v$.
\end{itemize}
Then $T$ is a dual Jacobi tree.
\end{prop}

\begin{prop}
\label{p-dualrev}Let $T$ be a dual Jacobi tree with at least 2 nodes
on its right branch, let $v$ be the terminus of $T$, let $u$ be
the parent of $v$, and let $T^{\prime\prime}$ be the left subtree
of $u$. Moreover, let $T^{\prime}$ be the tree obtained in the following
way: 
\begin{itemize}
\item If $T^{\prime\prime}$ is of even size, remove $u$, $v$, and $T^{\prime\prime}$
from $T$ to form $T^{\prime}$.
\item If $T^{\prime\prime}$ is of odd size, remove $v$ and $T^{\prime\prime}$
from $T$ to form $T^{\prime}$.
\end{itemize}
Then $T^{\prime}$ and $T^{\prime\prime}$ are dual Jacobi trees.
\end{prop}

Propositions \ref{p-dualodd}\textendash \ref{p-dualrev} justify
the following recursive decomposition for dual Jacobi trees: Every
nonempty dual Jacobi tree either consists of a single node, consists
of a root with a left dual Jacobi subtree of even size and whose right
child is the terminus, or is built up from dual Jacobi trees in the
way described in these propositions. See Figure \ref{f-dualdecomp}
for an illustration. This decomposition can be thought of as an analogue
of the undergrowth decomposition for (ordinary) Jacobi trees, although
here $T^{\prime\prime}$ is not the undergrowth of $T$.\footnote{One might instead call $T^{\prime\prime}$ the \textit{penundergrowth}
of $T$.}
\begin{figure}[!h]
\noindent \begin{centering}
\begin{center}
\begin{tikzpicture}[scale=0.4,auto,   thick,branch/.style={circle,fill=blue!20,draw,font=\sffamily,minimum size=1em},subtree/.style={circle,draw,dashed,font=\sffamily,minimum size=2em}]
%[every node/.style={circle,fill=blue!20}]

%\begin{scope}[shift={(2,5)},rotate=-45]
%\end{scope}

\node[trapezium, trapezium left angle=45, trapezium right angle=90,dashed,draw=red!50,fill=red!10, minimum width=2.5cm, minimum height=2.3cm, rounded corners=.5cm, rotate=135] at (15.5,6) {};

\node[trapezium, trapezium left angle=45, trapezium right angle=90,dashed,draw=red!50,fill=red!10, minimum width=2.5cm, minimum height=2.3cm, rounded corners=.5cm, rotate=135] at (26.5,8) {};

\node at (18.5,3.5) {$T'$};
\node at (29.5,5.5) {$T'$};

\node at (0,2) {\large $T = $};

\node at (5,2) {\large or};
\node at (12.5,2) {\large or};
\node at (25.5,2) {\large or};

\node[branch] (0) at (3,2) {};

\node[branch] (1) at (9,2) {};
\node[branch] (2) at (11,0) {};
\node[subtree] (1') at (7,0) {$T''$};
\node at (7,-2) {\footnotesize (even)};

\node[branch] (3) at (16,8) {};
\node[branch] (4) at (18,6) {};
\node[branch] (5) at (21,3) {};
\node[branch] (6) at (24,0) {};
\node[subtree] (3') at (14,6) {};
\node[subtree] (4') at (16,4) {};
\node[subtree] (5') at (18,0) {$T''$};
\node at (18,-2) {\footnotesize (odd)};

\node[branch] (7) at (27,10) {};
\node[branch] (8) at (29,8) {};
\node[branch] (9) at (32,5) {};
\node[branch] (10) at (35,2) {};
\node[branch] (11) at (37,0) {};
\node[subtree] (7') at (25,8) {};
\node[subtree] (8') at (27,6) {};
\node[subtree] (10') at (33,0) {$T''$};
\node at (33,-2) {\footnotesize (even)};

\foreach \from/\to in {1/2,3/4,7/8,10/11}
\draw (\from) -- (\to);

\foreach \from/\to in {1/1',4/5,5/6,3/3',4/4',5/5',8/9,9/10,7/7',8/8',10/10'}
\draw[dashed] (\from) -- (\to);

\end{tikzpicture}
\end{center}
\par\end{centering}
\caption{\label{f-dualdecomp}Recursive decomposition of a nonempty dual Jacobi
tree $T$.}
\end{figure}
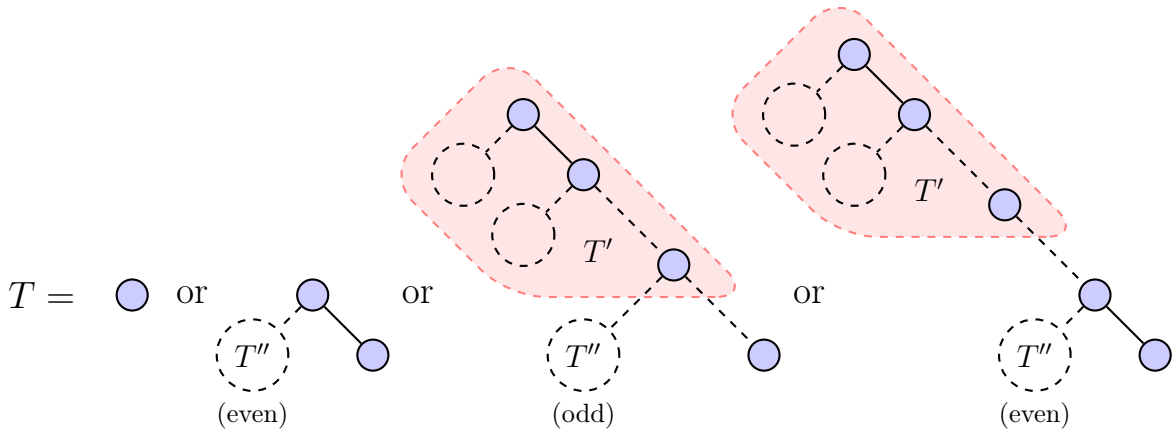

\section{\label{s-213-312}\texorpdfstring{$213$}{213}- and \texorpdfstring{$312$}{312}-avoiding
Jacobi permutations}

This section focuses on the Jacobi avoidance classes $\mathfrak{J}_{n}(213)$
and $\mathfrak{J}_{n}(312)$. We begin by restating Theorem \ref{t-jactree}
as an enumeration for $\mathfrak{J}_{n}(213)$ and $\mathfrak{J}_{n}(312)$.
\begin{thm}
\label{t-213-312}For all $n\geq0$, we have
\[
j_{2n}(213)=j_{2n}(312)=\frac{1}{2n+1}{3n \choose n}\quad\text{and}\quad j_{2n+1}(213)=j_{2n+1}(312)=\frac{1}{2n+1}{3n+1 \choose n+1}.
\]
\end{thm}

Using results from Section \ref{s-trees} relating certain permutation
statistics to tree statistics, we can determine the distributions
of these permutation statistics over $\mathfrak{J}_{n}(213)$ and
$\mathfrak{J}_{n}(312)$ by studying distributions of the corresponding
statistics over $\mathcal{J}_{n}$. We shall take this approach in
this section to count $213$- and $312$-avoiding Jacobi permutations
by the number of descents, the number of left-to-right minima, and
the value of the last letter.\footnote{While it is possible to prove the results in this section by working
directly with permutations, we found it easier to formulate our bijections
using Jacobi trees. (Some of our bijections do not seem to have simple
descriptions in terms of the permutations themselves.) Another advantage
of this approach is that it allows us to prove most of our results
simultaneously for $\mathfrak{J}_{n}(213)$ and $\mathfrak{J}_{n}(312)$.}
\begin{table}[H]
\begin{centering}
\renewcommand{\arraystretch}{1.1}%
\begin{tabular}{c|>{\centering}p{20bp}|>{\centering}p{20bp}|>{\centering}p{20bp}|>{\centering}p{20bp}|>{\centering}p{20bp}|>{\centering}p{20bp}|>{\centering}p{20bp}|>{\centering}p{20bp}|>{\centering}p{20bp}|>{\centering}p{20bp}|>{\centering}p{20bp}|>{\centering}p{20bp}}
$n$ & $0$ & $1$ & $2$ & $3$ & $4$ & $5$ & $6$ & $7$ & $8$ & $9$ & $10$ & $11$\tabularnewline
\hline 
$j_{n}(213)$ & $1$ & $1$ & $1$ & $2$ & $3$ & $7$ & $12$ & $30$ & $55$ & $143$ & $273$ & $728$\tabularnewline
\end{tabular}
\par\end{centering}
\caption{\label{tb-213-312}The numbers $j_{n}(213)=j_{n}(312)$ up to $n=11$.}
\end{table}

\subsection{Descents}

Counting Jacobi trees by the number of left children will allow us
to achieve a refined enumeration of $213$- and $312$-avoiding Jacobi
permutations by the number of descents. Our goal is to prove the following.
\begin{thm}
\label{t-213-312-des}We have
\begin{align*}
j_{2n,k}^{\des}(213)=j_{2n,k}^{\des}(312) & =\frac{1}{n}{n \choose k-n}{2n \choose k+1} & \text{for all }n\geq1\text{ and }k\geq0;\\
j_{2n+1,k}^{\des}(213)=j_{2n+1,k}^{\des}(312) & =\frac{1}{n+1}{n+1 \choose k-n}{2n \choose k} & \text{for all }n,k\geq0.
\end{align*}
\begin{table}
\begin{centering}
\renewcommand{\arraystretch}{1.1}%
\begin{tabular}{|>{\centering}p{20bp}|>{\centering}p{20bp}|>{\centering}p{20bp}|>{\centering}p{20bp}|>{\centering}p{20bp}|>{\centering}p{20bp}|>{\centering}p{20bp}|>{\centering}p{20bp}|>{\centering}p{20bp}|>{\centering}p{20bp}|}
\hline 
$n\backslash k$ & $0$ & $1$ & $2$ & $3$ & $4$ & $5$ & $6$ & $7$ & $8$\tabularnewline
\hline 
$0$ & $1$ &  &  &  &  &  &  &  & \tabularnewline
\hline 
$1$ & $1$ &  &  &  &  &  &  &  & \tabularnewline
\hline 
$2$ & $0$ & $1$ &  &  &  &  &  &  & \tabularnewline
\hline 
$3$ & $0$ & $1$ & $1$ &  &  &  &  &  & \tabularnewline
\hline 
$4$ & $0$ & $0$ & $2$ & $1$ &  &  &  &  & \tabularnewline
\hline 
$5$ & $0$ & $0$ & $2$ & $4$ & $1$ &  &  &  & \tabularnewline
\hline 
$6$ & $0$ & $0$ & $0$ & $5$ & $6$ & $1$ &  &  & \tabularnewline
\hline 
$7$ & $0$ & $0$ & $0$ & $5$ & $15$ & $9$ & $1$ &  & \tabularnewline
\hline 
$8$ & $0$ & $0$ & $0$ & $0$ & $14$ & $28$ & $12$ & $1$ & \tabularnewline
\hline 
$9$ & $0$ & $0$ & $0$ & $0$ & $14$ & $56$ & $56$ & $16$ & $1$\tabularnewline
\hline 
\end{tabular}
\par\end{centering}
\caption{\label{tb-213-312-des}The numbers $j_{n,k}^{\protect\des}(213)=j_{n,k}^{\protect\des}(312)$
up to $n=9$.}
\end{table}
\end{thm}

Define the generating functions
\[
F_{\e}=F_{\e}(t,x)\coloneqq\sum_{n=0}^{\infty}\sum_{T\in\mathcal{J}_{2n}}t^{\lchd(T)}x^{2n}\qquad\text{and}\qquad F_{\o}=F_{\o}(t,x)\coloneqq\sum_{n=0}^{\infty}\sum_{T\in\mathcal{J}_{2n+1}}t^{\lchd(T)}x^{2n+1},
\]
which count even- and odd-sized Jacobi trees, respectively, by the
number of left children. The next lemma provides several functional
equations satisfied by these generating functions.
\begin{lem}
\label{l-jactreefunceq}We have \leqnomode 
\begin{align*}
\tag{{a}}F_{\e} & =1+tx^{2}(1-t)F_{\e}^{2}+t^{2}x^{2}F_{\e}^{3},\\
\tag{{b}}F_{\e} & =\frac{1}{1-txF_{\o}},\\
\tag{{c}}F_{\o} & =xF_{\e}+tx(F_{\e}-1)F_{\e},\quad\text{and}\\
\tag{{d}}F_{\o} & =x+(t-1)tx^{2}F_{\o}+2txF_{\o}^{2}-t^{2}x^{2}F_{\o}^{3}.
\end{align*}
\end{lem}

\begin{proof}
We first prove (c). Let $T$ be a Jacobi tree of odd size. Then $T$
consists of a root whose left subtree $L$ and right subtree $R$
are both Jacobi trees of even size. The case where $L$ is empty contributes
the term $xF_{\e}$, since $\lchd(T)=\lchd(R)$. If $L$ is nonempty,
then $\lchd(T)=\lchd(L)+\lchd(R)+1$, so this case contributes $tx(F_{\e}-1)F_{\e}$.
Summing these two terms yields (c).

To prove (a) and (b), now let $T$ be a Jacobi tree of even size.
The case where $T$ is empty contributes the term 1. If $T$ is nonempty,
then $T$ consists of a root whose left subtree $L$ is a Jacobi tree
of odd size and whose right subtree $R$ is a Jacobi tree of even
size. In particular, $L$ is nonempty, so $\lchd(T)=\lchd(L)+\lchd(R)+1$.
Thus, this case contributes $txF_{\o}F_{\e}$, and we have 
\begin{align}
F_{\e} & =1+txF_{\o}F_{\e}.\label{e-213dese3}
\end{align}
Solving for $F_{\e}$ in (\ref{e-213dese3}) yields (b). Furthermore,
substituting (c) into (\ref{e-213dese3}) and simplifying results
in (a).

Finally, (d) can be obtained from substituting (b) into (c) and performing
some routine algebraic manipulations.
\end{proof}
In order to obtain formulas for their coefficients, we will relate
$F_{\e}$ and $F_{\o}$ to two generating functions which have previously
appeared in the literature. Let $U=U(t,x)$ and $V=V(t,x)$ be the
unique formal power series solutions to the functional equations 
\begin{align}
U & =(1+txU)(1+xU)^{2}\qquad\text{and}\label{e-Ufunc}\\
V(1-xV)^{2} & =1+(t-1)xV,\label{e-Vfunc}
\end{align}
respectively. The coefficient of $t^{k}x^{n}$ in $U$ counts ternary
trees on $n$ edges with $k$ middle children\textemdash as can be
seen directly from (\ref{e-Ufunc})\textemdash and these coefficients
are given by the formula
\begin{equation}
[t^{k}x^{n}]\,U=\frac{1}{n+1}{n+1 \choose k}{2(n+1) \choose n-k}\label{e-Ucoeff}
\end{equation}
\cite[A120986]{oeis}. The coefficients of $V$, on the other hand,
count a certain family of Feynman diagrams refined by the number of
fermionic loops \cite{Molinari2005}, and are given by

\begin{equation}
[t^{k}x^{n}]\,V=\frac{1}{n+1}{n+1 \choose k}{2n \choose n+k}\label{e-Vcoeff}
\end{equation}
\cite[A286784]{oeis}.
\begin{lem}
\label{l-jactreeUV}We have \leqnomode
\begin{align*}
\tag{{a}}F_{\e}(t,x) & =1+tx^{2}U(t,tx^{2})\quad\text{and}\\
\tag{{b}}F_{\o}(t,x) & =xV(t,tx^{2}).
\end{align*}
\end{lem}

\begin{proof}
To prove (a), first we replace each instance of $x$ with $tx^{2}$
in (\ref{e-Ufunc}) to obtain
\begin{equation}
U(t,tx^{2})=(1+t^{2}x^{2}U(t,tx^{2}))(1+tx^{2}U(t,tx^{2}))^{2}.\label{e-Ufuncsub}
\end{equation}
Let $\bar{U}=\bar{U}(t,x)\coloneqq1+tx^{2}U(t,tx^{2})$. Substituting
$U(t,tx^{2})=(\bar{U}-1)/tx^{2}$ into (\ref{e-Ufuncsub}) gives us
\[
\frac{\bar{U}-1}{tx^{2}}=\left(1+t^{2}x^{2}\left(\frac{\bar{U}-1}{tx^{2}}\right)\right)\left(1+tx^{2}\left(\frac{\bar{U}-1}{tx^{2}}\right)\right)^{2}.
\]
Multiplying both sides by $tx^{2}$, adding $1$ to both sides, and
simplifying yields 
\[
\bar{U}=1+tx^{2}(1-t)\bar{U}^{2}+t^{2}x^{2}\bar{U}^{3};
\]
comparing with Lemma \ref{l-jactreefunceq} (a), we conclude $F_{\e}(t,x)=\bar{U}=1+tx^{2}U(t,tx^{2})$.

Next, we prove (b). We make the substitution $x\mapsto tx^{2}$ in
(\ref{e-Vfunc}), giving us 
\begin{equation}
V(t,tx^{2})(1-tx^{2}V(t,tx^{2}))^{2}=1+(t-1)tx^{2}V(t,tx^{2}).\label{e-Vfuncsub}
\end{equation}
Let $\bar{V}=\bar{V}(t,x)\coloneqq xV(t,tx^{2})$, so that $V(t,tx^{2})=x^{-1}\bar{V}$.
Then (\ref{e-Vfuncsub}) becomes
\begin{equation}
\bar{V}(1-tx\bar{V})^{2}=x+(t-1)tx^{2}\bar{V}.\label{e-Vbar}
\end{equation}
Expanding the left-hand side of (\ref{e-Vbar}) and rearranging yields
\[
\bar{V}=x+(t-1)tx^{2}\bar{V}+2tx\bar{V}^{2}-t^{2}x^{2}\bar{V}^{3},
\]
which proves $F_{\o}(t,x)=\bar{V}=xV(t,tx^{2})$ upon comparing with
Lemma \ref{l-jactreefunceq} (d).
\end{proof}
We are now ready to prove Theorem \ref{t-213-312-des}.
\begin{proof}[Proof of Theorem \ref{t-213-312-des}]
Using Lemma \ref{l-jactreeUV} (a) and Equation (\ref{e-Ucoeff}),
we obtain
\begin{align}
F_{\e}(t,x)=1+tx^{2}U(t,tx^{2}) & =1+\sum_{n=0}^{\infty}\sum_{k=0}^{n}\frac{1}{n+1}{n+1 \choose k}{2(n+1) \choose n-k}t^{n+k+1}x^{2n+2}\nonumber \\
 & =1+\sum_{n=1}^{\infty}\sum_{k=0}^{n-1}\frac{1}{n}{n \choose k}{2n \choose n-1-k}t^{n+k}x^{2n}\nonumber \\
 & =1+\sum_{n=1}^{\infty}\sum_{k=n}^{2n-1}\frac{1}{n}{n \choose k-n}{2n \choose 2n-1-k}t^{k}x^{2n}\nonumber \\
 & =1+\sum_{n=1}^{\infty}\sum_{k=n}^{2n-1}\frac{1}{n}{n \choose k-n}{2n \choose k+1}t^{k}x^{2n}.\label{e-Fe}
\end{align}
Similarly, by Lemma \ref{l-jactreeUV} (b) and Equation (\ref{e-Vcoeff}),
we have
\begin{align}
F_{\o}(t,x)=xV(t,tx^{2}) & =\sum_{n=0}^{\infty}\sum_{k=0}^{n}\frac{1}{n+1}{n+1 \choose k}{2n \choose n+k}t^{n+k}x^{2n+1}\nonumber \\
 & =\sum_{n=0}^{\infty}\sum_{k=n}^{2n}\frac{1}{n+1}{n+1 \choose k-n}{2n \choose k}t^{k}x^{2n+1}.\label{e-Fo}
\end{align}
Moreover, we have
\[
j_{2n,k}^{\des}(213)=j_{2n,k}^{\des}(312)=[t^{k}x^{2n}]\,F_{\e}(t,x)
\]
and 
\[
j_{2n+1,k}^{\des}(213)=j_{2n+1,k}^{\des}(312)=[t^{k}x^{2n+1}]\,F_{\e}(t,x)
\]
from Corollary \ref{c-branchild} (b); combining these with (\ref{e-Fe})
and (\ref{e-Fo}) completes the proof.
\end{proof}

\subsection{Left-to-right minima}

Our next result counts permutations in $\mathfrak{J}_{n}(213)$ and
$\mathfrak{J}_{n}(312)$ by the number of left-to-right minima.
\begin{thm}
\label{t-213-312-lrmin}For all $n\geq k\geq1$, we have
\[
j_{n,k}^{\lrmin}(213)=j_{n,k}^{\lrmin}(312)=\frac{2k}{3n-k}{\frac{3n-k}{2} \choose n}
\]
if $n$ and $k$ have the same parity, and $j_{n,k}^{\lrmin}(213)=j_{n,k}^{\lrmin}(312)=0$
otherwise.
\begin{table}
\begin{centering}
\renewcommand{\arraystretch}{1.1}%
\begin{tabular}{|>{\centering}p{20bp}|>{\centering}p{20bp}|>{\centering}p{20bp}|>{\centering}p{20bp}|>{\centering}p{20bp}|>{\centering}p{20bp}|>{\centering}p{20bp}|>{\centering}p{20bp}|>{\centering}p{20bp}|>{\centering}p{20bp}|>{\centering}p{20bp}|}
\hline 
$n\backslash k$ & $0$ & $1$ & $2$ & $3$ & $4$ & $5$ & $6$ & $7$ & $8$ & $9$\tabularnewline
\hline 
$0$ & $1$ &  &  &  &  &  &  &  &  & \tabularnewline
\hline 
$1$ & $0$ & $1$ &  &  &  &  &  &  &  & \tabularnewline
\hline 
$2$ & $0$ & $0$ & $1$ &  &  &  &  &  &  & \tabularnewline
\hline 
$3$ & $0$ & $1$ & $0$ & $1$ &  &  &  &  &  & \tabularnewline
\hline 
$4$ & $0$ & $0$ & $2$ & $0$ & $1$ &  &  &  &  & \tabularnewline
\hline 
$5$ & $0$ & $3$ & $0$ & $3$ & $0$ & $1$ &  &  &  & \tabularnewline
\hline 
$6$ & $0$ & $0$ & $7$ & $0$ & $4$ & $0$ & $1$ &  &  & \tabularnewline
\hline 
$7$ & $0$ & $12$ & $0$ & $12$ & $0$ & $5$ & $0$ & $1$ &  & \tabularnewline
\hline 
$8$ & $0$ & $0$ & $30$ & $0$ & $18$ & $0$ & $6$ & $0$ & $1$ & \tabularnewline
\hline 
$9$ & $0$ & $55$ & $0$ & $55$ & $0$ & $25$ & $0$ & $7$ & $0$ & $1$\tabularnewline
\hline 
\end{tabular}
\par\end{centering}
\caption{\label{tb-213-312-lrmin}The numbers $j_{n,k}^{\protect\lrmin}(213)=j_{n,k}^{\protect\lrmin}(312)$
up to $n=9$.}
\end{table}
\end{thm}

From Table \ref{tb-213-312-lrmin}, the reader may notice a pattern
captured by the following recurrence.
\begin{prop}
\label{p-213-312-lrmin-rec}The numbers $j_{n,k}^{\lrmin}(213)=j_{n,k}^{\lrmin}(312)$
satisfy
\[
j_{n,k}^{\lrmin}(213)=j_{n-1,k-1}^{\lrmin}(213)+j_{n,k+2}^{\lrmin}(213)
\]
for all $n\geq k\geq1$.
\end{prop}

In fact, Theorem \ref{t-213-312-lrmin} follows readily from an induction
argument using Proposition~\ref{p-213-312-lrmin-rec}. We will omit
the details of this induction, and instead focus our attention on
proving Proposition~\ref{p-213-312-lrmin-rec}.
\begin{proof}
For convenience, let $\mathcal{J}_{n,k}$ denote the set of trees
in $\mathcal{J}_{n}$ that have a left branch of size $k$. By Corollary
\ref{c-branchild} (a), it suffices to construct a bijection $\psi\colon\mathcal{J}_{n,k}\rightarrow\mathcal{J}_{n-1,k-1}\sqcup\mathcal{J}_{n,k+2}$.

Given $T\in\mathcal{J}_{n,k}$, let $u$ denote the bottommost node
in its left branch. If the right subtree of $u$ is empty, then $\psi(T)\in\mathcal{J}_{n-1,k-1}$
is defined to be the Jacobi tree obtained from $T$ upon removing
$u$.

Now, suppose the right subtree of $u$ is nonempty. Let $v$ denote
the right child of $u$. If $v$ does not have a left child, then
$u$ would have a right subtree of odd size (consisting of $v$ and
the right subtree of $v$), a contradiction. So, let $w$ denote the
left child of $v$; also let $L$ be the left subtree of $w$, let
$R$ be the right subtree of $w$, and let $V$ be the right subtree
of $v$. By a similar argument, $L$ must have even size or else the
right subtree of $u$ would have odd size. We construct $\psi(T)$
from $T$ in the following way:
\begin{enumerate}
\item Remove $v$ and all its descendants from $T$.
\item Add a new node $v^{\prime}$ as the left child of $u$, and a new
node $w^{\prime}$ as the left child of $v^{\prime}$.
\item Attach $L$, $R$, and $V$ as the right subtrees of $w^{\prime}$,
$v^{\prime}$, and $u$, respectively.
\end{enumerate}
See Figure \ref{f-lrminbij} for an illustration. Observe that we
have increased the size of the left branch by 2, and that $\psi(T)$
is a Jacobi tree because $L$, $R$, and $V$ all have even size.
Therefore, $\psi(T)\in\mathcal{J}_{n,k+2}$ in this case.
\begin{figure}[!h]
\noindent \begin{centering}
\begin{center}
\begin{tikzpicture}[scale=0.4,auto,   thick,branch/.style={circle,fill=blue!20,draw,font=\sffamily,minimum size=1em},subtree/.style={circle,draw,dashed,font=\sffamily,minimum size=2em}]
%[every node/.style={circle,fill=blue!20}]

\node (arrow) at (16,26.2) {\large $\psi$};
\node (arrow) at (16,25) {\huge $\longmapsto$};

\node (arrow) at (16,14.2) {\large $\psi$};
\node (arrow) at (16,13) {\huge $\longmapsto$};

\node[branch] (1) at (10,30) {};
\node[branch] (2) at (8,28) {};
\node[branch] (3) at (5,25) {};
\node[branch] (4) at (3,23) {};
\node[subtree] (1') at (12,28) {};
\node[subtree] (2') at (10,26) {};
\node[subtree] (3') at (7,23) {};

\node[branch] (5) at (25,30) {};
\node[branch] (6) at (23,28) {};
\node[branch] (7) at (20,25) {};
\node[subtree] (5') at (27,28) {};
\node[subtree] (6') at (25,26) {};
\node[subtree] (7') at (22,23) {};

\node[branch] (8) at (10,19) {};
\node[branch] (9) at (8,17) {};
\node[branch] (10) at (5,14) {};
\node[branch] (11) at (3,12) {};
\node[branch] (12) at (5,10) {};
\node[branch] (13) at (3,8) {};
\node[subtree] (8') at (12,17) {};
\node[subtree] (9') at (10,15) {};
\node[subtree] (10') at (7,12) {};
\node[subtree] (11') at (7,8) {$V$};
\node[subtree] (12') at (1,6) {$L$};
\node[subtree] (13') at (5,6) {$R$};
\node at (9.75,8) {\footnotesize (even)};
\node at (1,4) {\footnotesize (even)};
\node at (5,4) {\footnotesize (even)};

\node[branch] (14) at (30,19) {};
\node[branch] (15) at (28,17) {};
\node[branch] (16) at (25,14) {};
\node[branch] (17) at (23,12) {};
\node[branch] (18) at (21,10) {};
\node[branch] (19) at (19,8) {};
\node[subtree] (14') at (32,17) {};
\node[subtree] (15') at (30,15) {};
\node[subtree] (16') at (27,12) {};
\node[subtree] (17') at (25,10) {$V$};
\node[subtree] (18') at (23,8) {$R$};
\node[subtree] (19') at (21,6) {$L$};

\foreach \from/\to in {1/2,3/4,5/6,8/9,10/11,11/12,12/13,14/15,16/17,17/18,18/19}
\draw (\from) -- (\to);

\foreach \from/\to in {2/3,1/1',2/2',3/3',6/7,5/5',6/6',7/7',9/10,8/8',9/9',10/10',12/11',13/12',13/13',15/16,14/14',15/15',16/16',17/17',18/18',19/19'}
\draw[dashed] (\from) -- (\to);

\end{tikzpicture}
\end{center}
\par\end{centering}
\caption{\label{f-lrminbij}The bijection $\psi$.}
\end{figure}
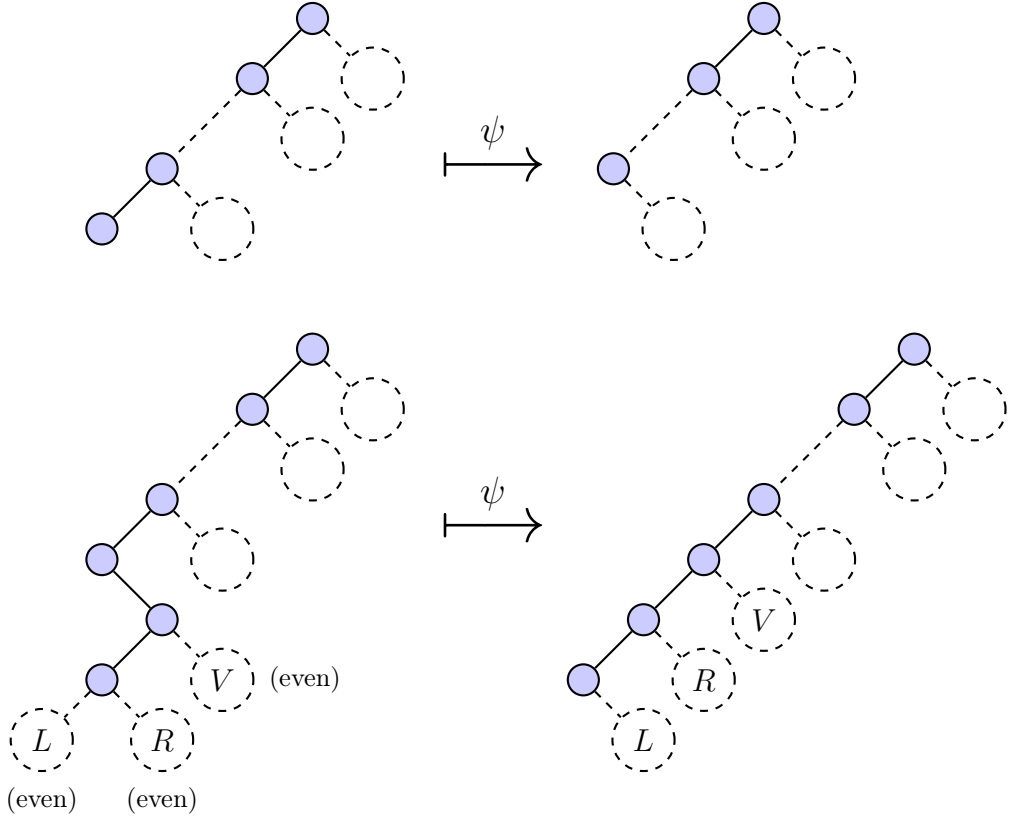

Now, suppose that $T\in\mathcal{J}_{n-1,k-1}\sqcup\mathcal{J}_{n,k+2}$.
We can recover the corresponding tree in $\mathcal{J}_{n,k}$ by either
adding a node to the bottom of the left branch if $T\in\mathcal{J}_{n-1,k-1}$,
or by reversing the steps in the above procedure for $T\in\mathcal{J}_{n,k+2}$.
We conclude that $\psi$ is a bijection.
\end{proof}

\subsection{Last letter}

Our last two theorems of this section give the distributions of $\last$
over $\mathfrak{J}_{n}(213)$ and $\mathfrak{J}_{n}(312)$.
\begin{thm}
\label{t-213last}For all $n\geq k\geq1$, we have
\[
j_{2n,k}^{\last}(213)=\frac{2k}{3n-k}{3n-k \choose 2n}\qquad\text{and}\qquad j_{2n-1,k}^{\last}(213)=\frac{2k-1}{3n-k-1}{3n-k-1 \choose 2n-1}.
\]
\end{thm}

\begin{thm}
\label{t-312last}For all $n\geq k\geq1$, we have 
\[
j_{2n,2k-1}^{\last}(312)=\frac{{3k-3 \choose k-1}{3n-3k+1 \choose n-k+1}}{(2k-1)(2n-2k+1)}\qquad\text{and}\qquad j_{2n+1,2k}^{\last}(312)=\frac{{3k-2 \choose k}{3n-3k+1 \choose n-k+1}}{(2k-1)(2n-2k+1)}.
\]
In addition, we have 
\[
j_{2n+1,1}^{\last}(312)=\frac{1}{2n+1}{3n \choose n}
\]
for all $n\geq0$ and, unless $k=1$, $j_{n,k}^{\last}(312)=0$ whenever
$n$ and $k$ have the same parity.
\begin{table}
\begin{centering}
\renewcommand{\arraystretch}{1.1}%
\begin{tabular}{|>{\centering}p{20bp}|>{\centering}p{14bp}|>{\centering}p{14bp}|>{\centering}p{14bp}|>{\centering}p{14bp}|>{\centering}p{14bp}|>{\centering}p{25bp}|>{\centering}p{20bp}|>{\centering}p{14bp}|>{\centering}p{14bp}|>{\centering}p{14bp}|>{\centering}p{14bp}|>{\centering}p{14bp}|>{\centering}p{14bp}|>{\centering}p{14bp}|>{\centering}p{14bp}|}
\multicolumn{6}{c}{$j_{n,k}^{\last}(213)$} & \multicolumn{1}{>{\centering}p{25bp}}{$\phantom{{\displaystyle \frac{dy}{dx}}}$} & \multicolumn{9}{c}{$j_{n,k}^{\last}(312)$}\tabularnewline
\cline{1-6} \cline{2-6} \cline{3-6} \cline{4-6} \cline{5-6} \cline{6-6} \cline{8-16} \cline{9-16} \cline{10-16} \cline{11-16} \cline{12-16} \cline{13-16} \cline{14-16} \cline{15-16} \cline{16-16} 
$n\backslash k$ & $1$ & $2$ & $3$ & $4$ & $5$ &  & $n\backslash k$ & $1$ & $2$ & $3$ & $4$ & $5$ & $6$ & $7$ & $8$\tabularnewline
\cline{1-6} \cline{2-6} \cline{3-6} \cline{4-6} \cline{5-6} \cline{6-6} \cline{8-16} \cline{9-16} \cline{10-16} \cline{11-16} \cline{12-16} \cline{13-16} \cline{14-16} \cline{15-16} \cline{16-16} 
$1$ & $1$ &  &  &  &  &  & $1$ & $1$ &  &  &  &  &  &  & \tabularnewline
\cline{1-6} \cline{2-6} \cline{3-6} \cline{4-6} \cline{5-6} \cline{6-6} \cline{8-16} \cline{9-16} \cline{10-16} \cline{11-16} \cline{12-16} \cline{13-16} \cline{14-16} \cline{15-16} \cline{16-16} 
$2$ & $1$ &  &  &  &  &  & $2$ & $1$ &  &  &  &  &  &  & \tabularnewline
\cline{1-6} \cline{2-6} \cline{3-6} \cline{4-6} \cline{5-6} \cline{6-6} \cline{8-16} \cline{9-16} \cline{10-16} \cline{11-16} \cline{12-16} \cline{13-16} \cline{14-16} \cline{15-16} \cline{16-16} 
$3$ & $1$ & $1$ &  &  &  &  & $3$ & $1$ & $1$ &  &  &  &  &  & \tabularnewline
\cline{1-6} \cline{2-6} \cline{3-6} \cline{4-6} \cline{5-6} \cline{6-6} \cline{8-16} \cline{9-16} \cline{10-16} \cline{11-16} \cline{12-16} \cline{13-16} \cline{14-16} \cline{15-16} \cline{16-16} 
$4$ & $2$ & $1$ &  &  &  &  & $4$ & $2$ & $0$ & $1$ &  &  &  &  & \tabularnewline
\cline{1-6} \cline{2-6} \cline{3-6} \cline{4-6} \cline{5-6} \cline{6-6} \cline{8-16} \cline{9-16} \cline{10-16} \cline{11-16} \cline{12-16} \cline{13-16} \cline{14-16} \cline{15-16} \cline{16-16} 
$5$ & $3$ & $3$ & $1$ &  &  &  & $5$ & $3$ & $2$ & $0$ & $2$ &  &  &  & \tabularnewline
\cline{1-6} \cline{2-6} \cline{3-6} \cline{4-6} \cline{5-6} \cline{6-6} \cline{8-16} \cline{9-16} \cline{10-16} \cline{11-16} \cline{12-16} \cline{13-16} \cline{14-16} \cline{15-16} \cline{16-16} 
$6$ & $7$ & $4$ & $1$ &  &  &  & $6$ & $7$ & $0$ & $2$ & $0$ & $3$ &  &  & \tabularnewline
\cline{1-6} \cline{2-6} \cline{3-6} \cline{4-6} \cline{5-6} \cline{6-6} \cline{8-16} \cline{9-16} \cline{10-16} \cline{11-16} \cline{12-16} \cline{13-16} \cline{14-16} \cline{15-16} \cline{16-16} 
$7$ & $12$ & $12$ & $5$ & $1$ &  &  & $7$ & $12$ & $7$ & $0$ & $4$ & $0$ & $7$ &  & \tabularnewline
\cline{1-6} \cline{2-6} \cline{3-6} \cline{4-6} \cline{5-6} \cline{6-6} \cline{8-16} \cline{9-16} \cline{10-16} \cline{11-16} \cline{12-16} \cline{13-16} \cline{14-16} \cline{15-16} \cline{16-16} 
$8$ & $30$ & $18$ & $6$ & $1$ &  &  & $8$ & $30$ & $0$ & $7$ & $0$ & $6$ & $0$ & $12$ & \tabularnewline
\cline{1-6} \cline{2-6} \cline{3-6} \cline{4-6} \cline{5-6} \cline{6-6} \cline{8-16} \cline{9-16} \cline{10-16} \cline{11-16} \cline{12-16} \cline{13-16} \cline{14-16} \cline{15-16} \cline{16-16} 
$9$ & $55$ & $55$ & $25$ & $7$ & $1$ &  & $9$ & $55$ & $30$ & $0$ & $14$ & $0$ & $14$ & $0$ & $30$\tabularnewline
\cline{1-6} \cline{2-6} \cline{3-6} \cline{4-6} \cline{5-6} \cline{6-6} \cline{8-16} \cline{9-16} \cline{10-16} \cline{11-16} \cline{12-16} \cline{13-16} \cline{14-16} \cline{15-16} \cline{16-16} 
\end{tabular}
\par\end{centering}
\caption{\label{tb-213-last}The numbers $j_{n,k}^{\protect\last}(213)$ and
$j_{n,k}^{\protect\last}(312)$ up to $n=9$.}
\end{table}
\end{thm}

See Table \ref{tb-213-last} for the numbers $j_{n,k}^{\last}(213)$
and $j_{n,k}^{\last}(312)$ up to $n=9$. Comparing the $j_{n,k}^{\last}(213)$
with the $j_{n,k}^{\lrmin}(213)=j_{n,k}^{\lrmin}(312)$ from Table
\ref{tb-213-312-lrmin} reveals that they are shifts of each other,
which we prove below.
\begin{prop}
\label{p-213-last-lrmin}For all $n\geq k\geq1$, we have
\[
j_{n,k}^{\last}(213)=\begin{cases}
{\displaystyle j_{n,2k}^{\lrmin}(213),} & \text{if }n\text{ is even,}\\
{\displaystyle j_{n,2k-1}^{\lrmin}(213),} & \text{if }n\text{ is odd.}
\end{cases}
\]
\end{prop}

\begin{proof}
By Corollary \ref{c-branchild} (a) and Proposition \ref{p-213-312-last}
(a), it suffices to construct a recursive bijection $\varphi\colon\mathcal{J}_{n}\rightarrow\mathcal{J}_{n}$
satisfying $\lbrch(T)=2\rbrch(\varphi(T))$ when $n$ is even and
$\lbrch(T)=2\rbrch(\varphi(T))-1$ when $n$ is odd.

Let us first define $\varphi$ for when $n$ is even. As the base
case, $\varphi$ sends the empty tree to the empty tree. Now, let
$T\in\mathcal{J}_{n}$ with $\lbrch(T)=2k$ for some $k\geq1$. Let
$u$ denote the bottommost node on the left branch of $T$, and let
$v$ denote its parent (the second bottommost node on that branch).
Moreover, let $U$ denote the right subtree of $u$ and $V$ the right
subtree of $v$; both have an even number of nodes. Remove the nodes
$u$ and $v$ (and their right subtrees) from $T$ to form the Jacobi
tree $T^{\prime}$, which has an even number of nodes because we have
deleted an even number of nodes from $T$. Then $\varphi(T^{\prime})$
is a Jacobi tree whose right branch has $k-1$ nodes. We construct
$\varphi(T)$ with the following steps:
\begin{enumerate}
\item Take the Jacobi tree $\varphi(T^{\prime})$ and add a new node $u^{\prime}$
to its right branch.
\item Add a new node $v^{\prime}$ as the left child of $u^{\prime}$.
\item Attach $U$ as the left subtree of $v^{\prime}$, and $V$ as the
right subtree of $v^{\prime}$.
\end{enumerate}
Note that $v^{\prime}$ along with its subtrees $U$ and $V$ form
the undergrowth of $\varphi(T)$ as defined in Section \ref{s-trees}.
See Figure \ref{f-lrminlastbij} for an illustration. 
\begin{figure}[!h]
\noindent \begin{centering}
\begin{center}
\begin{tikzpicture}[scale=0.4,auto,   thick,branch/.style={circle,fill=blue!20,draw,font=\sffamily,minimum size=1em},subtree/.style={circle,draw,dashed,font=\sffamily,minimum size=2em}]
%[every node/.style={circle,fill=blue!20}]

\begin{scope}[shift={(1.4,-0.3)},rotate=45]
\draw[dashed,red!50,fill=red!10] (10,0) ellipse (6.5 and 4.5);
\end{scope}

\begin{scope}[shift={(2.5,-14.5)},rotate=-45]
\draw[dashed,red!50,fill=red!10] (0,30) ellipse (6.5 and 4.5);
\end{scope}

\node (T') at (5.2,7.2) {$T'$};
\node (T'') at (27,7.2) {$\varphi(T')$};

\node at (16,2.2) {\large $\varphi$};
\node at (16,1) {\huge $\longmapsto$};

\node[branch] (1) at (10,10) {};
\node[branch] (2) at (8,8) {};
\node[branch] (3) at (5,5) {};
\node[branch] (4) at (1,1) {};
\node[branch] (5) at (-1,-1) {};
\node[subtree] (1') at (12,8) {};
\node[subtree] (2') at (10,6) {};
\node[subtree] (3') at (7,3) {};
\node[subtree] (4') at (3,-1) {$V$};
\node[subtree] (5') at (1,-3) {$U$};
\node at (5.75,-1) {\footnotesize (even)};
\node at (1,-5) {\footnotesize (even)};

\node[branch] (6) at (22,10) {};
\node[branch] (7) at (24,8) {};
\node[branch] (8) at (27,5) {};
\node[branch] (9) at (31,1) {};
\node[branch] (0) at (29,-1) {};
\node[subtree] (6') at (20,8) {};
\node[subtree] (7') at (22,6) {};
\node[subtree] (8') at (25,3) {};
\node[subtree] (9') at (27,-3) {$U$};
\node[subtree] (0') at (31,-3) {$V$};

\foreach \from/\to in {1/2,4/5,6/7,9/0}
\draw (\from) -- (\to);

\foreach \from/\to in {2/3,3/4,1/1',2/2',3/3',4/4',5/5',7/8,8/9,6/6',7/7',8/8',0/9',0/0'}
\draw[dashed] (\from) -- (\to);

\end{tikzpicture}
\end{center}
\par\end{centering}
\caption{\label{f-lrminlastbij}The bijection $\varphi$.}
\end{figure}
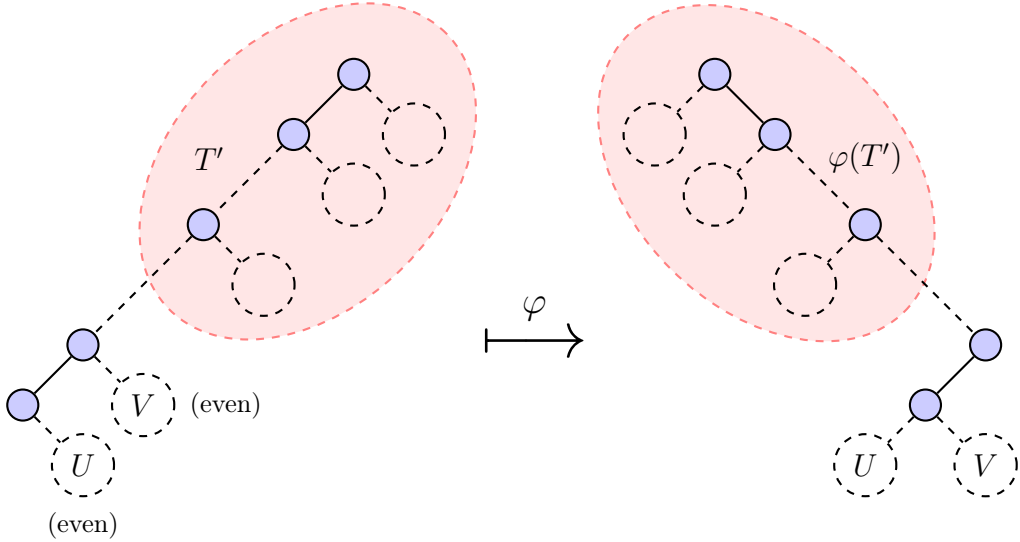

Given that $\varphi(T^{\prime})$ is a Jacobi tree and that $U$ and
$V$ each have an even number of nodes, it is evident that $\varphi(T)$
is a Jacobi tree as well. Furthermore, since the right branch of $\varphi(T^{\prime})$
has $k-1$ nodes, the right branch of $\varphi(T)$ has $k$ nodes,
so $\lbrch(T)=2\rbrch(\varphi(T))$ is satisfied. Finally, we see
from the undergrowth decomposition that every nonempty Jacobi tree
of even size is of the form depicted by the rightmost tree of Figure
\ref{f-lrminlastbij}, with both $U$ and $V$ of even size. Thus,
$\varphi$ is a bijection when $n$ is even.

Next, let us define $\varphi$ for when $n$ is odd. As the base case,
let $T$ be any tree in $\mathcal{J}_{n}$ whose left branch has a
single node\textemdash namely, the root of $T$. Then $T$ corresponds
to a Jacobi permutation $\pi$ that begins with the letter 1: $\pi=1\pi_{2}\pi_{3}\cdots\pi_{n}$.
We shall define $\varphi(T)$ so that it is the tree corresponding
to $\pi_{2}\pi_{3}\cdots\pi_{n}1$: remove the right subtree of the
root of $T$, and attach it back as the left subtree. Then $\varphi(T)$
is a Jacobi tree because $\pi_{2}\pi_{3}\cdots\pi_{n}1$ is a Jacobi
permutation, and $\varphi(T)$ has a single node on its right branch
as desired. For $T\in\mathcal{J}_{n}$ with $\lbrch(T)=2k-1$ where
$k\geq2$, we define $\varphi(T)$ in the same way as the even case.
Observe that $\varphi$ restricts to a bijection between trees in
${\cal J}_{n}$ whose left branch has a single node and those whose
right branch has a single node. Additionally, by similar reasoning
as the even case, the Jacobi trees whose right branch has at least
two nodes are all of the form of the rightmost tree of Figure \ref{f-lrminlastbij},
and it follows that $\varphi$ is a bijection when $n$ is odd.
\end{proof}
Theorem \ref{t-213last} is obtained by combining Theorem \ref{t-213-312-lrmin}
with Proposition \ref{p-213-last-lrmin}. Let us now move on to the
proof of Theorem \ref{t-312last}.
\begin{proof}[Proof of Theorem \ref{t-312last}]
Let 
\[
G_{\e}(x)\coloneqq\sum_{n=0}^{\infty}\left|\mathcal{J}_{2n}\right|x^{2n}\qquad\text{and}\qquad G_{\o}(x)\coloneqq\sum_{n=0}^{\infty}\left|\mathcal{J}_{2n+1}\right|x^{2n+1}
\]
be the ordinary generating functions for Jacobi trees of even and
odd size, respectively, and let $G(x)\coloneqq G_{\e}(x)+G_{\o}(x)$.
From Theorem \ref{t-jactree}, we have
\[
G_{\e}(x)=\sum_{n=0}^{\infty}\frac{1}{2n+1}{3n \choose n}x^{2n}\qquad\text{and}\qquad G_{\o}(x)=\sum_{n=0}^{\infty}\frac{1}{2n+1}{3n+1 \choose n+1}x^{2n+1}.
\]

To determine the numbers $j_{n,k}^{\last}(312)$, we shall derive
an expression for the generating function 
\[
J^{\last}(t,x;312)\coloneqq\sum_{n=1}^{\infty}\sum_{\pi\in\mathfrak{J}_{n}(312)}t^{\last(\pi)}x^{n}=\sum_{n=1}^{\infty}\sum_{k=1}^{n}j_{n,k}^{\last}(312)t^{k}x^{n}.
\]
By Proposition \ref{p-213-312-last} (b), we seek a weighted count
of Jacobi trees such that each node of the undergrowth contributes
a weight of $x$ and each other node contributes a weight of $tx$.
Using the undergrowth decomposition, we have {\allowdisplaybreaks
\begin{align*}
J^{\last}(t,x;312) & =txG(x)+tx(G(tx)-1)G_{\o}(x)\\
 & =txG_{\e}(x)+txG_{\o}(x)+txG(tx)G_{\o}(x)-txG_{\o}(x)\\
 & =txG_{\e}(x)+txG(tx)G_{\o}(x)\\
 & =\sum_{n=0}^{\infty}\frac{1}{2n+1}{3n \choose n}tx^{2n+1}+\sum_{n=0}^{\infty}\sum_{k=0}^{n}\frac{{3k \choose k}{3(n-k)+1 \choose n-k+1}}{(2k+1)(2(n-k)+1)}t^{2k+1}x^{2n+2}\\
 & \qquad\qquad+\sum_{n=0}^{\infty}\sum_{k=0}^{n}\frac{{3k+1 \choose k+1}{3(n-k)+1 \choose n-k+1}}{(2k+1)(2(n-k)+1)}t^{2k+2}x^{2n+3},
\end{align*}
}from which the desired result follows.
\end{proof}

\section{\label{s-231}\texorpdfstring{$231$}{231}-avoiding Jacobi permutations}

Through Jacobi trees, we saw that $\mathfrak{J}_{n}(213)$ and $\mathfrak{J}_{n}(312)$
are in bijection, and that the statistics $\des$ and $\lrmin$ each
has the same distribution over $\mathfrak{J}_{n}(213)$ as it does
over $\mathfrak{J}_{n}(312)$. In this section, we shall show that
$\mathfrak{J}_{n}(231)$ has the same enumeration and the same distributions
of $\des$ and $\lrmin$ as these two Jacobi avoidance classes, and
we will then use dual Jacobi trees to obtain a formula for counting
permutations in $\mathfrak{J}_{n}(231)$ with respect to the last
letter.

Inverses of permutations will play an essential role in this section,
so let us start by recalling some relevant properties of the inverse.

\subsection{\label{ss-arrayinv}Permutation arrays and inverses}

A permutation $\pi\in\mathfrak{S}_{n}$ is displayed diagrammatically
by plotting each of the points $(k,\pi_{k})$ in an $n$-by-$n$ grid,
where the first coordinate indicates the column and the second indicates
the row, with columns numbered 1 through $n$ from left to right,
and rows numbered 1 through $n$ from bottom to top. The resulting
diagram is called the \textit{array} of $\pi$.

Let $\pi^{-1}$ denote the inverse of $\pi\in\mathfrak{S}_{n}$. We
adopt the convention that the inverse of the empty permutation is
itself, and that if $\pi\in\mathfrak{S}_{S}$, then $\pi^{-1}$ is
the unique permutation of $S$ whose standardization is equal to $\std(\pi)^{-1}$.
For instance, we have $(72485)^{-1}=45827$. Inverses of permutations
in $\mathfrak{S}_{n}$ have a nice description in terms of their permutation
arrays: the inverse of $\pi\in\mathfrak{S}_{n}$ is the permutation
whose array is obtained by reflecting the array of $\pi$ about the
main diagonal. See Figure \ref{f-inverse} for an example.
\begin{figure}
\begin{center}
\begin{tikzpicture}[scale=0.6]       
\foreach \x in {0,1,...,6}             
{                 
\draw (-3, -3 + \x) -- (3, -3 + \x);             
}
\foreach \x in {0,1,...,6}             
{                 
\draw (-3 + \x, -3) -- (-3 + \x, 3);             
}         
\foreach \x in {1,...,6}             
{                 
\draw (-3.5 + \x, -3.5) node {\small \x};             
}                      
\foreach \x in {1,...,6}             
{                 
\draw (-3.5, -3.5 + \x) node {\small \x};             
}
        
\filldraw[black](-2.5, -1.5) circle [radius=3pt]                    
(-1.5, 1.5) circle [radius=3pt]                    
(-.5, -.5) circle [radius=3pt]                    
(.5, 2.5) circle [radius=3pt]                    
(1.5, .5) circle [radius=3pt]                    
(2.5, -2.5) circle [radius=3pt];         
\draw[dotted] (-3, -3) -- (3, 3);

\foreach \x in {0,1,...,6}             
{                 
\draw (8, -3 + \x) -- (14, -3 + \x);             
}
\foreach \x in {0,1,...,6}             
{                 
\draw (8 + \x, -3) -- (8 + \x, 3);             
}         
\foreach \x in {1,...,6}             
{                 
\draw (7.5 + \x, -3.5) node {\small \x};
}                      
\foreach \x in {1,...,6}             
{                 
\draw (7.5, -3.5 + \x) node {\small \x};
}
        
\filldraw[black](8.5, 2.5) circle [radius=3pt]                    
(9.5, -2.5) circle [radius=3pt]                    
(10.5, -.5) circle [radius=3pt]                    
(11.5, 1.5) circle [radius=3pt]                    
(12.5, -1.5) circle [radius=3pt]                    
(13.5, .5) circle [radius=3pt];    
\draw[dotted] (8, -3) -- (14, 3);
                 
\end{tikzpicture}
\end{center}
\vspace{-10bp}

\caption{\label{f-inverse}The arrays of $\pi=253641$ (left) and its inverse
$\pi^{-1}=613524$ (right).}
\end{figure}
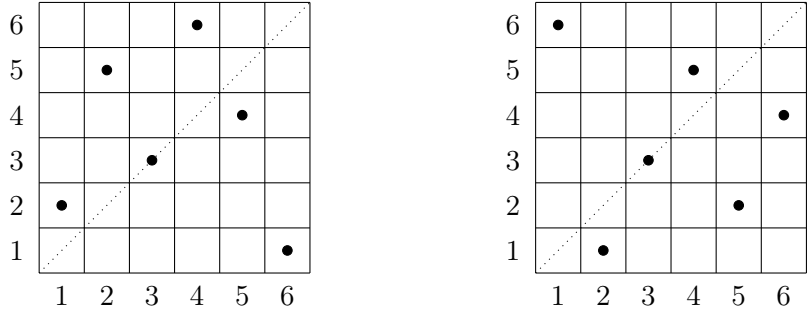

We will need a couple of facts about the inverses of certain kinds
of permutations, which are most conveniently expressed using the terminology
of direct and skew sums. Given $\alpha\in\mathfrak{S}_{m}$ and $\beta\in\mathfrak{S}_{n}$,
the \textit{direct sum} $\alpha\oplus\beta$ and the \textit{skew
sum} $\alpha\ominus\beta$ are defined to be the permutations in $\mathfrak{S}_{m+n}$
whose letters are given by
\[
(\alpha\oplus\beta)_{k}\coloneqq\begin{cases}
\alpha_{k}, & \text{if }1\leq k\leq m,\\
\beta_{k-m}+m, & \text{if }m+1\leq k\leq m+n,
\end{cases}
\]
and 
\[
(\alpha\ominus\beta)_{k}\coloneqq\begin{cases}
\alpha_{k}+n, & \text{if }1\leq k\leq m,\\
\beta_{k-m}, & \text{if }m+1\leq k\leq m+n.
\end{cases}
\]
For example, if $\alpha=132$ and $\beta=21$, then $\alpha\oplus\beta=13254$
and $\alpha\ominus\beta=35421$. We note that $\alpha\oplus\beta$
can more intuitively be described as the permutation whose array is
obtained by juxtaposing the arrays of $\alpha$ and $\beta$ diagonally,
and similarly with $\alpha\ominus\beta$ except that they are juxtaposed
antidiagonally.

The next two lemmas, well known in the permutation patterns literature,
are best explained using permutation arrays. We provide a visual depiction
of these lemmas in Figure \ref{f-directskew} in lieu of proofs.
\begin{lem}
\label{l-invconcat1}If $\pi=\alpha\oplus\beta$, then $\pi^{-1}=\alpha^{-1}\oplus\beta^{-1}$.
\end{lem}

\begin{lem}
\label{l-invconcat2}If $\pi=\alpha\ominus\beta$, then $\pi^{-1}=\beta^{-1}\ominus\alpha^{-1}$.
\end{lem}

\begin{figure}
\begin{center}
\begin{tikzpicture}[scale=0.6]       

\node at (-4.55,0) {$\alpha \oplus \beta=$};
\draw (-3,-3) -- (3,-3);
\draw (-3,3) -- (3,3);
\draw (-3,-3) -- (-3,3);
\draw (3,-3) -- (3,3);
\draw[color=gray] (-1,-3) -- (-1,3);
\draw[color=gray] (-3,-1) -- (3,-1);              
\draw[dotted] (-3, -3) -- (3, 3);
\node at (-2,-2) {$\alpha$};
\node at (1,1) {$\beta$};

\node at (6.85,0) {$(\alpha \oplus \beta)^{-1}=$};
\draw (9,-3) -- (15,-3);
\draw (9,3) -- (15,3);
\draw (9,-3) -- (9,3);
\draw (15,-3) -- (15,3);
\draw[color=gray] (11,-3) -- (11,3);
\draw[color=gray] (9,-1) -- (15,-1);
\draw[dotted] (9, -3) -- (15, 3);
\node at (10,-2) {$\alpha^{-1}$};
\node at (13,1) {$\beta^{-1}$};

\node at (-4.55,-8) {$\alpha \ominus \beta=$};
\draw (-3,-11) -- (3,-11);
\draw (-3,-5) -- (3,-5);
\draw (-3,-11) -- (-3,-5);
\draw (3,-11) -- (3,-5);
\draw[color=gray] (-1,-11) -- (-1,-5);
\draw[color=gray] (-3,-7) -- (3,-7);              
\draw[dotted] (-3,-11) -- (3,-5);
\node at (-2,-6) {$\alpha$};
\node at (1,-9) {$\beta$};

\node at (6.85,-8) {$(\alpha \ominus \beta)^{-1}=$};
\draw (9,-11) -- (15,-11);
\draw (9,-5) -- (15,-5);
\draw (9,-11) -- (9,-5);
\draw (15,-11) -- (15,-5);
\draw[color=gray] (13,-11) -- (13,-5);
\draw[color=gray] (9,-9) -- (15,-9);
\draw[dotted] (9,-11) -- (15,-5);
\node at (11,-7) {$\beta^{-1}$};
\node at (14,-10) {$\alpha^{-1}$};

\end{tikzpicture}
\end{center}

\caption{\label{f-directskew}A direct sum, a skew sum, and their inverses.}
\end{figure}
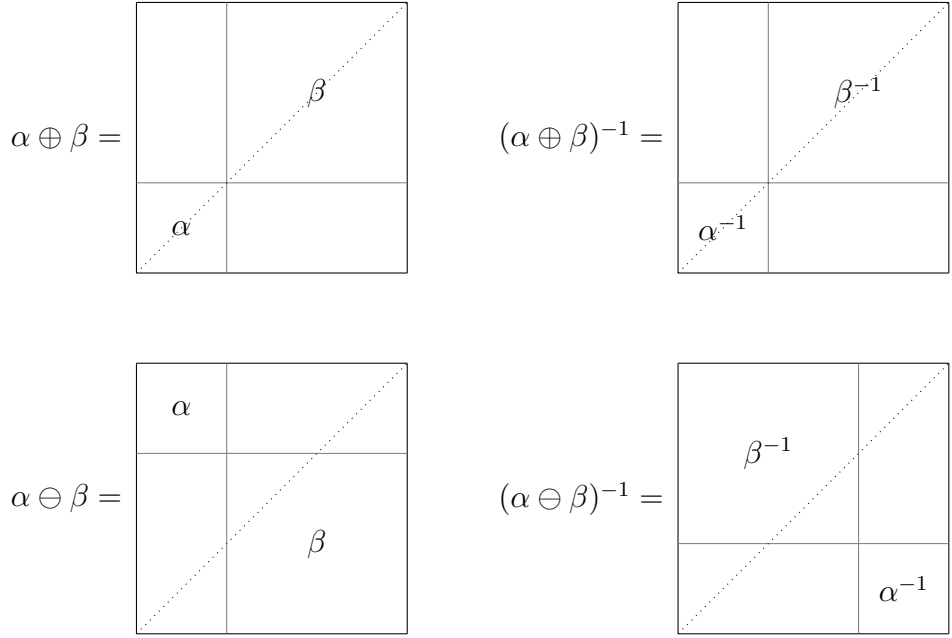

The following two facts, also readily seen through permutation arrays,
will be useful later.
\begin{lem}
\label{l-lrminv}A letter $\pi_{k}$ is a left-to-right minimum of
$\pi\in\mathfrak{S}_{n}$ if and only if $k$ is a left-to-right minimum
of $\pi^{-1}$.
\end{lem}

\begin{lem}
\label{l-pavinv}A permutation $\pi$ avoids $\sigma$ if and only
if $\pi^{-1}$ avoids $\sigma^{-1}$.
\end{lem}

Lastly, when a direct sum $\alpha\oplus\beta$ is Jacobi, the next
lemma provides some information about its summands $\alpha$ and $\beta$.
\begin{lem}
\label{l-prefixJacobi}Suppose that $\pi=\alpha\oplus\beta$ is Jacobi.
\begin{enumerate}
\item [\normalfont{(a)}]If $\alpha$ is nonempty, then $\beta$ has even
length.
\item [\normalfont{(b)}]The permutation $\alpha$ is Jacobi.
\end{enumerate}
\end{lem}

\begin{proof}
Let $\beta^{\prime}$ denote the permutation obtained from $\beta$
by adding $\left|\alpha\right|$ to every letter, so that $\pi$ is
the concatenation $\pi=\alpha\beta^{\prime}$.

Note that (b) is trivial if $\alpha$ is empty, so assume that $\alpha$
is nonempty for both (a) and (b). Let $y$ be the last letter of $\alpha$.
Since $y$ is smaller than every letter of $\beta^{\prime}$, we have
$\rho_{\pi}(y)=\beta^{\prime}$. Then $\rho_{\pi}(y)=\beta^{\prime}$
has even length because $\pi$ is Jacobi, and we have shown (a).

Now, let $x$ be an arbitrary letter of $\alpha$ other than the last
letter $y$. Suppose that $y$ is in $\rho_{\alpha}(x)$. Because
every letter of $\beta^{\prime}$ is larger than $x$, we know that
$\rho_{\pi}(x)$ is the concatenation $\rho_{\pi}(x)=\rho_{\alpha}(x)\beta^{\prime}.$
Since $\rho_{\pi}(x)$ and $\beta^{\prime}$ both have even length,
it follows that $\rho_{\alpha}(x)$ has even length as well. In the
case that $\rho_{\alpha}(x)$ does not contain $y$, we have $\rho_{\alpha}(x)=\rho_{\pi}(x)$,
which is of even length. Thus, (b) is proven.
\end{proof}

\subsection{\label{ss-213-231-312-doubly}\texorpdfstring{$213$}{213}-, \texorpdfstring{$231$}{231}-,
and \texorpdfstring{$312$}{312}-avoiding Jacobi permutations are
doubly Jacobi}

Because $312$ and $231$ are inverses, we know from Lemma \ref{l-pavinv}
that $\pi\in\mathfrak{S}_{n}$ avoids $312$ if and only if $\pi^{-1}$
avoids $231$. Perhaps surprisingly, the inverse preserves the Jacobi
property on these two avoidance classes.
\begin{thm}
\label{t-312-231-inv}For all $n\geq0$, a permutation $\pi\in\mathfrak{S}_{n}(312)$
is Jacobi if and only if $\pi^{-1}\in\mathfrak{S}_{n}(231)$ is Jacobi.
Consequently, $\pi\mapsto\pi^{-1}$ is a bijection from $\mathfrak{J}_{n}(312)$
to $\mathfrak{J}_{n}(231)$.
\end{thm}

Throughout this proof and the next one, we adopt the notation $\bar{\pi}$
for the standardization $\std(\pi)$.
\begin{proof}
We proceed by induction, with the base case ($n=0$) being true by
convention. Suppose that the result holds for all lengths less than
$n$.

Let $\pi\in\mathfrak{J}_{n}(312)$. We write $\pi=\alpha1\beta$,
so that $\alpha$ and $\beta$ are $312$-avoiding Jacobi permutations,
$\beta$ is of even length, and every letter of $\alpha$ is smaller
than every letter of $\beta$. In fact, we can write $\pi$ as the
direct sum $\pi=(\alpha1)\oplus\bar{\beta}$, where both summands
$\alpha1$ and $\bar{\beta}$ are Jacobi. Let us consider two cases:
\begin{itemize}
\item Suppose that $\bar{\beta}$ is nonempty, so that the lengths of $\alpha1$
and $\bar{\beta}$ are both less than $n$. By the induction hypothesis,
both $(\alpha1)^{-1}$ and $\bar{\beta}^{-1}$ are Jacobi. Recall
that $\pi^{-1}=(\alpha1)^{-1}\oplus\bar{\beta}^{-1}=(\alpha1)^{-1}\beta^{-1}$
by Lemma \ref{l-invconcat1}. Taking $y$ to be the last letter of
$(\alpha1)^{-1}$, we see from Lemma~\ref{l-rlminJacobi} that $\pi^{-1}$
is Jacobi.
\item If $\beta$ is empty, then $\pi=\bar{\alpha}\ominus1$, so we have
$\pi^{-1}=1\ominus\bar{\alpha}^{-1}=n\bar{\alpha}^{-1}$ by Lemma
\ref{l-invconcat2}. Moreover, because $\alpha1$ is Jacobi, $\alpha$
and thus $\bar{\alpha}$ are Jacobi as well. The induction hypothesis
then tells us that $\bar{\alpha}^{-1}$ is Jacobi, so $\pi^{-1}=n\bar{\alpha}^{-1}$
is Jacobi by Lemma \ref{l-Jacobi} (b).
\end{itemize}
\indent Now, let $\pi\in\mathfrak{J}_{n}(231)$. Writing $\pi=\alpha n\beta$,
we have that $\alpha$ and $\beta$ are $231$-avoiding permutations
with every letter of $\alpha$ being smaller than every letter of
$\beta$. Thus, we can write $\pi=\alpha\oplus\overline{n\beta}$.
By Lemmas \ref{l-Jacobi} (a) and \ref{l-prefixJacobi} (b), $\overline{n\beta}$
and $\alpha$ are Jacobi. Dividing into cases again:
\begin{itemize}
\item If $\alpha$ is nonempty, then $\alpha$ and $\overline{n\beta}$
have lengths less than $n$, so the induction hypothesis tells us
that $\alpha^{-1}$ and $(\overline{n\beta})^{-1}$ are Jacobi. Moreover,
$n\beta$\textemdash and consequently $(n\beta)^{-1}$\textemdash have
even length according to Lemma \ref{l-prefixJacobi} (a). Similar
to before, we have $\pi^{-1}=\alpha^{-1}\oplus(\overline{n\beta})^{-1}=\alpha^{-1}(n\beta)^{-1}$
by Lemma \ref{l-invconcat1}, and taking $y$ to be the last letter
of $\alpha^{-1}$, Lemma~\ref{l-rlminJacobi} implies that $\pi^{-1}=\alpha^{-1}(n\beta)^{-1}$
is Jacobi.
\item If $\alpha$ is empty, then we have $\pi=n\beta=1\ominus\beta$, so
$\pi^{-1}=\beta^{-1}\ominus1=\beta^{-1}1$ by Lemma \ref{l-invconcat2}.
We know that $\beta$ is Jacobi because of Lemma \ref{l-Jacobi} (a),
so from the induction hypothesis, $\beta^{-1}$ is Jacobi as well.
Therefore, $\pi^{-1}=\beta^{-1}1$ is Jacobi.\qedhere
\end{itemize}
\end{proof}
Recall that Jacobi trees are the unlabeled trees corresponding to
the increasing binary trees of $312$-avoiding Jacobi permutations,
and that dual Jacobi trees are those corresponding to the decreasing
binary trees of $231$-avoiding Jacobi permutations. It is not obvious
from their definitions that Jacobi and dual Jacobi trees are in bijection,
but Theorem \ref{t-312-231-inv} shows that the inverse induces a
bijection between $\mathcal{J}_{n}$ and $\tilde{\mathcal{J}}_{n}$,
justifying Theorem \ref{t-dualjactree}.

In analogy to the notion of doubly alternating permutations\textemdash studied,
e.g., in \cite{Gessel2024,Guibert2000,Ouchterlony2005,Stanley2007}\textemdash let
us call a permutation $\pi$ \textit{doubly Jacobi} if both $\pi$
and $\pi^{-1}$ are Jacobi. Then Theorem~\ref{t-312-231-inv} implies
that all permutations in $\mathfrak{J}_{n}(312)$ and $\mathfrak{J}_{n}(231)$
are doubly Jacobi, and the next result tells us that the same is true
for $\mathfrak{J}_{n}(213)$.
\begin{thm}
\label{t-213-inv}For all $n\geq0$, the permutation $\pi^{-1}$ is
Jacobi whenever $\pi\in\mathfrak{J}_{n}(213)$.
\end{thm}

\begin{proof}
We follow the same inductive approach as the proof of Theorem \ref{t-312-231-inv}.
Again, the base case is trivial, so assume that the result holds for
all lengths less than $n$. Let $\pi\in\mathfrak{J}_{n}(213)$, and
write $\pi=\alpha1\beta$ where $\alpha$ and $\beta$ are $213$-avoiding
Jacobi permutations, $\beta$ is of even length, and every letter
of $\alpha$ is larger than every letter of $\beta$. In other words,
$\pi$ is the skew sum $\pi=\bar{\alpha}\ominus(1\beta)$. Note that
$1\beta$ is Jacobi because $\beta$ is Jacobi and has even length.
\begin{itemize}
\item If $\alpha$ is nonempty, then both $\bar{\alpha}$ and $1\beta$
have length less than $n$, so $\bar{\alpha}^{-1}$ and $(1\beta)^{-1}$
are Jacobi by the induction hypothesis. Because the first letter of
$\bar{\alpha}^{-1}$ corresponds to a left-to-right minimum of $\pi^{-1}=(1\beta)^{-1}\ominus\bar{\alpha}^{-1}$,
it follows from Lemma \ref{l-lrminJacobi} that $\pi^{-1}$ is Jacobi.
\item If $\alpha$ is empty, then $\pi^{-1}=1\oplus\bar{\beta}^{-1}=1\beta^{-1}$.
Observe that $\beta^{-1}$ is Jacobi due to the induction hypothesis,
and because $\beta^{-1}$ is of even length, $\pi^{-1}=1\beta^{-1}$
is Jacobi.\qedhere
\end{itemize}
\end{proof}

\subsection{Enumeration, descents, and left-to-right minima}

The fact that $\mathfrak{J}_{n}(312)$ and $\mathfrak{J}_{n}(231)$
are in bijection allows several of our results for $\mathfrak{J}_{n}(312)$
to directly translate over to $\mathfrak{J}_{n}(231)$.
\begin{thm}
\label{t-231}For all $n\geq0$, we have
\[
j_{2n}(231)=\frac{1}{2n+1}{3n \choose n}\quad\text{and}\quad j_{2n+1}(231)=\frac{1}{2n+1}{3n+1 \choose n+1}.
\]
\end{thm}

\begin{proof}
This follows immediately from Lemma \ref{l-pavinv} along with Theorems
\ref{t-213-312} and \ref{t-312-231-inv}.
\end{proof}
\begin{thm}
\label{t-231-des}We have
\begin{align*}
j_{2n,k}^{\des}(231) & =\frac{1}{n}{n \choose k-n}{2n \choose k+1} & \text{for all }n\geq1\text{ and }k\geq0;\\
j_{2n+1,k}^{\des}(231) & =\frac{1}{n+1}{n+1 \choose k-n}{2n \choose k} & \text{for all }n,k\geq0.
\end{align*}
\end{thm}

\begin{proof}
It is known that $\des(\pi)=\des(\pi^{-1})$ for any $\pi\in\mathfrak{S}_{n}(312)$;
see, e.g., \cite[Corollary 3.2]{Stump2008/09}. The desired result
follows from this fact, Lemma \ref{l-pavinv}, and Theorems \ref{t-213-312-des}
and \ref{t-312-231-inv}.
\end{proof}
\begin{thm}
\label{t-231-lrmin}For all $n\geq k\geq1$, we have 
\[
j_{n,k}^{\lrmin}(231)=\frac{2k}{3n-k}{\frac{3n-k}{2} \choose n}
\]
if $n$ and $k$ have the same parity, and $j_{n,k}^{\lrmin}(231)=0$
otherwise.
\end{thm}

\begin{proof}
Lemma \ref{l-lrminv} implies that $\lrmin(\pi)=\lrmin(\pi^{-1})$
for all $\pi\in\mathfrak{S}_{n}$. The desired result follows from
this fact, Lemma \ref{l-pavinv}, and Theorems \ref{t-213-312-lrmin}
and \ref{t-312-231-inv}.
\end{proof}

\subsection{Last letter}

Unlike $\des$ and $\lrmin$, the distribution of $\last$ over $\mathfrak{J}_{n}(231)$
does not translate nicely from its distribution over $\mathfrak{J}_{n}(213)$
or $\mathfrak{J}_{n}(312)$. Instead, we will prove the next result
using dual Jacobi trees, which we recall are in bijection with $231$-avoiding
Jacobi permutations.
\begin{thm}
\label{t-231-last}For all $n,k\geq1$, we have 
\begin{align*}
j_{2n,k}^{\last}(231) & =\sum_{m=0}^{n-1}\frac{n+k-3m-1}{n+k-1}{n-m-1 \choose k-2m-1}{n+k-1 \choose m}\quad\text{and}\\
j_{2n+1,k}^{\last}(231) & =\sum_{m=0}^{n-1}\frac{n+k-3m-1}{n+k-1}{n-m \choose k-2m-1}{n+k-1 \choose m}.
\end{align*}
\begin{table}
\begin{centering}
\renewcommand{\arraystretch}{1.1}%
\begin{tabular}{|>{\centering}p{20bp}|>{\centering}p{20bp}|>{\centering}p{20bp}|>{\centering}p{20bp}|>{\centering}p{20bp}|>{\centering}p{20bp}|>{\centering}p{20bp}|>{\centering}p{20bp}|>{\centering}p{20bp}|}
\hline 
$n\backslash k$ & $1$ & $2$ & $3$ & $4$ & $5$ & $6$ & $7$ & $8$\tabularnewline
\hline 
$1$ & $1$ &  &  &  &  &  &  & \tabularnewline
\hline 
$2$ & $1$ &  &  &  &  &  &  & \tabularnewline
\hline 
$3$ & $1$ & $1$ &  &  &  &  &  & \tabularnewline
\hline 
$4$ & $1$ & $1$ & $1$ &  &  &  &  & \tabularnewline
\hline 
$5$ & $1$ & $2$ & $2$ & $2$ &  &  &  & \tabularnewline
\hline 
$6$ & $1$ & $2$ & $3$ & $3$ & $3$ &  &  & \tabularnewline
\hline 
$7$ & $1$ & $3$ & $5$ & $7$ & $7$ & $7$ &  & \tabularnewline
\hline 
$8$ & $1$ & $3$ & $6$ & $9$ & $12$ & $12$ & $12$ & \tabularnewline
\hline 
$9$ & $1$ & $4$ & $9$ & $16$ & $23$ & $30$ & $30$ & $30$\tabularnewline
\hline 
\end{tabular}
\par\end{centering}
\caption{\label{tb-231-last}The numbers $j_{n,k}^{\protect\last}(231)$ up
to $n=9$.}
\end{table}
\end{thm}

Our first step in proving Theorem \ref{t-231-last} is to establish
the following expression for the generating function
\[
J^{\last}(t,x;231)\coloneqq\sum_{n=1}^{\infty}\sum_{\pi\in\mathfrak{J}_{n}(231)}t^{\last(\pi)}x^{n}=\sum_{n=1}^{\infty}\sum_{k=1}^{n}j_{n,k}^{\last}(231)t^{k}x^{n}
\]
in terms of the generating function for even-sized dual Jacobi trees.
\begin{prop}
\label{p-231-lastgf}We have 
\begin{align*}
J^{\last}(t,x;231) & =\frac{tx(1+xG_{\e}(tx))}{1-x^{2}G_{\e}(tx)\left(1+tG_{\e}(tx)\right)}
\end{align*}
where $G_{\e}(x)=\sum_{n=0}^{\infty}\frac{1}{2n+1}{3n \choose n}x^{2n}$
is the generating function for dual Jacobi trees of even size \textup{(}equivalently,
Jacobi trees of even size, or 231-avoiding Jacobi permutations of
even length\textup{)}.
\end{prop}

\begin{proof}
Let 
\[
G_{\e}(x)\coloneqq\sum_{n=0}^{\infty}\vert\mathcal{\tilde{J}}_{2n}\vert x^{2n}\qquad\text{and}\qquad G_{\o}(x)\coloneqq\sum_{n=0}^{\infty}\vert\mathcal{\tilde{J}}_{2n+1}\vert x^{2n+1}
\]
be the generating functions for even- and odd-sized dual Jacobi trees,
respectively. By Proposition~\ref{p-231-last}, we have 
\begin{align*}
J^{\last}(t,x;231) & =\sum_{n=1}^{\infty}\sum_{T\in\tilde{\mathcal{J}}_{n}}t^{n-\rbrch(T)+1}x^{n}=\sum_{n=1}^{\infty}\sum_{T\in\tilde{\mathcal{J}}_{n}}x^{\rbrch(T)-1}(tx)^{n-\rbrch(T)+1};
\end{align*}
we can interpret this as a weighted count of dual Jacobi trees where
each node on the right branch, except the root, contributes $x$ to
the weight and all other nodes each contributes $tx$ to the weight.
From the recursive decomposition for dual Jacobi trees (see Figure
\ref{f-dualdecomp}), we obtain 
\[
J^{\last}(t,x;231)=tx+tx^{2}G_{\e}(tx)+xG_{\o}(tx)J^{\last}(t,x;231)+x^{2}G_{\e}(tx)J^{\last}(t,x;231),
\]
whence
\begin{align}
J^{\last}(t,x;231) & =\frac{tx+tx^{2}G_{\e}(tx)}{1-xG_{\o}(tx)-x^{2}G_{\e}(tx)}.\label{e-231last}
\end{align}
Since dual Jacobi trees are in bijection with (ordinary) Jacobi trees,
let us now interpret $G_{\text{e}}(x)$ and $G_{\text{\ensuremath{\o}}}(x)$
as being generating functions for the latter. Then $G_{\o}(x)=xG_{\text{e}}(x)^{2}$,
because every Jacobi tree of odd size consists of a root whose subtrees
are both of even size. Substituting $G_{\o}(tx)=txG_{\text{e}}(tx)^{2}$
into (\ref{e-231last}) yields the desired formula.
\end{proof}
We prove Theorem \ref{t-231-last} from Proposition \ref{p-231-lastgf}
by extracting coefficients from the generating function $J^{\last}(t,x;231)$.
\begin{proof}[Proof of Theorem \ref{t-231-last}]
We first expand the expression for $J^{\last}(t,x;231)$ given in
Proposition~\ref{p-231-lastgf}, with the help of the binomial theorem,
to obtain
\begin{align}
J^{\last}(t,x;231) & =\frac{tx(1+xG_{\e}(tx))}{1-x^{2}G_{\e}(tx)\left(1+tG_{\e}(tx)\right)}\nonumber \\
 & =tx(1+xG_{\e}(tx))\sum_{i=0}^{\infty}x^{2i}G_{\e}(tx)^{i}\left(1+tG_{\e}(tx)\right)^{i}\nonumber \\
 & =tx(1+xG_{\e}(tx))\sum_{i,l=0}^{\infty}{i \choose l}G_{\e}(tx)^{i+l}t^{l}x^{2i}\nonumber \\
 & =\sum_{i,l=0}^{\infty}{i \choose l}G_{\e}(tx)^{i+l+1}t^{l+1}x^{2i+2}+\sum_{i,l=0}^{\infty}{i \choose l}G_{\e}(tx)^{i+l}t^{l+1}x^{2i+1}.\label{e-J231lastgf1}
\end{align}
It is well known that the $k$-fold convolution of the sequence $\left\{ \frac{1}{2n+1}{3n \choose n}\right\} _{n\geq0}$
is given by $\left\{ \frac{k}{3n+k}{3n+k \choose n}\right\} _{n\geq0}$;
see, e.g., \cite[Section 3.3]{Gessel2016}. That is, we have
\[
G_{\e}(tx)^{k}=\sum_{m=0}^{\infty}\frac{k}{3m+k}{3m+k \choose m}t^{2m}x^{2m}
\]
for all $k\geq1$. Substituting this expression, for the appropriate
values of $k$, into (\ref{t-231-last}) yields
\begin{align}
J^{\last}(t,x;231) & =\sum_{i,l,m=0}^{\infty}{i \choose l}\frac{i+l+1}{3m+i+l+1}{3m+i+l+1 \choose m}t^{l+2m+1}x^{2i+2m+2}\nonumber \\
 & \qquad\qquad+tx+\sum_{\substack{i,l,m=0\\
i+l\neq0
}
}^{\infty}{i \choose l}\frac{i+l}{3m+i+l}{3m+i+l \choose m}t^{l+2m+1}x^{2i+2m+1}.\label{e-J231lastgf2}
\end{align}
We take the right-hand side of (\ref{e-J231lastgf2}), make the substitutions
$n=i+m+1$ and $k=l+2m+1$ in the first sum, and substitute $n=i+m$
and $k=l+2m+1$ in the second sum to get
\begin{align*}
J^{\last}(t,x;231) & =\sum_{n,k=1}^{\infty}\sum_{m=0}^{n-1}\frac{n+k-3m-1}{n+k-1}{n-m-1 \choose k-2m-1}{n+k-1 \choose m}t^{k}x^{2n}\\
 & \qquad\quad+tx+\sum_{n,k=1}^{\infty}\sum_{m=0}^{n-1}\frac{n+k-3m-1}{n+k-1}{n-m \choose k-2m-1}{n+k-1 \choose m}t^{k}x^{2n+1}.
\end{align*}
Extracting coefficients completes the proof.
\end{proof}

\section{\label{s-123}\texorpdfstring{$123$}{123}-avoiding Jacobi permutations}

Next, we study $123$-avoiding Jacobi permutations, whose enumeration
is given by the following.
\begin{thm}
\label{t-123}For all $n\geq1$, we have
\[
j_{n}(123)=\sum_{k=0}^{\left\lfloor (n-1)/2\right\rfloor }\frac{1}{2k+1}{n-k-1 \choose k}{n \choose 2k}.
\]
\end{thm}

These numbers appear in \cite[A101785]{oeis}, where it is stated
that the entries also count Dyck paths of semilength $n$ whose descents
are all of odd length. We will prove Theorem~\ref{t-123} by showing
that a certain bijection between $123$-avoiding permutations and
Dyck paths restricts to a bijection between $123$-avoiding Jacobi
permutations and Dyck paths with all descents odd. Furthermore, we
will determine the joint distribution of left-to-right minima and
ascents over $\mathfrak{J}_{n}(123)$ by counting these Dyck paths
with respect to certain subword statistics.
\begin{table}[H]
\begin{centering}
\renewcommand{\arraystretch}{1.1}%
\begin{tabular}{c|>{\centering}p{22bp}|>{\centering}p{22bp}|>{\centering}p{22bp}|>{\centering}p{22bp}|>{\centering}p{22bp}|>{\centering}p{22bp}|>{\centering}p{22bp}|>{\centering}p{22bp}|>{\centering}p{22bp}|>{\centering}p{22bp}|>{\centering}p{23bp}|>{\centering}p{23bp}}
$n$ & $0$ & $1$ & $2$ & $3$ & $4$ & $5$ & $6$ & $7$ & $8$ & $9$ & $10$ & $11$\tabularnewline
\hline 
$j_{n}(123)$ & $1$ & $1$ & $1$ & $2$ & $5$ & $12$ & $30$ & $79$ & $213$ & $584$ & $1628$ & $4600$\tabularnewline
\end{tabular}
\par\end{centering}
\caption{\label{tb-123}The numbers $j_{n}(123)$ up to $n=11$.}
\end{table}

\subsection{Dyck paths and \texorpdfstring{$123$}{123}-avoiding permutations}

A \textit{Dyck path} of \textit{semilength} $n$ is a lattice path
in $\mathbb{Z}^{2}$\textemdash consisting of $2n$ steps from the
step set $\{(1,1),(1,-1)\}$\textemdash which starts at the origin
$(0,0)$, ends at $(2n,0)$, and never traverses below the $x$-axis.
The steps $(1,1)$ are denoted by $U$ and called \textit{up steps},
whereas the $(1,-1)$ are denoted by $D$ and called \textit{down
steps}. Thus, Dyck paths can be represented as words on the alphabet
$\{U,D\}$ subject to certain restrictions.

Let $\mathcal{D}_{n}$ denote the set of Dyck paths of semilength
$n$. We shall define a bijection from $\mathfrak{S}_{n}(123)$ to
$\mathcal{D}_{n}$ for which it will be convenient to illustrate Dyck
paths as instead having steps from the step set $\{(0,-1),(1,0)\}$,
beginning and ending on the antidiagonal line $y=-x$, and never traversing
above this line. We can identify these paths with the Dyck paths defined
earlier by swapping the ``south steps'' $(0,-1)$ with up steps, and
the ``east steps'' $(1,0)$ with down steps. We will draw Dyck paths
in this latter way (with south and east steps), while still thinking
of them using the original definition (with up and down steps).

Define $\chi\colon\mathfrak{S}_{n}(123)\rightarrow\mathcal{D}_{n}$
in the following way. Given $\pi\in\mathfrak{S}_{n}(123)$, we first
draw the array of $\pi$ (as defined in Section \ref{ss-arrayinv}).
Then, $\chi(\pi)$ is the Dyck path drawn from the upper-left corner
to the lower-right corner of the array which leaves all of the points
$(k,\pi_{k})$ to its right but remains as close as possible to the
antidiagonal. Equivalently, if $\pi_{i_{1}}>\pi_{i_{2}}>\cdots>\pi_{i_{m}}$
are the left-to-right minima of $\pi$, then 
\begin{equation}
\chi(\pi)=U^{n+1-\pi_{i_{1}}}D^{i_{2}-i_{1}}U^{\pi_{i_{1}}-\pi_{i_{2}}}D^{i_{3}-i_{2}}\cdots U^{\pi_{i_{m-1}}-\pi_{i_{m}}}D^{n+1-i_{m}}.\label{e-chidef}
\end{equation}
See Figure \ref{f-kratbij} for an example.
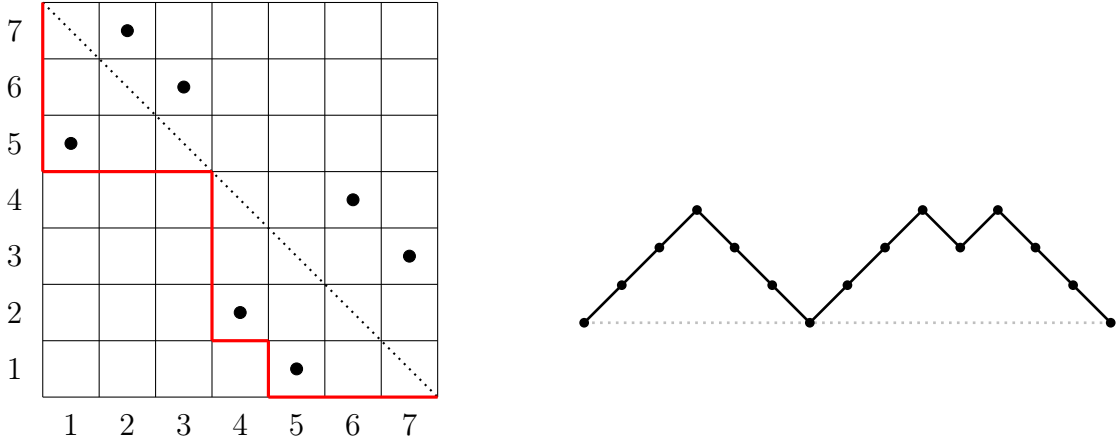
\begin{figure}
\begin{center}
\begin{tikzpicture}[scale=0.75]         
\foreach \x in {0,1,...,7}             
{                 
\draw (-3, -3 + \x) -- (4, -3 + \x);             
}
\foreach \x in {0,1,...,7}             
{                 
\draw (-3 + \x, -3) -- (-3 + \x, 4);             
}         
\foreach \x in {1,...,7}             
{                 
\draw (-3.5 + \x, -3.5) node {\x};             
}                      
\foreach \x in {1,...,7}             
{                 
\draw (-3.5, -3.5 + \x) node {\x};             
}
        
\filldraw[black](-2.5, 1.5) circle [radius=3pt]                    
(-1.5, 3.5) circle [radius=3pt]                    
(-.5, 2.5) circle [radius=3pt]                    
(.5, -1.5) circle [radius=3pt]                    
(1.5, -2.5) circle [radius=3pt]                    
(2.5, .5) circle [radius=3pt]         
(3.5, -.5) circle [radius=3pt];
\draw[dotted, thick] (-3, 4) -- (4, -3);
        
\draw[red, very thick] (-3, 4) -- (-3, 1);
\draw[red, very thick] (-3, 1) -- (0, 1);
\draw[red, very thick] (0, 1) -- (0, -2);
\draw[red, very thick] (0, -2) -- (1, -2);            
\draw[red, very thick] (1, -2) -- (1, -3);
\draw[red, very thick] (1, -3) -- (4, -3);               

\end{tikzpicture} \qquad\qquad
\begin{tikzpicture}[scale=0.5] 

\draw[pathcolorlight] (0,0) -- (14,0); 
\drawlinedots[pathdefault]{0,1,2,3,4,5,6,7,8,9,10,11,12,13,14}{0,1,2,3,2,1,0,1,2,3,2,3,2,1,0}

\node at (0,-3) {};

\end{tikzpicture}

\end{center}
\vspace{-15bp}\caption{\label{f-kratbij}The Dyck path $\chi(5762143)=UUUDDDUUUDUDDD$, drawn
in both ways.}
\end{figure}

A \textit{factor} of a Dyck path refers to a factor (i.e., a consecutive
subword) of the corresponding word. Given a Dyck path $\mu$ and a
word $\alpha$ on the alphabet $\{U,D\}$, let $\occ_{\alpha}(\mu)$
denote the number of $\alpha$-factors in $\mu$, that is, occurrences
of $\alpha$ as a factor of $\mu$. For example, if $\mu$ is the
Dyck path in Figure \ref{f-kratbij}, then $\occ_{UD}(\mu)=3$ and
$\occ_{UDD}(\mu)=2$.

To see why $\chi$ is a bijection, first note that every $123$-avoiding
permutation is a shuffle\footnote{Given permutations $\alpha$ and $\beta$ on disjoint sets of letters,
we say that $\pi$ is a \textit{shuffle} of $\alpha$ and $\beta$
if $\left|\pi\right|=\left|\alpha\right|+\left|\beta\right|$ and
both $\alpha$ and $\beta$ are subwords of $\pi$.} of two decreasing subwords: one consisting of its left-to-right minima,
and one consisting of all its other letters. The $UD$-factors of
$\chi(\pi)$ correspond precisely to the left-to-right minima of $\pi$.
So, given a Dyck path $\mu$, we can recover the permutation $\pi$
for which $\mu=\chi(\pi)$ by first identifying the left-to-right
minima of $\pi$ using the $UD$-factors of $\chi(\pi)$, which then
determines the rest of $\pi$ as the remaining letters must be placed
in descending order.

The bijection $\chi$ is a slight modification of a bijection due
to Krattenthaler \cite{Krattenthaler2001}. Krattenthaler's original
bijection is defined in the same way, except that the points $(k,\pi_{k})$
are to the left of the path (as opposed to the right), and the path
never goes below (as opposed to above) the antidiagonal. As such,
the $UD$-factors correspond to right-to-left maxima (as opposed to
left-to-right minima).

We use our bijection $\chi$ in lieu of Krattenthaler's because there
is a nice characterization of the Dyck paths corresponding to Jacobi
permutations under $\chi$. A \textit{descent} in a Dyck path is a
maximal consecutive subword of down steps. For example, the Dyck path
in Figure \ref{f-kratbij} has 3 descents; read from left to right,
the descents have lengths $3$, $1$, and $3$. Notice that the permutation
in Figure \ref{f-kratbij} is Jacobi, which is not a coincidence.

For convenience, we say that a descent is \textit{odd} if its length
is odd.
\begin{prop}
\label{p-123path}Let $\pi\in\mathfrak{S}_{n}(123)$. Then $\pi$
is Jacobi if and only if $\chi(\pi)$ has all descents odd.
\end{prop}

\begin{proof}
Because $\pi$ avoids $123$, it is a shuffle of two decreasing subwords,
one consisting of the left-to-right minima of $\pi$, and one consisting
of all other letters of $\pi$. Let $\pi_{i_{1}}>\pi_{i_{2}}>\cdots>\pi_{i_{m}}$
denote the left-to-right minima of $\pi$, and let $i_{m+1}=n+1$.
Then $\left|\rho_{\pi}(\pi_{i_{k}})\right|=i_{k+1}-i_{k}-1$ for all
$k\in[m]$, and $\left|\rho_{\pi}(x)\right|=0$ whenever $x$ is not
a left-to-right minimum of $\pi$. Thus, $\pi$ is Jacobi if and only
if $i_{k+1}-i_{k}-1$ is even\textemdash or equivalently, $i_{k+1}-i_{k}$
is odd\textemdash for all $k\in[m]$, and we see from (\ref{e-chidef})
that this condition is equivalent to $\chi(\pi)$ having all descents
odd.
\end{proof}
We are ready to prove Theorem \ref{t-123}.
\begin{proof}[Proof of Theorem \ref{t-123}]
Proposition \ref{p-123path} implies that the bijection $\chi$ restricts
to a bijection between $\mathfrak{J}_{n}(123)$ and the subset of
paths in $\mathcal{D}_{n}$ with all descents odd. Since the latter
is counted by $\sum_{k=0}^{\left\lfloor (n-1)/2\right\rfloor }\frac{1}{2k+1}{n-k-1 \choose k}{n \choose 2k}$
\cite[A101785]{oeis}, it follows that $j_{n}(123)$ is equal to this
number.
\end{proof}
The next lemma allows us to translate left-to-right minima and ascents
in $\pi$ to occurrences of $UD$- and $UDD$-factors of $\chi(\pi)$.
\begin{lem}
\label{l-123pathstats}Let $\pi\in\mathfrak{S}_{n}(123)$. Then:
\begin{enumerate}
\item [\normalfont{(a)}]$\lrmin(\pi)=\occ_{UD}(\chi(\pi))$, and
\item [\normalfont{(b)}]$\asc(\pi)=\occ_{UDD}(\chi(\pi))$.
\end{enumerate}
\end{lem}

\begin{proof}
We have already seen why (a) is true. For (b), notice that $\pi_{i}<\pi_{i+1}$
precisely when $\pi_{i}$ is a left-to-right minimum of $\pi$ and
$\pi_{i+1}$ is not a left-to-right minimum of $\pi$. The left-to-right
minimum $\pi_{i}$ corresponds to a $UD$-factor in $\chi(\pi)$,
and the fact that the next letter is not a left-to-right minimum means
that this $UD$-factor is followed by another down step. Hence, each
ascent of $\pi$ corresponds to a $UDD$-factor of $\chi(\pi)$. The
converse\textemdash each $UDD$-factor of $\chi(\pi)$ corresponds
to an ascent of $\pi$\textemdash is true by similar reasoning, and
(b) follows.
\end{proof}
One can also show that $\des(\pi)=\occ_{DU}(\chi(\pi))+\occ_{DDD}(\chi(\pi))$,
but it is easier to count Dyck paths by $UDD$-factors than by $DU$-
and $DDD$-factors simultaneously. For this reason, we will work with
ascents as opposed to descents in $123$-avoiding permutations.

\subsection{Left-to-right minima and ascents}

So far, we have only studied distributions of individual statistics,
but it turns out that we can find the joint distribution of $\lrmin$
and $\asc$ over $\mathfrak{J}_{n}(123)$ without much added difficulty.
To that end, let 
\[
j_{n,i,k}^{(\lrmin,\asc)}(123)\coloneqq\left|\{\,\pi\in\mathfrak{J}_{n}(123):\lrmin(\pi)=i\text{ and }\asc(\pi)=k\,\}\right|.
\]
Equivalently, $j_{n,i,k}^{(\lrmin,\asc)}(123)$ is the number of Dyck
paths of semilength $n$ with all descents odd and exactly $i$ $UD$-factors
and $k$ $UDD$-factors.
\begin{thm}
\label{t-123-lrmin-asc}For all $n$, $i$, and $k$ satisfying $n\geq i\geq k\geq0$
and $k\leq(n-i)/2$, we have 
\begin{equation}
j_{n,i,k}^{(\lrmin,\asc)}(123)={\displaystyle \frac{1}{n+1}{n+1 \choose k}{n+1-k \choose i-k}{\frac{n-i}{2}-1 \choose k-1}}\label{e-123-lrmin-asc}
\end{equation}
if $n$ and $i$ have the same parity, and $j_{n,i,k}^{(\lrmin,\asc)}(123)=0$
otherwise.
\end{thm}

The Dyck path interpretation of the numbers $j_{n,i,k}^{(\lrmin,\asc)}(123)$
will be essential to the proof of Theorem \ref{t-123-lrmin-asc}.
We will also need to consider Dyck paths with last descent even and
all other descents odd. Note that the empty path is excluded from
this family of Dyck paths since the empty path does not have a last
descent, but is vacuously a Dyck path with all descents odd.

Define the generating functions 
\begin{align*}
P=P(s,t,x) & \coloneqq\sum_{n=0}^{\infty}\sum_{\mu\in\mathcal{D}_{n}^{(1)}}s^{\occ_{UD}(\mu)}t^{\occ_{UDD}(\mu)}x^{n}\qquad\text{and}\\
Q=Q(s,t,x) & \coloneqq\sum_{n=1}^{\infty}\sum_{\mu\in\mathcal{D}_{n}^{(2)}}s^{\occ_{UD}(\mu)}t^{\occ_{UDD}(\mu)}x^{n},
\end{align*}
where $\mathcal{D}_{n}^{(1)}$ is the set of Dyck paths of semilength
$n$ with all descents odd, and $\mathcal{D}_{n}^{(2)}$ is the set
of Dyck paths of semilength $n$ with last descent even and all other
descents odd. 
\begin{lem}
\label{l-Pfunc}The generating function $P$ satisfies the functional
equation
\begin{equation}
P=1+sxP-x^{2}P^{2}+(x^{2}+stx^{3}-sx^{3})P^{3}.\label{e-Pfunc}
\end{equation}
\end{lem}

\begin{proof}
We first establish the system of functional equations 
\begin{align}
P & =1+sxP+xPQ\label{e-pathodd}\\
Q & =stx^{2}P^{2}+xP(P-1-sxP).\label{e-patheven}
\end{align}
To do so, we will use the \textit{last return decomposition} of Dyck
paths: every nonempty Dyck path $\mu$ can be uniquely written in
the form $\mu=\nu U\eta D$ where $\nu$ and $\eta$ are Dyck paths.

Let $\mu\in\mathcal{D}_{n}^{(1)}$. The empty path contributes $1$
to (\ref{e-pathodd}), so now assume that $\mu$ is nonempty. Then
$\mu=\nu U\eta D$ where $\nu$ has all descents odd, and $\eta$
is either the empty path or has all descents odd except for its last
one. The case where $\eta$ is the empty path contributes $sxP$ to
(\ref{e-pathodd}) as 
\[
\occ_{UD}(\mu)=\occ_{UD}(\nu)+1\quad\text{and}\quad\occ_{UDD}(\mu)=\occ_{UDD}(\nu),
\]
whereas if $\eta$ is nonempty, then 
\[
\occ_{UD}(\mu)=\occ_{UD}(\nu)+\occ_{UD}(\eta)\quad\text{and}\quad\occ_{UDD}(\mu)=\occ_{UDD}(\nu)+\occ_{UDD}(\eta),
\]
so this case contributes $xPQ$ to (\ref{e-pathodd}). Hence, (\ref{e-pathodd})
holds.

To prove (\ref{e-patheven}), let $\mu\in\mathcal{D}_{n}^{(2)}$.
Then $\mu=\nu U\eta D$ where both $\nu$ and $\eta$ have all descents
odd and $\eta$ is nonempty. If $\eta$ ends with a $UD$-factor,
then 
\[
\occ_{UD}(\mu)=\occ_{UD}(\nu)+\occ_{UD}(\eta)\quad\text{and}\quad\occ_{UDD}(\mu)=\occ_{UDD}(\nu)+\occ_{UDD}(\eta)+1;
\]
the generating function for these $\eta$ is $sxP$, so this case
provides the contribution $stx^{2}P^{2}$ to (\ref{e-patheven}).
On the other hand, if $\eta$ does not end with a $UD$-factor, then
\[
\occ_{UD}(\mu)=\occ_{UD}(\nu)+\occ_{UD}(\eta)\quad\text{and}\quad\occ_{UDD}(\mu)=\occ_{UDD}(\nu)+\occ_{UDD}(\eta);
\]
these $\eta$ have generating function $P-1-sxP$, thus accounting
for the $x(P-1-sxP)$ term in (\ref{e-patheven}). Therefore, (\ref{e-patheven})
is proven.

The desired equation is obtained by substituting (\ref{e-patheven})
into (\ref{e-pathodd}) and performing some algebraic manipulations.
\end{proof}
We now use Lemma \ref{l-Pfunc} to prove Theorem \ref{t-123-lrmin-asc}.
\begin{proof}[Proof of Theorem \ref{t-123-lrmin-asc}]
It follows from Lemma \ref{l-123pathstats} that 
\[
j_{n,i,k}^{(\lrmin,\asc)}(123)=[s^{i}t^{k}x^{n}]\,P;
\]
we will determine this coefficient using the Lagrange inversion formula
\cite{Gessel2016}. In doing so, it will be helpful to define $\bar{P}\coloneqq xP$,
so that we seek the coefficient $[s^{i}t^{k}x^{n+1}]\,\bar{P}$. Multiplying
both sides of (\ref{e-Pfunc}) by $x$ and performing some algebraic
manipulations gives us 
\[
\bar{P}=x\left(1+s\bar{P}+\frac{st\bar{P}^{3}}{1-\bar{P}^{2}}\right),
\]
so $\bar{P}=xR(\bar{P})$ if we take $R(u)\coloneqq1+su+stu^{3}/(1-u^{2})$.
Applying Lagrange inversion, along with the binomial theorem and the
well-known identity 
\[
\left(\frac{1}{1-x}\right)^{k}=\sum_{l=0}^{\infty}{l+k-1 \choose k-1}x^{l},
\]
yields {\allowdisplaybreaks
\begin{align*}
[s^{i}t^{k}x^{n+1}]\,\bar{P} & =\frac{1}{n+1}\,[s^{i}t^{k}u^{n}]\,\left(1+su+\frac{stu^{3}}{1-u^{2}}\right)^{n+1}\\
 & \hspace{-2em}=\frac{1}{n+1}\,[s^{i}t^{k}u^{n}]\,\sum_{k=0}^{n+1}{n+1 \choose k}\left(\frac{stu^{3}}{1-u^{2}}\right)^{k}(1+su)^{n+1-k}\\
 & \hspace{-2em}=\frac{1}{n+1}\,[s^{i}t^{k}u^{n}]\,\sum_{k=0}^{n+1}\sum_{m=0}^{n+1-k}{n+1 \choose k}{n+1-k \choose m}s^{m}u^{m}\sum_{l=0}^{\infty}s^{k}t^{k}u^{3k}{l+k-1 \choose k-1}u^{2l}\\
 & \hspace{-2em}=\frac{1}{n+1}\,[s^{i}t^{k}u^{n}]\,\sum_{k=0}^{n+1}\sum_{m=0}^{n+1-k}\sum_{l=0}^{\infty}{n+1 \choose k}{n+1-k \choose m}{l+k-1 \choose k-1}s^{m+k}t^{k}u^{m+2l+3k}\\
 & \hspace{-2em}=\frac{1}{n+1}\,[s^{i}t^{k}u^{n}]\,\sum_{k=0}^{n+1}\sum_{i=k}^{n+1}\sum_{l=0}^{\infty}{n+1 \choose k}{n+1-k \choose i-k}{l+k-1 \choose k-1}s^{i}t^{k}u^{i+2l+2k}.
\end{align*}
}We have $j_{n,i,k}^{(\lrmin,\asc)}(123)=0$ whenever $n$ and $i$
have opposite parity, as in that case $i+2l+2k$ cannot equal $n$.
Assuming that $n$ and $i$ have the same parity, note that $i+2l+2k=n$
is equivalent to $l=(n-i-2k)/2$, so
\[
j_{n,i,k}^{(\lrmin,\asc)}(123)=\frac{1}{n+1}{n+1 \choose k}{n+1-k \choose i-k}{\frac{n-i}{2}-1 \choose k-1}
\]
if $0\leq k\leq i\leq n$ and $(n-i-2k)/2\geq0$; the latter inequality
is equivalent to $k\leq(n-i)/2$.
\end{proof}
The distributions of the individual statistics $\lrmin$ and $\asc$
over $\mathfrak{J}_{n}(123)$ are given by the following two theorems,
which can be proven by either summing (\ref{e-123-lrmin-asc}) over
$i$ or $k$, or by setting $s=1$ or $t=1$ in (\ref{e-Pfunc}) prior
to performing Lagrange inversion.
\begin{thm}
\label{t-123lrmin}For all $n\geq k\geq0$, we have 
\[
j_{n,k}^{\lrmin}(123)=\frac{1}{n+1}{n+1 \choose k}{\frac{n+k}{2}-1 \choose k-1}
\]
if $n$ and $k$ have the same parity, and $j_{n,k}^{\lrmin}(123)=0$
otherwise.
\end{thm}

\begin{thm}
\label{t-123asc}For all $n\geq k\geq0$, we have 
\[
j_{n,k}^{\asc}(123)=\frac{1}{n+1}\sum_{i=0}^{\left\lfloor (n-3k)/2\right\rfloor }{n+1 \choose k}{n-k+1 \choose 2i+2k+1}{i+k-1 \choose k-1}.
\]
\end{thm}

\begin{table}[!h]
\begin{centering}
\renewcommand{\arraystretch}{1.1}%
\begin{tabular}{|>{\centering}p{20bp}|>{\centering}p{20bp}|>{\centering}p{20bp}|>{\centering}p{20bp}|>{\centering}p{20bp}|>{\centering}p{20bp}|>{\centering}p{20bp}|>{\centering}p{20bp}|>{\centering}p{20bp}|>{\centering}p{20bp}|>{\centering}p{20bp}|}
\hline 
$n\backslash k$ & $0$ & $1$ & $2$ & $3$ & $4$ & $5$ & $6$ & $7$ & $8$ & $9$\tabularnewline
\hline 
$0$ & $1$ &  &  &  &  &  &  &  &  & \tabularnewline
\hline 
$1$ & $0$ & $1$ &  &  &  &  &  &  &  & \tabularnewline
\hline 
$2$ & $0$ & $0$ & $1$ &  &  &  &  &  &  & \tabularnewline
\hline 
$3$ & $0$ & $1$ & $0$ & $1$ &  &  &  &  &  & \tabularnewline
\hline 
$4$ & $0$ & $0$ & $4$ & $0$ & $1$ &  &  &  &  & \tabularnewline
\hline 
$5$ & $0$ & $1$ & $0$ & $10$ & $0$ & $1$ &  &  &  & \tabularnewline
\hline 
$6$ & $0$ & $0$ & $9$ & $0$ & $20$ & $0$ & $1$ &  &  & \tabularnewline
\hline 
$7$ & $0$ & $1$ & $0$ & $42$ & $0$ & $35$ & $0$ & $1$ &  & \tabularnewline
\hline 
$8$ & $0$ & $0$ & $16$ & $0$ & $140$ & $0$ & $56$ & $0$ & $1$ & \tabularnewline
\hline 
$9$ & $0$ & $1$ & $0$ & $120$ & $0$ & $378$ & $0$ & $84$ & $0$ & $1$\tabularnewline
\hline 
\end{tabular}
\par\end{centering}
\caption{\label{tb-123-lrmin}The numbers $j_{n,k}^{\protect\lrmin}(123)$
up to $n=9$.}
\end{table}
\begin{table}[!h]
\begin{centering}
\renewcommand{\arraystretch}{1.1}%
\begin{tabular}{|>{\centering}p{20bp}|>{\centering}p{20bp}|>{\centering}p{20bp}|>{\centering}p{20bp}|>{\centering}p{20bp}|}
\hline 
$n\backslash k$ & $0$ & $1$ & $2$ & $3$\tabularnewline
\hline 
$0$ & $1$ &  &  & \tabularnewline
\hline 
$1$ & $1$ &  &  & \tabularnewline
\hline 
$2$ & $1$ &  &  & \tabularnewline
\hline 
$3$ & $1$ & $1$ &  & \tabularnewline
\hline 
$4$ & $1$ & $4$ &  & \tabularnewline
\hline 
$5$ & $1$ & $11$ &  & \tabularnewline
\hline 
$6$ & $1$ & $26$ & $3$ & \tabularnewline
\hline 
$7$ & $1$ & $57$ & $21$ & \tabularnewline
\hline 
$8$ & $1$ & $120$ & $92$ & \tabularnewline
\hline 
$9$ & $1$ & $247$ & $324$ & $12$\tabularnewline
\hline 
\end{tabular}
\par\end{centering}
\caption{\label{tb-123-asc}The numbers $j_{n,k}^{\protect\asc}(123)$ up to
$n=9$.}
\end{table}

Unfortunately, we do not have a formula for the distribution of the
last letter statistic over $\mathfrak{J}_{n}(123)$. Unlike with left-to-right
minima and ascents, which correspond to $UD$- and $UDD$-factors,
the last letter of $\pi$ does not seem to correspond nicely to a
feature of the path $\chi(\pi)$. The value of the last letter does
translate to the height of the last $UD$-factor under Krattenthaler's
original bijection, which can be used to find the distribution of
$\last$ over the full avoidance class $\mathfrak{S}_{n}(123)$. However,
we were unable to adapt this approach to the Jacobi permutation setting,
as we do not have a simple characterization of the Dyck paths corresponding
to Jacobi permutations under Krattenthaler's bijection.

\subsection{\texorpdfstring{$123$}{123}-avoiding doubly Jacobi permutations}

Recall from Section \ref{ss-213-231-312-doubly} that all permutations
in $\mathfrak{J}_{n}(213)$, $\mathfrak{J}_{n}(231)$, and $\mathfrak{J}_{n}(312)$
are doubly Jacobi. The same is not true for $\mathfrak{J}_{n}(123)$,
but the doubly Jacobi permutations in $\mathfrak{J}_{n}(123)$ have
a nice enumeration.

Define the numbers $s_{n}$ by 
\[
s_{n+1}=s_{n}+\sum_{k=1}^{n-1}s_{k}s_{n-1-k}
\]
for all $n\geq0$, with initial term $s_{0}=1$. The $s_{n}$ have
been called \textit{secondary structure numbers} in the literature
(e.g., in \cite{Deutsch2012}), as they count certain graphs which
are used to model secondary structures of RNA molecules. 
\begin{table}[H]
\begin{centering}
\renewcommand{\arraystretch}{1.1}%
\begin{tabular}{c|>{\centering}p{22bp}|>{\centering}p{22bp}|>{\centering}p{22bp}|>{\centering}p{22bp}|>{\centering}p{22bp}|>{\centering}p{22bp}|>{\centering}p{22bp}|>{\centering}p{22bp}|>{\centering}p{22bp}|>{\centering}p{22bp}|>{\centering}p{22bp}|>{\centering}p{22bp}}
$n$ & $0$ & $1$ & $2$ & $3$ & $4$ & $5$ & $6$ & $7$ & $8$ & $9$ & $10$ & $11$\tabularnewline
\hline 
$s_{n}$ & $1$ & $1$ & $1$ & $2$ & $4$ & $8$ & $17$ & $37$ & $82$ & $185$ & $423$ & $978$\tabularnewline
\end{tabular}
\par\end{centering}
\caption{\label{tb-secstruc}The secondary structure numbers $s_{n}$ up to
$n=11$.}
\end{table}

An \textit{ascent} of a Dyck path is a maximal consecutive subword
of up steps, and an ascent is \textit{odd} if it has odd length. Another
combinatorial interpretation for $s_{n}$, which we shall use for
our proof of the next theorem, is the number of Dyck paths of semilength
$n$ with every ascent and descent odd \cite[A004148]{oeis}.
\begin{thm}
\label{t-123doubly}For all $n\geq0$, the number of doubly Jacobi
permutations in $\mathfrak{J}_{n}(123)$ is equal to the secondary
structure number $s_{n}$.
\end{thm}

\begin{proof}
Let $\pi\in\mathfrak{J}_{n}(123)$, so that $\chi(\pi)$ has all descents
odd (Proposition \ref{p-123path}). It suffices to prove that $\pi^{-1}$
is Jacobi if and only if $\chi(\pi)$ has all ascents odd.

Let $\pi_{i_{1}}>\pi_{i_{2}}>\cdots>\pi_{i_{m}}$ be the left-to-right
minima of $\pi$, so that $i_{m}>\cdots>i_{2}>i_{1}$ are the left-to-right
minima of $\pi^{-1}$ by Lemma \ref{l-lrminv}. Also, set $i_{0}=0$
and $\pi_{0}=n+1$. By the same reasoning as in the proof of Proposition
\ref{p-123path}, $\pi^{-1}$ is Jacobi if and only if $\pi_{i_{k-1}}-\pi_{i_{k}}$
is odd for all $k\in[m]$, but observe from (\ref{e-chidef}) that
this is precisely the condition under which $\chi(\pi)$ has all ascents
odd.
\end{proof}
One can study the distributions of $\lrmin$ and $\asc$ over $123$-avoiding
doubly Jacobi permutations by counting $UD$- and $UDD$-factors in
Dyck paths with all ascents and descents odd, but we will not do this
here.

\section{\label{s-132-321}\texorpdfstring{$132$}{132}- and \texorpdfstring{$321$}{321}-avoiding
Jacobi permutations}

The two remaining single-pattern Jacobi avoidance classes that we
have yet to enumerate are $\mathfrak{J}_{n}(132)$ and $\mathfrak{J}_{n}(321)$.
In this section, we shall find that $\mathfrak{J}_{n}(132)=\{n\cdots21\}$
for all $n\geq1$, and that the permutations in $\mathfrak{J}_{n}(321)$
are alternating.

\subsection{\texorpdfstring{$132$}{132}-avoidance}
\begin{prop}
\label{p-132}The only permutation in $\mathfrak{J}_{S}(132)$ is
the decreasing permutation of $S$.
\end{prop}

\begin{proof}
Any decreasing permutation is clearly Jacobi and $132$-avoiding,
so it suffices to show that any nondecreasing Jacobi permutation has
an occurrence of $132$. Let $\pi$ be a nondecreasing Jacobi permutation,
and let $k$ be its largest ascent. Then, by Lemma \ref{l-Jacobi}
(d), we know that $\pi_{k+1}$ is not the last letter of $\pi$. Since
$k$ is the largest ascent of $\pi$, we have $\pi_{k+1}>\pi_{k+2}$.
In addition, we have $\pi_{k}<\pi_{k+2}$ or else $\rho_{\pi}(\pi_{k})=\pi_{k+1}$
which is of odd length. Therefore, $\pi_{k}\pi_{k+1}\pi_{k+2}$ is
an occurrence of $132$ in $\pi$.
\end{proof}
\begin{cor}
\label{c-132}We have $j_{n}(132)=1$ for all $n\geq0$.
\end{cor}

\subsection{\label{ss-321}\texorpdfstring{$321$}{321}-avoidance: enumeration,
ascents, left-to-right minima}

The theorem below completely characterizes $321$-avoiding Jacobi
permutations.
\begin{thm}
\label{t-321}Let $\pi\in\mathfrak{S}_{S}(321)$, $n=\left|S\right|$,
and $y=\min S$.
\begin{enumerate}
\item [\normalfont{(a)}]If $n$ is even, then $\pi$ is Jacobi if and only
if $\pi$ is down-up alternating.
\item [\normalfont{(b)}]If $n$ is odd, then $\pi$ is Jacobi if and only
if $\pi$ is up-down alternating and $\pi_{1}=y$.
\end{enumerate}
\end{thm}

Our proof of Theorem \ref{t-321} will require two lemmas. Together,
these lemmas tell us that an even-length $321$-avoiding Jacobi permutation
must have its smallest letter in its second position and its largest
letter in its penultimate position, and that an odd-length $321$-avoiding
Jacobi permutation must begin with its smallest letter.
\begin{lem}
\label{l-321-even}If $\pi\in\mathfrak{J}_{S}(321)$ where $\left|S\right|=2n>0$,
$y=\min S$, and $z=\max S$, then $\pi_{2}=y$ and $\pi_{2n-1}=z$.
\end{lem}

\begin{proof}
Write $\pi=\alpha y\beta$, so that $\alpha$ and $\beta$ are $321$-avoiding
Jacobi permutations and $\beta$ is of even length. Observe that $\alpha$
cannot be empty, or else $\pi=y\beta$ would be of odd length. On
the other hand, if $\alpha=\alpha_{1}\alpha_{2}\cdots\alpha_{m}$
where $m\geq2$, then $\alpha_{m-1}\alpha_{m}y$ would be an occurrence
of $321$ in $\pi$; see Lemma \ref{l-Jacobi}~(d). Therefore, $\alpha$
must have length 1, which implies $\pi_{2}=y$. The fact that $\pi_{2n-1}=z$
is also a consequence of Lemma \ref{l-Jacobi}~(d); $\pi$ cannot
end with $z$ because $\pi$ ends with a descent, and there cannot
be more than one letter in $\pi$ after $z$, or else $z$ and the
last two letters of $\pi$ would form an occurrence of $321$.
\end{proof}
We omit the proof of the next lemma because it is very similar to
that of the above.
\begin{lem}
\label{l-321-odd}If $\pi\in\mathfrak{J}_{S}(321)$ where $\left|S\right|=2n+1$
and $y=\min S$, then $\pi_{1}=y$.
\end{lem}

We now prove Theorem \ref{t-321}.
\begin{proof}[Proof of Theorem \ref{t-321}]
We first prove the forward direction of (a) using induction, with
the base case ($n=0$) beng trivial. Let $\pi$ be a $321$-avoiding
Jacobi permutation of even length $n$, and suppose that the forward
direction of (a) holds for all even lengths less than $n$. By Lemma
\ref{l-321-even}, $\pi=\pi_{1}y\beta$ where $\beta$ is a $321$-avoiding
Jacobi permutation of even length. Then $\beta$ is down-up according
to the induction hypothesis, so $\pi=\pi_{1}y\beta$ is down-up as
desired.

Conversely, suppose that $\pi$ is a $321$-avoiding down-up permutation
of even length. Because $\pi$ is down-up, we have $\left|\rho_{\pi}(\pi_{k})\right|=0$
for all odd $k\in[n]$. On the other hand, if $k\in[n]$ is even,
then $\pi_{k}$ is a right-to-left minimum of $\pi$; otherwise, if
there is a letter $x$ appearing after $\pi_{k}$ which is less than
$\pi_{k}$, then $\pi_{k-1}\pi_{k}x$ would be an occurrence of $321$
in $\pi$. Thus, $\rho_{\pi}(\pi_{k})=\pi_{k+1}\pi_{k+2}\cdots\pi_{n}$
for all even $k\in[n]$, and these subwords are all of even length.
It follows that $\pi$ is Jacobi, and we have proven (a).

To prove the forward direction of (b), let $\pi$ be a $321$-avoiding
Jacobi permutation of odd length $n$. Then, by Lemma \ref{l-321-odd},
we have $\pi=y\beta$ where $\beta$ is a $321$-avoiding Jacobi permutation
of even length. Then, applying part (a) of this theorem, we have that
$\beta$ is down-up. Hence, $\pi=y\beta$ is up-down.

Finally, let $\pi$ be a $321$-avoiding up-down permutation of odd
length $n$ with $\pi_{1}=y$. Also, let $\beta=\pi_{2}\pi_{3}\cdots\pi_{n}$,
so that $\pi=y\beta$. Observe that $\beta$ is a $321$-avoiding
down-up permutation of even length, so $\beta$ is Jacobi by part
(a). Then $\rho_{\pi}(\pi_{k})=\rho_{\beta}(\pi_{k})$ is of even
length for all $2\leq k\leq n$, and $\rho_{\pi}(y)=\beta$ is of
even length as well. We conclude that $\pi$ is Jacobi.
\end{proof}
Let 
\[
C_{n}\coloneqq\frac{1}{n+1}{2n \choose n}
\]
denote the $n$th \textit{Catalan number}. We have the following corollary
to Theorem \ref{t-321}.
\begin{cor}
\label{c-321}For all $n\geq0$, we have $j_{n}(321)=C_{\left\lfloor n/2\right\rfloor }$.
\end{cor}

\begin{proof}
It is known \cite[Bijective Exercise 146]{Stanley2015} that $C_{n}$
is the number of down-up permutations in $\mathfrak{S}_{2n}(321)$,
so $j_{2n}(321)=C_{n}$. In addition, the up-down permutations in
$\mathfrak{S}_{2n+1}(321)$ beginning with $1$ are in bijection with
the down-up permutations in $\mathfrak{S}_{2n}(321)$: simply remove
the first letter and standardize the resulting permutation. Thus,
$j_{2n+1}(321)=C_{n}$ as well.
\end{proof}
The next two corollaries are immediate consequences of Theorem \ref{t-321}.
\begin{cor}
\label{c-321-asc}For all $n\geq1$ and $\pi\in\mathfrak{J}_{2n-1}(321)\sqcup\mathfrak{J}_{2n}(321)$,
we have $\asc(\pi)=n-1$. 
\end{cor}

\begin{cor}
\label{c-321-lrmin}For all $n\geq1$ and $\pi\in\mathfrak{J}_{n}(321)$,
we have 
\[
\lrmin(\pi)=\begin{cases}
2, & \text{if }n\text{ is even,}\\
1, & \text{if }n\text{ is odd.}
\end{cases}
\]
\end{cor}

\subsection{\texorpdfstring{$321$}{321}-avoidance: last letter}

While the distributions of $\asc$ and $\lrmin$ over $\mathfrak{J}_{n}(321)$
are not very interesting, it takes more work to prove the next theorem,
which gives the distribution of $\last$ over $\mathfrak{J}_{n}(321)$.
\begin{thm}
\label{t-321-last}For all $n\geq1$, we have
\[
j_{2n,k}^{\last}(321)=\begin{cases}
{\displaystyle \frac{2n-k}{n}{k-1 \choose n-1}}, & \text{if }n\leq k\leq2n-1,\\
0, & \text{otherwise},
\end{cases}
\]
and 
\[
j_{2n+1,k}^{\last}(321)=\begin{cases}
{\displaystyle \frac{2n-k+1}{n}{k-2 \choose n-1}}, & \text{if }n+1\leq k\leq2n,\\
0, & \text{otherwise}.
\end{cases}
\]
\end{thm}

\begin{table}
\begin{centering}
\renewcommand{\arraystretch}{1.1}%
\begin{tabular}{|>{\centering}p{20bp}|>{\centering}p{20bp}|>{\centering}p{20bp}|>{\centering}p{20bp}|>{\centering}p{20bp}|>{\centering}p{20bp}|>{\centering}p{20bp}|>{\centering}p{20bp}|>{\centering}p{20bp}|}
\hline 
$n\backslash k$ & $1$ & $2$ & $3$ & $4$ & $5$ & $6$ & $7$ & $8$\tabularnewline
\hline 
$1$ & $1$ &  &  &  &  &  &  & \tabularnewline
\hline 
$2$ & $1$ &  &  &  &  &  &  & \tabularnewline
\hline 
$3$ & $0$ & $1$ &  &  &  &  &  & \tabularnewline
\hline 
$4$ & $0$ & $1$ & $1$ &  &  &  &  & \tabularnewline
\hline 
$5$ & $0$ & $0$ & $1$ & $1$ &  &  &  & \tabularnewline
\hline 
$6$ & $0$ & $0$ & $1$ & $2$ & $2$ &  &  & \tabularnewline
\hline 
$7$ & $0$ & $0$ & $0$ & $1$ & $2$ & $2$ &  & \tabularnewline
\hline 
$8$ & $0$ & $0$ & $0$ & $1$ & $3$ & $5$ & $5$ & \tabularnewline
\hline 
$9$ & $0$ & $0$ & $0$ & $0$ & $1$ & $3$ & $5$ & $5$\tabularnewline
\hline 
\end{tabular}
\par\end{centering}
\caption{\label{tb-321-last}The numbers $j_{n,k}^{\protect\last}(321)$ up
to $n=9$.}
\end{table}

One may notice from Table \ref{tb-321-last} that the $j_{2n,k}^{\last}(321)$
and $j_{2n+1,k}^{\last}(321)$ are each a shifted version of the \textit{ballot
numbers} $B_{n,k}$, defined by 
\[
B_{n,k}\coloneqq\frac{n-k+1}{n+1}{n+k \choose n}
\]
for all $n\geq k\geq0$ \cite[A009766]{oeis}.

We first state a few lemmas needed for the proof of Theorem \ref{t-321-last}.
\begin{lem}
\label{l-321-lastbound}Let $n\geq1$. 
\begin{enumerate}
\item [\normalfont{(a)}]If $\pi\in\mathfrak{J}_{2n}(321)$, then $n\leq\pi_{2n}\leq2n-1$.
\item [\normalfont{(b)}]If $\pi\in\mathfrak{J}_{2n+1}(321)$, then $n+1\leq\pi_{2n+1}\leq2n$.
\end{enumerate}
\end{lem}

\begin{proof}
Every $321$-avoiding permutation is a shuffle of two increasing subwords,
and because $\pi$ is down-up, the odd-indexed letters form one increasing
subword and the even-indexed letters form the other. If $\pi$ has
length $2n$, then we have $\pi_{2n}\geq n$ because $\pi_{2n}$ is
the $n$th term of an increasing subword of length $n$, and we also
know from Lemma~\ref{l-Jacobi}~(d) that $\pi_{2n}\neq2n$ or else
$\pi$ would end with an ascent. The same reasoning holds when $\pi$
has odd length.
\end{proof}
\begin{lem}
\label{l-321-lastrec}For all $n\geq2$ and $n\leq k\leq2n-1$, we
have
\[
j_{2n,k}^{\last}(321)=\sum_{l=n-1}^{k-1}j_{2n-2,l}^{\last}(321).
\]
\end{lem}

\begin{proof}
Let $\mathfrak{J}_{n,k}^{\last}(321)$ denote the subset of permutations
in $\mathfrak{J}_{n}(321)$ with $\last(\pi)=k$. We provide a bijection
$\delta\colon\bigsqcup_{l=n-1}^{k-1}\mathfrak{J}_{2n-2,l}^{\last}(321)\rightarrow\mathfrak{J}_{2n,k}^{\last}(321).$
Given $\pi\in\mathfrak{J}_{2n-2,l}^{\last}(321)$, define $\delta(\pi)$
to be the permutation obtained from $\pi$ by increasing every letter
at least $k$ by 1, and then appending the letters $2n$ and $k$
to the result. For example, if $n=4$ and $k=5$, then $\delta(415263)=41627385$. 

To see why $\delta$ is well-defined, first recall from Theorem \ref{t-321}
(a) that $\pi$ is down-up. By construction, $\delta(\pi)$ is necessarily
a down-up permutation in $\mathfrak{S}_{2n}$ that ends with $k$.
Suppose to the contrary that $\delta(\pi)$ has an occurrence of $321$.
Then the last letter $k$ must correspond to the ``1''\textemdash i.e.,
there exist two letters $x$ and $y$ in $\delta(\pi)$ for which
$xyk$ is an occurrence of $321$. Because $x,y>k$, they correspond
to the letters $x-1$ and $y-1$ in $\pi$, neither of which is the
last letter $l$ of $\pi$ because $l<k$. It follows that $(x-1)(y-1)l$
is an occurrence of $321$ in $\pi$, a contradiction. Thus, $\delta(\pi)$
is a $321$-avoiding down-up permutation, so $\delta(\pi)$ is Jacobi
by Theorem \ref{t-321} (a), and therefore $\delta(\pi)\in\mathfrak{J}_{2n,k}^{\last}(321)$.

The inverse of $\delta$ is as follows: given $\pi\in\mathfrak{J}_{2n,k}^{\last}(321)$,
we remove the last two letters of $\pi$, and then decrease every
letter greater than $k$ by 1. The result\textemdash call it $\pi^{\prime}$\textemdash clearly
remains a $321$-avoiding down-up permutation. Also, because $2n$
is the penultimate letter of $\pi$ (Lemma~\ref{l-321-even}), it
was removed and therefore $\pi^{\prime}$ is a permutation on the
letters $1,2,\dots,2n-2$. 

The last letter $\pi_{2n-2}^{\prime}$ of $\pi^{\prime}$ is at least
$n-1$ by Lemma \ref{l-321-lastbound} (a); it remains to show that
$\pi_{2n-2}^{\prime}\leq k-1$. Suppose instead that $\pi_{2n-2}^{\prime}\geq k$.
Since $\pi^{\prime}$ is down-up and of even length, we have $\pi_{2n-3}^{\prime}>\pi_{2n-2}^{\prime}$,
so $\pi_{2n-3}>\pi_{2n-2}$. Then $\pi_{2n-3}\pi_{2n-2}k$ would be
an occurrence of $321$ in $\pi$, a contradiction. Hence, $\pi^{\prime}\in\mathfrak{J}_{2n-2,l}^{\last}(321)$
for some $n-1\leq l\leq k-1$. Since $\delta$ has a well-defined
inverse, we conclude that it is a bijection.
\end{proof}
\begin{lem}
\label{l-321-lastrep}For all $n\geq1$ and $n+1\leq k\leq2n$, we
have
\[
j_{2n+1,k}^{\last}(321)=j_{2n,k-1}^{\last}(321).
\]
\end{lem}

\begin{proof}
It is clear that the bijection from $\mathfrak{J}_{2n+1}(321)\rightarrow\mathfrak{J}_{2n}(321)$\textemdash removing
the initial letter $1$ and standardizing\textemdash decreases the
value of the last letter by 1.
\end{proof}
We are now ready to prove Theorem \ref{t-321-last}.
\begin{proof}[Proof of Theorem \ref{t-321-last}]
The ballot numbers $B_{n,k}$ satisfy the recurrence
\[
B_{n,k}=\sum_{i=0}^{k}B_{n-1,i}
\]
for all $n\geq1$ and $0\leq k\leq n$, along with initial term $B_{0,0}=1$
\cite[A009766]{oeis}. Comparing this recurrence with that for $j_{2n,k}^{\last}(321)$
given in Lemma \ref{l-321-lastrec}, we find that 
\[
j_{2n,k}^{\last}(321)=B_{n-1,k-n}=\frac{2n-k}{n}{k-1 \choose n-1}
\]
for all $n\geq2$ and $n\leq k\leq2n-1$.\footnote{In the case of $n=1$, the formula for $j_{2n,k}^{\last}(321)$ can
be verified directly.} Moreover, we know from Lemma \ref{l-321-lastbound} (a) that $j_{2n,k}^{\last}(321)=0$
if $k$ does not satisfy $n\leq k\leq2n-1$.

The formula for $j_{2n+1,k}^{\last}(321)$ follows from that for $j_{2n,k}^{\last}(321)$
along with Lemmas \ref{l-321-lastbound} (b) and \ref{l-321-lastrep}.
\end{proof}

\subsection{\texorpdfstring{$321$}{321}-avoiding doubly Jacobi permutations}

Finally, we prove that for all $n\geq0$, there is exactly one doubly
Jacobi permutation in $\mathfrak{J}_{n}(321)$.
\begin{lem}
\label{l-321doublyeven}Let $n\geq1$. If $\pi\in\mathfrak{J}_{2n}(321)$
is doubly Jacobi, then $\pi_{1}=2$.
\end{lem}

\begin{proof}
The result clearly holds for $n=1$, so assume $n\geq2$. Let $\pi\in\mathfrak{J}_{2n}(321)$
be doubly Jacobi, and suppose for the sake of contradiction that $\pi_{1}\neq2$.
By Lemma \ref{l-321-even}, we know that $\pi_{1}\neq1$, so $\pi_{1}\geq3$.
This means that there are at least two letters preceding $1$ in $\pi^{-1}$\textemdash i.e.,
if $\pi_{k}^{-1}=1$, then $k\geq3$. Notice that we must have $\pi_{k-2}^{-1}>\pi_{k-1}^{-1}$
because $\pi^{-1}$ is Jacobi; otherwise, we would have $\rho_{\pi^{-1}}(\pi_{k-2}^{-1})=\pi_{k-1}^{-1}$,
which has odd length. But then $\pi_{k-2}^{-1}\pi_{k-1}^{-1}\pi_{k}^{-1}$
is an occurrence of $321$ in $\pi^{-1}$, which is impossible because
$\pi$ being $321$-avoiding means that $\pi^{-1}$ is $321$-avoiding
as well (Lemma \ref{l-pavinv}). Therefore, $\pi_{1}=2$.
\end{proof}
\begin{prop}
\label{p-321doublyeven}If $\pi\in\mathfrak{J}_{2n}(321)$ is doubly
Jacobi, then 
\[
\pi=\underset{n\text{ copies}}{\underbrace{21\oplus21\oplus\cdots\oplus21}}=2143\cdots(2n)(2n-1).
\]
\end{prop}

\begin{proof}
We induct on $n$, with the base case ($n=0$) being trivial. Taking
$n\geq1$, suppose that the only doubly Jacobi permutation in $\mathfrak{J}_{2n-2}(321)$
is $\underset{n-1\text{ copies}}{\underbrace{21\oplus21\oplus\cdots\oplus21}}$,
and let $\pi\in\mathfrak{J}_{2n}(321)$ be doubly Jacobi. Lemmas \ref{l-321-even}
and \ref{l-321doublyeven} tell us that $\pi=21\oplus\beta$ for some
$\beta\in\mathfrak{J}_{2n-2}(321)$, which implies $\pi^{-1}=21\oplus\beta^{-1}$
by Lemma \ref{l-invconcat1}. Then, because $\pi^{-1}$ is Jacobi,
we know from Lemma \ref{l-Jacobi} (a) that $\beta^{-1}$ is also
Jacobi. Hence, $\beta$ is doubly Jacobi, and we conclude from the
induction hypothesis that
\[
\pi=21\oplus\beta=21\oplus(\underset{n-1\text{ copies}}{\underbrace{21\oplus21\oplus\cdots\oplus21}})=\underset{n\text{ copies}}{\underbrace{21\oplus21\oplus\cdots\oplus21}}.\qedhere
\]
\end{proof}
\begin{prop}
\label{p-321doublyodd}If $\pi\in\mathfrak{J}_{2n+1}(321)$ is doubly
Jacobi, then 
\[
\pi=1\oplus(\underset{n\text{ copies}}{\underbrace{21\oplus21\oplus\cdots\oplus21}})=13254\cdots(2n+1)(2n).
\]
\end{prop}

\begin{proof}
We know from Lemma \ref{l-321-odd} that $\pi=1\oplus\beta$ for some
$\beta\in\mathfrak{J}_{2n}(321)$. The same reasoning used in the
previous proof shows that $\beta$ is doubly Jacobi, and the result
follows from Proposition \ref{p-321doublyeven}.
\end{proof}
\begin{cor}
For all $n\geq0$, there exists exactly one doubly Jacobi permutation
in $\mathfrak{J}_{n}(321)$.
\end{cor}

Notice that $2143\cdots(2n)(2n-1)$ and $13254\cdots(2n+1)(2n)$ are
involutions, that is, they are self-inverse. In fact, $2143\cdots(2n)(2n-1)$
is the only involution in $\mathfrak{J}_{2n}(321)$ and $13254\cdots(2n+1)(2n)$
is the only involution in $\mathfrak{J}_{2n+1}(321)$, so $\pi\in\mathfrak{J}_{n}(321)$
is doubly Jacobi if and only if $\pi=\pi^{-1}$.

\section{\label{s-mult}Multiple restrictions}

We have, as of this point, enumerated each of the Jacobi avoidance
classes $\mathfrak{J}_{n}(\sigma)$ where $\sigma$ is a single pattern
of length 3. In this section, we shall enumerate $\mathfrak{J}_{n}(\Pi)$
for all subsets $\Pi\subseteq\mathfrak{S}_{3}$ of size at least $2$.
Unlike before, we will only determine $j_{n}(\Pi)$ for each of these
$\Pi$; we will not study the distribution of any statistics over
these avoidance classes.

\subsection{The trivial cases \texorpdfstring{$132 \in \Pi$}{132 in Pi} and
\texorpdfstring{$321 \in \Pi$}{321 in Pi}}

If $\Pi\subseteq\mathfrak{S}_{3}$ with $\left|\Pi\right|\geq2$ contains
$132$ or $321$, then $j_{n}(\Pi)\leq1$. We consider these trivial
cases here before turning our attention to more interesting results.
\begin{prop}
\label{p-132-mult}If $132\in\Pi\subseteq\mathfrak{S}_{3}$, then
\[
j_{n}(\Pi)=\begin{cases}
0, & \text{if }321\in\Pi\\
1, & \text{otherwise.}
\end{cases}
\]
\end{prop}

\begin{proof}
Immediate from Proposition \ref{p-132}.
\end{proof}
\begin{prop}
\label{p-321-mult}Suppose $321\in\Pi\subseteq\mathfrak{S}_{3}$ and
$\left|\Pi\right|\geq2$.
\begin{enumerate}
\item [\normalfont{(a)}]If $\Pi$ contains $123$, $132$, or $213$, then
$j_{n}(\Pi)=0$ for all $n\geq5$.
\item [\normalfont{(b)}]Otherwise, $j_{n}(\Pi)=1$ for each $n\geq0$.
\end{enumerate}
\end{prop}

\begin{proof}
The case $123\in\Pi$ is an immediate consequence of the Erd\H{o}s\textendash Szekeres
theorem \cite{Erdoes1935}, whereas the case $132\in\Pi$ follows
from Proposition \ref{p-132-mult}.

Suppose that $213\in\Pi$, and let $\pi\in\mathfrak{J}_{n}(\Pi)$
where $n\geq4$. If $n$ is even, then Lemma \ref{l-321-even} implies
that $\pi_{2}=1$, and in fact we must have $\pi_{1}=n$ in order
for $\pi$ to avoid $213$. But recall from Lemma \ref{l-Jacobi}
(d) that $\pi$ ends with a descent, so $\pi_{1}\pi_{n-1}\pi_{n}$
is an occurrence of $321$ in $\pi$. Hence, no such $\pi$ can exist.
If $n$ is odd, then from Lemma \ref{l-321-odd}, we have $\pi=1\oplus\beta$
where $\beta\in\mathfrak{J}_{n-1}(\Pi)$. However, we had just shown
that there are no permutations in $\mathfrak{J}_{n-1}(\Pi)$, as $n-1$
is even. Therefore, (a) is proven.

Next, suppose that $\Pi$ does not contain one of $123$, $132$,
or $213$. Then because $\left|\Pi\right|\geq2$, we must have $231\in\Pi$
or $312\in\Pi$. 

Assume that $231\in\Pi$. Let $\pi\in\mathfrak{J}_{2n}(\Pi)$ where
$n\geq1$. Again, $\pi_{2}=1$, and to avoid $312$ we must have $\pi_{1}=2$.
Thus, $\pi=21\oplus\beta$ where $\beta\in\mathfrak{J}_{2n-2}(\Pi)$;
this gives us a bijection between $\mathfrak{J}_{2n}(\Pi)$ and $\mathfrak{J}_{2n-2}(\Pi)$.
On the other hand, when $\pi\in\mathfrak{J}_{2n-1}(\Pi)$, we have
$\pi=1\oplus\beta$ where $\beta\in\mathfrak{J}_{2n-2}(\Pi)$. So,
$\mathfrak{J}_{2n-1}(\Pi)$ and $\mathfrak{J}_{2n-2}(\Pi)$ are in
bijection as well. Consequently, $j_{2n-2}(\Pi)=j_{2n-1}(\Pi)=j_{2n}(\Pi)$
for all $n\geq1$, and because $j_{0}(\Pi)=1$, we conclude that $j_{n}(\Pi)=1$
for all $n\geq0$.

The result for $312\in\Pi$ follows from Lemma \ref{l-pavinv} and
Theorem \ref{t-312-231-inv}, thus completing the proof of (b).
\end{proof}

\subsection{Double restrictions}

We will now determine $j_{n}(\Pi)$ for each $\Pi\subseteq\mathfrak{S}_{3}$
of size 2 (and where $132,321\notin\Pi$). Our first result here will
feature an appearance of the \textit{Fibonacci numbers} $f_{n}$,
defined by $f_{n}=f_{n-1}+f_{n-2}$ for all $n\geq2$ along with $f_{0}=f_{1}=1$.
But first, we require a preliminary lemma.
\begin{lem}
\label{l-123-213-ini}Let $n\geq2$. If $\pi\in\mathfrak{J}_{n}(123,213)$,
then either $\pi_{1}=n$, or $\pi_{2}=n$ and $\pi_{3}=n-1$.
\end{lem}

\begin{proof}
Suppose that $\pi\in\mathfrak{J}_{n}(123,213)$ and $\pi_{1}\neq n$.
Then $\pi_{2}=n$, or else $\pi_{1}\pi_{2}n$ would be an occurrence
of $123$ or $213$ in $\pi$. Also, we cannot have $\pi_{1}=n-1$,
or else $\rho_{\pi}(\pi_{1})=\pi_{2}$ which has odd length. And if
$\pi_{k}=n-1$ for some $k\geq4$, then $\pi_{1}\pi_{3}\pi_{k}$ would
be an occurrence of $123$ or $213$ in $\pi$.
\end{proof}
\begin{thm}
\label{t-123-213}For all $n\geq1$, we have $j_{n}(123,213)=f_{n-1}$.
\end{thm}

\begin{proof}
Define $\phi\colon\mathfrak{J}_{n}(123,213)\rightarrow\mathfrak{J}_{n-1}(123,213)\sqcup\mathfrak{J}_{n-2}(123,213)$
by 
\[
\phi(\pi)=\begin{cases}
\pi_{2}\pi_{3}\cdots\pi_{n}, & \text{if }\pi_{1}=n,\\
\pi_{1}\pi_{4}\pi_{5}\cdots\pi_{n}, & \text{otherwise.}
\end{cases}
\]
For example, $\phi(7465132)=465132$ and $\phi(5762431)=52431$. 

We claim that $\phi$ is a well-defined bijection. It is easy to see
that if $\pi\in\mathfrak{J}_{n}(123,213)$ and $\pi_{1}=n$, then
$\pi_{2}\pi_{3}\cdots\pi_{n}\in\mathfrak{J}_{n-1}(123,213)$. Now,
suppose that $\pi\in\mathfrak{J}_{n}(123,213)$ and $\pi_{1}\neq n$.
Recall from Lemma \ref{l-123-213-ini} that $\pi_{2}=n$ and $\pi_{3}=n-1$,
so $\phi(\pi)$ is indeed a permutation of the set $[n-2]$. Moreover,
we have $\rho_{\phi(\pi)}(\pi_{k})=\rho_{\pi}(\pi_{k})$ for all $4\leq k\leq n$,
as well as $\rho_{\pi}(\pi_{1})=n(n-1)\rho_{\phi(\pi)}(\pi_{1})$;
and because each of these $\rho_{\pi}(\pi_{k})$ has even length,
the $\rho_{\phi(\pi)}(\pi_{k})$ all have even length as well. Therefore,
$\phi(\pi)$ is Jacobi (and clearly avoids $123$ and $213$). 

The inverse of $\phi$ is defined by
\[
\phi^{-1}(\pi)=\begin{cases}
n\pi, & \text{if }\pi\in\mathfrak{J}_{n-1}(123,213),\\
\pi_{1}n(n-1)\pi_{2}\pi_{3}\cdots\pi_{n}, & \text{if }\pi\in\mathfrak{J}_{n-2}(123,213);
\end{cases}
\]
verifying that $\phi^{-1}(\pi)\in\mathfrak{J}_{n}(123,213)$ can be
done similarly.

Since $\phi$ is a bijection, it follows that $j_{n}(123,213)$ satisfies
the Fibonacci recurrence, and because $j_{1}(123,213)=j_{2}(123,213)=1$,
we conclude that $j_{n}(123,213)=f_{n-1}$ for all $n\geq1$.
\end{proof}
The next two Jacobi avoidance classes are counted by the quarter-squares
plus 1.
\begin{thm}
\label{t-123-231or312}For all $n\geq0$ and $\sigma\in\{231,312\}$,
we have ${\displaystyle j_{n}(123,\sigma)=1+\left\lfloor \frac{(n-1)^{2}}{4}\right\rfloor }$.
\end{thm}

For the proof of this theorem, it will be helpful to have the notion
of an \textit{interval} of a permutation $\pi$, which is a consecutive
subword of $\pi$ consisting of consecutive integer letters. For example,
$3412$ is an interval of $\pi=5734126$.
\begin{proof}
First, we prove the result for $\sigma=231$. The formula clearly
holds for $n=0$, so assume that $n\geq1$. Notice that any nonempty
$\pi\in\mathfrak{J}_{n}(123,231)$ is of the form $\pi=\alpha1\beta$
where $\beta$ is a decreasing interval of even length and $\alpha$
is decreasing. The choice of $\beta$ completely determines $\alpha$,
as $\alpha$ must consist of the remaining letters arranged in decreasing
order; thus, it suffices to count the choices for $\beta$.

The length of $\beta$ can be any nonnegative even integer $2k$ up
to $n-1$, and assuming that $\left|\beta\right|=2k>0$, there are
$n-2k$ choices for the first letter $\beta_{1}$ of $\beta$. Because
$\beta$ is a decreasing interval, the choice of $\beta_{1}$ completely
determines $\beta$, so there are $1+\sum_{k=1}^{\left\lfloor (n-1)/2\right\rfloor }(n-2k)$
choices for $\beta$. A routine argument shows that 
\[
\sum_{k=1}^{\left\lfloor (n-1)/2\right\rfloor }(n-2k)=\left\lfloor \frac{(n-1)^{2}}{4}\right\rfloor ;
\]
this completes the proof for $\sigma=231$.

The result for $\sigma=312$ follows from Lemma \ref{l-pavinv} and
Theorem \ref{t-312-231-inv}.
\end{proof}
Lastly, we consider three Jacobi avoidance classes which are enumerated
by (repeated) powers of $2$.
\begin{thm}
\label{t-213-231or312}For all $n\geq1$ and $\sigma\in\{231,312\}$,
we have $j_{n}(213,\sigma)=2^{\left\lfloor (n-1)/2\right\rfloor }$.
\end{thm}

\begin{proof}
We take $\sigma=312$; the result for $\sigma=231$ will follow from
Lemma \ref{l-pavinv} and Theorem \ref{t-312-231-inv}.

Let $\pi\in\mathfrak{J}_{n}(213,312)$ be nonempty. Then $1$ is either
the first letter of $\pi$ or the last letter of $\pi$, so either
$\pi=\alpha\ominus1$ or $\pi=1\oplus\beta$, where $\alpha$ and
$\beta$ are Jacobi permutations that avoid $213$ and $312$, and
$\beta$ is of even length. When $n\geq3$ is odd, this gives us a
$1$-to-$2$ map from $\mathfrak{J}_{n-1}(213,312)$ to $\mathfrak{J}_{n}(213,312)$.
On the other hand, we cannot have $\pi=1\oplus\beta$ when $n\geq2$
is even, so in this case $\mathfrak{J}_{n-1}(213,312)$ and $\mathfrak{J}_{n}(213,312)$
are in bijection. In summary, we have 
\[
j_{n}(213,312)=\begin{cases}
2j_{n-1}(213,312), & \text{if }n\geq3\text{ is odd,}\\
j_{n-1}(213,312), & \text{if }n\geq2\text{ is even,}
\end{cases}
\]
which along with $j_{1}(213,312)=1$ implies $j_{n}(213,312)=2^{\left\lfloor (n-1)/2\right\rfloor }$
for all $n\geq1$.
\end{proof}
\begin{thm}
\label{t-231-312}For all $n\geq1$, we have $j_{n}(231,312)=2^{\left\lfloor (n-1)/2\right\rfloor }$.
\end{thm}

\begin{proof}
Let $\pi\in\mathfrak{J}_{n}(231,312)$ be nonempty. Then we can write
$\pi=\alpha1\beta$ where $\alpha$ is of the form $k(k-1)\cdots2$
and $\beta$ is a Jacobi permutation of even length on the letters
$k+1,k+2,\dots,n$ that avoids $231$ and $312$. Taking $J_{\e}\coloneqq\sum_{n=0}^{\infty}j_{2n}(231,312)x^{2n}$,
we thus have the equation 
\[
J_{\e}=1+\frac{x^{2}}{1-x^{2}}J_{\e},
\]
which leads to
\[
J_{\e}=1+\frac{x^{2}}{1-2x^{2}}=1+\sum_{n=1}^{\infty}2^{n-1}x^{2n}.
\]
Furthermore, if $J_{\o}\coloneqq\sum_{n=1}^{\infty}j_{2n-1}(231,312)x^{2n-1}$,
then 
\[
J_{\o}=\frac{x}{1-x^{2}}J_{\e}=\frac{x}{1-2x^{2}}=\sum_{n=1}^{\infty}2^{n-1}x^{2n-1}.
\]
Adding these expressions for $J_{\e}$ and $J_{\o}$ gives us
\[
\sum_{n=0}^{\infty}j_{n}(231,312)x^{n}=1+\sum_{n=1}^{\infty}2^{n-1}x^{2n}+\sum_{n=1}^{\infty}2^{n-1}x^{2n-1}=1+\sum_{n=1}^{\infty}2^{\left\lfloor (n-1)/2\right\rfloor }x^{n},
\]
and comparing coefficients completes the proof.
\end{proof}

\subsection{Triple and quadruple restrictions}

We end this section by stating several results for $\mathfrak{J}_{n}(\Pi)$
where $\left|\Pi\right|=3$ or $\left|\Pi\right|=4$. Together with
the results presented earlier, these complete the enumeration of $\mathfrak{J}_{n}(\Pi)$
over all subsets $\Pi\subseteq\mathfrak{S}_{3}$.
\begin{thm}
\label{t-123-213-231or312}For all $n\geq1$ and $\sigma\in\{231,312\}$,
we have $j_{n}(123,213,\sigma)=\left\lceil n/2\right\rceil $. 
\end{thm}

\begin{proof}
We prove the result for $\sigma=312$. As before, the result for $\sigma=231$
will then follow from Lemma~\ref{l-pavinv} and Theorem \ref{t-312-231-inv}.

Let $\pi\in\mathfrak{J}_{n}(123,213,312)$ with $n\ge2$. Then $1$
is either the first letter of $\pi$ or the last letter of $\pi$,
so either $\pi=\alpha\ominus1$ where $\alpha\in\mathfrak{J}_{n-1}(123,213,312)$,
or $\pi=1n(n-1)\cdots2$; the latter is only possible when $n$ is
odd. Thus, we have 
\[
j_{n}(123,213,312)=\begin{cases}
j_{n-1}(123,213,312), & \text{if }n\text{ is even,}\\
j_{n-1}(123,213,312)+1, & \text{if }n\text{ is odd.}
\end{cases}
\]
Because $j_{1}(123,213,312)=1$, the result follows.
\end{proof}
\begin{thm}
\label{t-123-231-312}For all $n\geq1$, we have $j_{n}(123,231,312)=\left\lceil n/2\right\rceil $.
\end{thm}

\begin{proof}
Let $\pi\in\mathfrak{J}_{n}(123,231,312)$ be nonempty. Recall from
the proof of Theorem \ref{t-123-231or312} that $\pi=\alpha1\beta$
where $\beta$ is a decreasing interval of even length and $\alpha$
is decreasing, but since $\pi$ also avoids $312$ here, $\beta$
must either be empty or begin with $n$. Adapting that proof accordingly,
we have $1+\left\lfloor (n-1)/2\right\rfloor =\left\lceil n/2\right\rceil $
choices for $\beta$, which again completely determines $\pi$.
\end{proof}
\begin{thm}
\label{t-213-231-312}For all $n\geq2$, we have 
\[
j_{n}(213,231,312)=j_{n}(123,213,231,312)=\begin{cases}
1, & \text{if }n\text{ is even,}\\
2, & \text{if }n\text{ is odd.}
\end{cases}
\]
\end{thm}

\begin{proof}
The proof is the same as that of Theorem \ref{t-213-231or312}, except
that when $n$ is even, $\pi$ is forced to be $\pi=n\cdots21$ as
we are now requiring $\pi$ to avoid $231$. Hence, there is only
one Jacobi permutation of length $n$ avoiding the prescribed patterns
for each even $n\geq2$, and twice as many for each subsequent odd
$n$.
\end{proof}

\section{\label{s-conclusion}Future directions for research}

To summarize, our work provides the first in-depth study of Jacobi
permutations since their introduction by Viennot. We achieved a complete
enumeration of the Jacobi avoidance classes $\mathfrak{J}_{n}(\Pi)$
where $\Pi$ ranges over all subsets of $\mathfrak{S}_{3}$, and we
have also studied the distribution of several statistics\textemdash the
number of ascents/descents, the number of left-to-right minima, and
the last letter\textemdash over the single-pattern avoidance classes
and as well as the full set $\mathfrak{J}_{n}$. 

Among those that we considered, the only distribution that eluded
us was that of $\last$ over $\mathfrak{J}_{n}(123)$, so we leave
this as an open problem.
\begin{problem}
Determine the distribution of $\last$ over $\mathfrak{J}_{n}(123)$.
\end{problem}

Along the way, we discovered that all permutations in $\mathfrak{J}_{n}(213)$,
$\mathfrak{J}_{n}(231)$, and $\mathfrak{J}_{n}(231)$ are doubly
Jacobi, which led us to investigate doubly Jacobi permutations in
the other single-pattern avoidance classes. However, we did not enumerate
doubly Jacobi permutations in $\mathfrak{J}_{n}$ (without pattern
avoidance restrictions).
\begin{problem}
Enumerate doubly Jacobi permutations in $\mathfrak{J}_{n}$.
\end{problem}

Doubly alternating permutations, which motivated our notion of doubly
Jacobi permutations, were also first studied with pattern avoidance
restrictions \cite{Guibert2000,Ouchterlony2005}. The enumeration
of all doubly alternating permutations in $\mathfrak{S}_{n}$ was
obtained later by Stanley \cite{Stanley2007} using symmetric function
theory. Stanley's proof relies on the fact that whether a permutation
is alternating is completely determined by its descent set. Since
the Jacobi property is not determined by the descent set, Stanley's
approach is inapplicable to our setting.

Given a permutation $\pi$ of length at least 2, define $\cpen(\pi)$
to be the penultimate (second-to-last) letter of the complement of
$\pi$. For example, the complement of $\pi=263154$ is $\pi^{c}=514623$,
so $\cpen(\pi)=2$. Recall from Theorem \ref{t-last} that the distribution
of $\last$ over $\mathfrak{J}_{n}$ is given by the Entringer numbers.
We conjecture that the same is true for $\cpen$.
\begin{conjecture}
\label{cj-cpen}For all $n\geq2$ and $k\geq1$, we have $j_{n,k}^{\cpen}=E_{n,k}$.
\end{conjecture}

Furthermore, empirical evidence suggests the following conjecture
about the joint distribution of $\cpen$ and $\last$ over $\mathfrak{J}_{n}$.
Let 
\[
j_{n,i,k}^{(\cpen,\last)}\coloneqq\left|\{\,\pi\in\mathfrak{J}_{n}:\cpen(\pi)=i\text{ and }\last(\pi)=k\,\}\right|.
\]

\begin{conjecture}
\label{cj-cpen-last}For all $n\geq3$ and $i,k\geq1$, we have $j_{n,i,k}^{(\cpen,\last)}=E_{n-1,n+1-i-k}$.
\end{conjecture}

In particular, Conjecture \ref{cj-cpen-last} implies that $\cpen$
and $\last$ have a symmetric joint distribution over $\mathfrak{J}_{n}$.

Recall from Corollary \ref{c-jacandreasc} and Theorem \ref{t-last}
that the distribution of $\asc$ over $\mathfrak{J}_{n}$ is equal
to that of $\des$ over the set $\mathfrak{A}_{n}$ of Andr\'{e}
permutations, and also that the distribution of $\last$ over $\mathfrak{J}_{n}$
is equal to the distribution of $\first$ over $\mathfrak{A}_{n}$.
Our next conjecture asserts that the joint distribution of $(\asc,\last)$
over $\mathfrak{J}_{n}$ is equal to that of $(\des,\first)$ over
$\mathfrak{A}_{n}$. Let 
\begin{align*}
j_{n,i,k}^{(\asc,\last)} & \coloneqq\left|\{\,\pi\in\mathfrak{J}_{n}:\asc(\pi)=i\text{ and }\last(\pi)=k\,\}\right|\qquad\text{and}\\
a_{n,i,k}^{(\des,\first)} & \coloneqq\left|\{\,\pi\in\mathfrak{A}_{n}:\des(\pi)=i\text{ and }\first(\pi)=k\,\}\right|.
\end{align*}

\begin{conjecture}
\label{cj-asc-last}For all $n\geq1$, $i\geq0$, and $k\geq1$, we
have $j_{n,i,k}^{(\asc,\last)}=a_{n,i,k}^{(\des,\first)}$.
\end{conjecture}

One might call this conjecturally-shared joint distribution the \textit{Andr\'{e}\textendash Entringer}
\textit{distribution}, since it simultaneously refines the distributions
given by the Andr\'{e} polynomials and the Entringer numbers. To
the best of our knowledge, the Andr\'{e}\textendash Entringer distribution
has not appeared in the literature, which leads us to pose the following.
\begin{problem}
Find a formula for computing the Andr\'{e}\textendash Entringer distribution.
\end{problem}

Finally, motivated by the various notions of ``$r$-alternating permutation''
in the literature,\footnote{See, for example, Mendes\textendash Remmel \cite[p.\ 126]{Mendes2015}
and Stanley \cite[p.\ 17]{Stanley2010}.} we propose a notion of ``$r$-Jacobi permutation''. Given a positive
integer $r$, we say that a permutation $\pi\in\mathfrak{S}_{S}$
is \textit{$r$-Jacobi} if $\left|\rho_{\pi}(x)\right|$ is divisible
by $r$ for each $x\in S$. Notice that all permutations are $1$-Jacobi,
and 2-Jacobi permutations are precisely the Jacobi permutations.

Let $j_{n}^{(r)}$ denote the number of $r$-Jacobi permutations in
$\mathfrak{S}_{n}$; see Table \ref{tb-kJacobi}. The $j_{n}^{(r)}$
can be thought of as generalizing the Euler numbers. We believe that
$r$-Jacobi permutations and the generalized Euler numbers $j_{n}^{(r)}$
are worthy of study.
\begin{table}[!h]
\begin{centering}
\renewcommand{\arraystretch}{1.1}%
\begin{tabular}{|>{\centering}p{34bp}|>{\centering}p{34bp}|>{\centering}p{34bp}|>{\centering}p{34bp}|>{\centering}p{34bp}|>{\centering}p{34bp}|>{\centering}p{34bp}|>{\centering}p{34bp}|>{\centering}p{34bp}|>{\centering}p{34bp}|}
\hline 
$r\backslash n$ & $1$ & $2$ & $3$ & $4$ & $5$ & $6$ & $7$ & $8$ & $9$\tabularnewline
\hline 
$1$ & $1$ & $2$ & $6$ & $24$ & $120$ & $720$ & $5040$ & $40320$ & $362880$\tabularnewline
\hline 
$2$ & $1$ & $1$ & $2$ & $5$ & $16$ & $61$ & $272$ & $1385$ & $7936$\tabularnewline
\hline 
$3$ & $1$ & $1$ & $1$ & $2$ & $6$ & $16$ & $52$ & $234$ & $1018$\tabularnewline
\hline 
$4$ & $1$ & $1$ & $1$ & $1$ & $2$ & $7$ & $22$ & $57$ & $184$\tabularnewline
\hline 
$5$ & $1$ & $1$ & $1$ & $1$ & $1$ & $2$ & $8$ & $29$ & $85$\tabularnewline
\hline 
\end{tabular}
\par\end{centering}
\caption{\label{tb-kJacobi}The numbers $j_{n}^{(r)}$ up to $r=5$ and $n=9$.}
\end{table}

\vspace{10bp}

\noindent \textbf{Acknowledgments.} We thank Bernard Deconinck for
helpful e-mail correspondence. The authors were partially supported
by NSF grant DMS-2316181.

\bibliographystyle{plain}
\addcontentsline{toc}{section}{\refname}\bibliography{bibliography}

\end{document}